\documentclass[12pt,reqno]{amsart}
\usepackage[margin=1in]{geometry}
\usepackage{amsmath,amssymb,amsthm,graphicx,amsxtra, setspace}
\usepackage[utf8]{inputenc}
\usepackage{mathrsfs}
\usepackage{hyperref}
\usepackage{upgreek}
\usepackage{mathtools}
\usepackage[dvipsnames]{xcolor}
\usepackage[mathcal]{euscript}
\usepackage{amsmath,units}
\usepackage{pgfplots}
\usepackage{tikz}
\usepackage{verbatim}
\allowdisplaybreaks

\usepackage{fontenc}
\usepackage{textcomp}
\usepackage{marvosym}
\usepackage{eurosym}
\usepackage{dsfont}

\usepackage[pagewise]{lineno}

\DeclareMathAlphabet{\mathpzc}{OT1}{pzc}{m}{it}

\usepackage[cyr]{aeguill}

\colorlet{darkblue}{blue!50!black}

\hypersetup{
	colorlinks,%
	citecolor=blue,%
	filecolor=red,%
	linkcolor=red,%
	urlcolor=blue,%
	pdfnewwindow=true,%
	pdfstartview={FitH}
}

\newtheorem{theorem}{Theorem}[section]
\newtheorem{lemma}[theorem]{Lemma}
\newtheorem{proposition}[theorem]{Proposition}

\newtheorem{corollary}[theorem]{Corollary}
\newtheorem{definition}[theorem]{Definition}
\newtheorem{example}[theorem]{Example}
\newtheorem{remark}[theorem]{Remark}

\allowdisplaybreaks

\let\originalleft\left
\let\originalright\right
\renewcommand{\left}{\mathopen{}\mathclose\bgroup\originalleft}
\renewcommand{\right}{\aftergroup\egroup\originalright}

\newcommand{\Tr}{\mathop{\mathrm{Tr}}}
\renewcommand{\d}{\/\mathrm{d}\/}

\def\w{\textbf{W}^{\varepsilon}_{{\theta}^{\varepsilon}}}

\def\e{\varepsilon}

\def\L{\mathbb{L}}
\def\sgn{\mathrm{sgn}}
\def\A{\mathrm{A}}
\def\I{\mathrm{I}}

\def\C{\mathrm{C}}
\def\f{\boldsymbol{f}}
\def\J{\mathrm{J}}
\def\B{\mathrm{B}}
\def\D{\mathrm{D}}
\def\y{\boldsymbol{y}}

\def\Y{\mathrm{Y}}
\def\Z{\mathrm{Z}}
\def\E{\mathbb{E}}
\def\X{\mathbf{X}}
\def\Y{\mathbf{Y}}
\def\Z{\mathbf{Z}}
\def\U{\mathbb{U}}
\def\x{\boldsymbol{x}}

\def\g{\boldsymbol{g}}

\def\z{\boldsymbol{z}}
\def\v{\boldsymbol{v}}
\def\V{\mathbb{v}}
\def\w{\boldsymbol{w}}
\def\W{\mathrm{W}}

\def\Q{\mathrm{Q}}

\def\N{\mathbb{N}}

\def\no{\nonumber}
\def\V{\mathbb{V}}
\def\wi{\widetilde}
\def\Q{\mathrm{Q}}
\def\u{\mathrm{U}}
\def\P{\mathbb{P}}
\def\u{\boldsymbol{u}}
\def\H{\mathbb{H}}

\def\0{\boldsymbol{0}}

\newcommand{\eps}{\varepsilon}

\newcommand{\R}{\mathbb{R}}

\renewcommand{\d}{\/\mathrm{d}\/}

\newcommand{\Addresses}{{% additional braces for segregating \footnotesize
		\footnote{
			\noindent \textsuperscript{1,2,3}Department of Mathematics, Indian Institute of Technology Roorkee-IIT Roorkee,
			Haridwar Highway, Roorkee, Uttarakhand 247667, INDIA.\par\nopagebreak
			\noindent  \textit{e-mail:} \texttt{Manil T. Mohan: maniltmohan@ma.iitr.ac.in, maniltmohan@gmail.com.}
			
			\textit{e-mail:} \texttt{Sagar Gautam: sagar\_g@ma.iitr.ac.in.}
			
			\noindent \textsuperscript{*}Corresponding author.
			
			\textit{Key words:} extended Convective Brinkman-Forchheimer equations, monotone operators, approximate controllability, transition semigroup,  irreducibility.
			
			Mathematics Subject Classification (2020): Primary 93B05, 35R15; Secondary 60H30, 37L40, 76D03.
			
			}}}

\begin{document}
	%	\linenumbers
	
	\title[Irreducibility of the transition semigroup for CBFeD  equations]{Approximate controllability and Irreducibility of the transition semigroup associated with Convective Brinkman-Forchheimer extended Darcy Equations
		\Addresses}
	\author[S. Gautam and M. T. Mohan]
	{Sagar Gautam\textsuperscript{1} and Manil T. Mohan\textsuperscript{2*}}
	
	\maketitle
	
		\begin{abstract}
		In this article, the following controlled convective Brinkman-Forchheimer extended Darcy (CBFeD) system is considered in a $d$-dimensional torus $\mathbb{T}^d$:
		\begin{align*}
			\frac{\partial\boldsymbol{y}}{\partial t}-\mu \Delta\boldsymbol{y}+(\boldsymbol{y}\cdot\nabla)\boldsymbol{y}+\alpha\boldsymbol{y}+\beta\vert \boldsymbol{y}\vert^{r-1}\boldsymbol{y}+\gamma\vert \boldsymbol{y}\vert ^{q-1}\boldsymbol{y}+\nabla p=\boldsymbol{g}+\boldsymbol{u},\ \nabla\cdot\boldsymbol{y}=0, 
		\end{align*}
		where $d\in\{2,3\}$, $\mu,\alpha,\beta>0$, $\gamma\in\mathbb{R}$, $r,q\in[1,\infty)$ with $r>q\geq 1$ and $\boldsymbol{u}$ is the control. For the super critical ($r>3$) and critical  ($r=3$ with $2\beta\mu>1$) cases, we  first show the approximate controllability of the above system in the usual energy space (divergence-free $\mathbb{L}^2(\mathbb{T}^d)$ space). As an application of the approximate controllability result, we establish the  irreducibility of the transition semigroup associated with stochastic CBFeD system perturbed by non-degenerate Gaussian noise  in the usual energy space by exploiting the regularity of solutions, smooth approximation of the multi-valued map $\mathrm{sgn}(\cdot)$ a density argument and monotonicity properties of the linear and nonlinear operators. 
	\end{abstract}

\section{Introduction}\label{thm}\setcounter{equation}{0}
Mathematical modeling and analysis for the study of fluids are topics of great interest, both for our understanding of the phenomena related to fluids and for applications. The investigation of heat and fluid movement in porous media is of significant importance  to numerous scientific and engineering domains. Many models of porous media, especially those which describe the flow of fluids, are based on \emph{Darcy's law}. It is an empirical law which describes a linear relationship between flow rate and pressure gradient ($\nabla p$), that  is, $$\y_d=-\frac{k}{\nu}\nabla p,$$  where $\y_d$ is the Darcy velocity,  $k$ is the permeability of the fluid, $\nu>0$ is the dynamic viscosity and $p$ is the pressure in the fluid flow (cf. \cite{MTT}). Many investigations demonstrate that as flow rate increases, the relationship changes from linear to nonlinear (see, for instance, oil reservoirs and petroleum development, as well as certain phenomena like radial flow patterns).  In order to describe the high flow rate in porous media, Forchheimer introduced a quadratic term involving velocity into the Darcy law. This correction yields the \emph{Darcy-Forchheimer law}, which takes the form $$\nabla p=-\frac{\nu}{k}\y_f-\gamma\rho_f|\y_f|\y_f,$$ where $\y_f$ is the Forchheimer velocity, $\gamma>0$ is the Forchheimer coefficient and $\rho_f$ is the density (cf. \cite{MTT}). Based upon the \emph{Darcy-Forchheimer law}, the following classical configuration of the convective Brinkman-Forchheimer extended Darcy model was originally derived in the framework of thermal dispersion in a porous medium using the method of volume averaging of the velocity and temperature deviations in the pores:
\begin{equation*}
	\left\{
	\begin{aligned}
		\frac{\partial \y_f}{\partial t}- \mu\Delta\y_f+(\y_f\cdot\nabla)\y_f+a_0\y_f+a_1|\y_f|\y_f+\nabla p&=\g, \\ \nabla\cdot\y&=0.
	\end{aligned}
	\right.
\end{equation*}	
However,  the quadratic nonlinearity in the Forchheimer equation can be further generalized to incorporate some additional nonlinear terms from the mathematical point of view. In fact, at higher flow rates through porous media, the most practical case is described by a \emph{linear and cubic Darcy-Forchheimer law} (see \cite{MTT}), $$\nabla p=-\frac{\nu}{k}\y_f-\gamma\rho_f|\y_f|^2\y_f.$$ On taking into account these nonlinear corrections of Darcy's law, the author in \cite{asch} discussed the following convective Brinkman-Forchheimer extended Dracy model:
\begin{equation*}
	\left\{
	\begin{aligned}
		\frac{\partial \y_f}{\partial t}-\mu\Delta\y_f+ (\y_f\cdot\nabla)\y_f+a_0\y_f+a_1|\y_f|\y_f+a_2|\y_f|^2\y_f+\nabla p&=\g, \\ \nabla\cdot\y&=0.
	\end{aligned}
	\right.
\end{equation*}	
In this work, we consider a further generalization of the convective Brinkman-Forchheimer extended Dracy equations by taking into account a \emph{pumping term} which also have a similar nonlinearity but with the opposite sign. 

\subsection{The Model}
Let us consider a $d$-dimensional torus $\mathbb{T}^d=\left(\frac{\R}{\mathrm{L}\mathbb{Z}}\right)^d$ where $d\in\{2,3\}$ and $L>0$. The convective Brinkman-Forchheimer extended Darcy (CBFeD)  equations describe the motion of  incompressible fluid flows in a saturated porous medium and are  given by
\begin{equation}\label{1}
	\left\{
	\begin{aligned}
		\frac{\partial \y}{\partial t}-\mu\Delta\y+\alpha\y+(\y\cdot\nabla)\y+\beta|\y|^{r-1}\y +\gamma|\y|^{q-1}\y+\nabla p&=\g, \ \text{ in } \ \mathbb{T}^d\times(0,\infty), \\ \nabla\cdot\y&=0, \ \text{ in } \ \mathbb{T}^d\times(0,\infty), \\
		\y(0)&=\y_0 \ \text{ in } \ \mathbb{T}^d,
	\end{aligned}
	\right.
\end{equation}	
where $\y(x,t):\mathbb{T}^d\times(0,\infty)\to\R^d$ represents the velocity field at time $t$ and position $x$, $p(x,t):\mathbb{T}^d\times(0,\infty)\to\R$ denotes the pressure field, $\g(x,t):\mathbb{T}^d\times(0,\infty)\to\R^d$ is an external forcing. Moreover, $\y(\cdot,\cdot)$, $p(\cdot,\cdot)$ and $\g(\cdot,\cdot)$ satisfy the following periodic conditions:
\begin{align}\label{2}
	\y(x+\mathrm{L}e_{i},\cdot) = \y(x,\cdot), \ p(x+\mathrm{L}e_{i},\cdot) = p(x,\cdot) \ \text{ and } \ \g(x+\mathrm{L}e_{i},\cdot) = \g(x,\cdot),
\end{align}
for every $x\in\R^{d}$ and $i=1,\ldots,d,$ where $\{e_{1},\dots,e_{d}\}$ is the canonical basis of $\R^{d}.$ The constant $\mu>0$ denotes the \emph{Brinkman coefficient} (effective viscosity) and the constants $\alpha$ and $\beta$ are due to Darcy-Forchheimer law which are termed as \emph{Darcy} (permeability of porous medium) and \emph{Forchheimer} (proportional to the porosity of the material) coefficients, respectively. The nonlinear term $\gamma|\y|^{q-1}\y$ appearing in \eqref{1} functions as \emph{damping} for $\gamma>0$ and \emph{pumping} for $\gamma<0$.    The absorption exponent $r\in[1,\infty)$ and  $r=3$ is known as the \emph{critical exponent}. For $\gamma=0$, one obtains the convective Brinkman-Forchheimer (CBF) equations (\cite{KWH}).  The critical homogeneous CBF equations (\eqref{1} with  $r=3$ and $\g=\mathbf{0}$)  have the same scaling as Navier-Stokes equations (NSE) only when $\alpha=0$ (\cite{KWH}).  We refer the case $r<3$ as \emph{subcritical} and $r>3$ as \emph{supercritical} (or fast growing nonlinearities). The model is accurate when the flow velocity is too large for Darcy's law to be valid, and apart from that  the porosity is not too small (\cite{MTT}).  If one considers \eqref{1} with $\alpha=\beta=\gamma=0$, then we obtain the \emph{classical NSE}, and if $\alpha, \beta,\gamma>0$, then it can be considered as \emph{damped NSE}. 

\subsection{Literature review} We present a summary of the literature-based research on  the model \eqref{1}, approximation controllability and  irreducibility. 
\subsubsection{The deterministic problem.} The problem \eqref{1} has important applications not only from the physical but also from mathematics perspective.  Due to Leary and Hopf (cf. \cite{Hopf,Leray}), the existence of at least one weak solution satisfying the energy inequality of 3D NSE is known. But the uniqueness is still a challenging open problem for the mathematical community. Therefore,  several mathematicians came up with various modifications of 3D NSE (cf. \cite{ZCQJ,KT2,ZZXW}, etc.). The authors in \cite{SNA1,SNA} introduced NSE modified by an absorption term $\beta|\y|^{r-1}\y$ with $r\geq 1$ and proved the existence of weak solution for any $d\geq2$ and uniqueness for $d=2$. In the literature, modifications of the classical NSE by the damping term $\alpha\y+\beta|\y|^{r-1}\y$ is known as the \emph{convective Brinkman-Forchheimer equations} (cf. \cite{KWH}, etc.). In \cite{MTT}, the authors considered the classical 3D NSE with damping $\beta|\y|^{r-1}\y$ as well as pumping $\gamma|\y|^{q-1}\y$ (without the linear damping $\alpha\y$) and established the existence of weak solutions for $r>q$ and uniqueness for $r>3$. Moreover, they demonstrated the global existence of a unique strong solution for $r>3$ with the initial data in $\mathbb{H}^1$. Similar to 3D NSE, the existence of a unique global (in time) weak solution of 3D CBFeD with $r\in[1,3)$ (for any $\beta,\mu>0$) and $r=3$ (for $2\beta\mu<1$) is also an open problem. 

\subsubsection{The stochastic problem.} 
Let us now discuss the stochastic counterpart of  the system \eqref{1} and related models. The existence of a pathwise unique strong solution for the stochastic tamed NSE (driven by Gaussian), in the whole space as well as in the periodic boundary case, was established in \cite{MRXZ1}. They also proved the existence of a  unique  invariant measure for the corresponding transition semigroup.  Recently, the authors in \cite{ZBGD} improved their results  for a slightly simplified system. In \cite{WL}, it is demonstrated that strong solutions exist and are unique for a wide class of SPDE, where the coefficients satisfy Lyapunov and local monotonicity conditions. The author used the stochastic tamed 3D Navier-Stokes equations as an example.  All the above  works established the existence and uniqueness of strong solutions in the following regularity class (see Subsection \ref{sub2.1}):
$$\C([0,T];\V)\cap\mathrm{L}^2(0,T;\D(\A)),\ \P\text{-a.s.},$$
by assuming $\mu=1$. It appears to us that the model covered in these papers ia a special case of our model with $r=3$, $\mu=\beta=1$ and $\alpha=\gamma=0$, so that the condition $2\beta\mu>1$ is automatically holds (cf. the estimates \eqref{vreg9}, \eqref{vreg10} and \cite[Lemma 2.3]{MRXZ1}).

 Let us now look into the results available in the literature for stochastic CBF equations and related models in bounded domains and torus. By exploiting a monotonicity
 property of the linear and nonlinear operators as well as a stochastic generalization of the
 Minty-Browder technique, the author in \cite{MT4,MT2} established the existence and uniqueness of
 a global strong solution $$\y\in\C([0,T];\H)\cap\mathrm{L}^2(0,T;\V)\cap\mathrm{L}^{r+1}(0,T;\wi\L^{r+1}),\ \P\text{-a.s.},$$ for $\y_0\in\H$,  satisfying the energy equality (It\^o's formula) for stochastic CBF equations (in bounded and torus) driven by multiplicative Gaussian and pure jump noises, respectively, for $d=2,3$ and $r\in[3,\infty)$ $(2\beta\mu\geq1)$ for $d=r=3$. Under suitable assumptions on the initial
data ($\y_0\in\V$) and noise coefficients, the author has also showed the following regularity result: 
\begin{align}\label{13}
	\y\in\C([0,T];\V)\cap\mathrm{L}^2(0,T;\D(\A))\cap\mathrm{L}^{r+1}(0,T;\wi{\L}^{p(r+1)}),\ \P\text{-a.s.},
	\end{align}
where $p\in[2,\infty)$ for $d=2$ and $p=3$ for $d=3$. The author in \cite{MT10} established, using the contraction mapping principle, the existence and uniqueness of local and global pathwise mild solutions for stochastic CBF equations perturbed by an additive L\'evy noise in $\R^d$, for $d=2,3$. For 2D and 3D  stochastic CBF equations perturbed by a multiplicative L\'evy noise, the existence of a weak martingale solution is proved in \cite{MT6}. This is achieved by using a version of the Skorokhod embedding theorem for nonmetric   spaces (for $d=2,3$ and $r\in[1,\infty)$), a compactness method, and the classical Faedo-Galerkin   approximation. Recently the author in \cite{MT8} established the   asymptotic irreducibility, asymptotic strong Feller property and ergodicity  by an application of  the asymptotic log harmonic inequality. Furthermore, for the CBFeD equations \eqref{1}, the author in   \cite{MT7}  showed that the solutions 2D and 3D stochastic CBFeD equations driven by Brownian motion can be approximated by 2D and 3D stochastic CBFeD equations forced by pure jump noise/random kicks on the state space $\D([0,T];\H),$ where $\D([0,T];\H)$ is the space of c\`adl\`ag paths from $[0,T]$ to $\H$. 

\subsubsection{The control problem.}
%\textcolor{red}{Within the mathematical scientific community, it is widely acknowledged that one of the most significant mathematical issues is the rigorous understanding of turbulence and related hydrodynamics questions. The \emph{control of turbulence in a fluid flow} is another related problem, albeit from a more practical standpoint (see \cite{cft,fmrt,RT1} for a comprehensive literature on the control of turbulence). In the latter case, we take into account a flow in a specific physical domain that has a known initial configuration. We then address the problem of figuring out the optimal  course of action to take on the system in order to minimize turbulence within the flow, using a variety of devices such as body forces, boundary values, temperature, etc. \cite{FART,SSS}.}  
Controllability theory for evolution partial differential equations first emerged in the 1960s. Egorov \cite{jve}, Russell \cite{dr1}, and Fattorini \cite{hf} established the basis for this theory. Specifically, these papers developed the moment method, which reduces the exact controllability problem's solution to problems in the theory of exponential series, as well as introduced the duality principle, which reduces the controllability problem for an evolution equation to the adjoint equation's observability problem. In the late 1980's, Lions presented the subject of fluid flow controllability (see \cite{JL}), which refers to how the NSE system can be brought to a desired plausible state, for instance, vanishing velocity, by controlling the flow on a portion of the boundary. He posed the question of whether, for any finite energy initial data, the solution of  NSE system starts at the initial data and then reaches rest. The question was fairly impressive as the answer was not even known in the case of the heat equation. The authors in \cite{avf1,gLLr} demonstrated the first breakthroughs in the heat equation utilising the Carleman estimate associated with elliptic and parabolic operators. 
%The controllability of 2D-Euler's equations has been proven in \cite{jmc8} in suitable spaces. Later, this result was extended to 3D Euler's equations in \cite{og1,og2}.

For many years, researchers have been studying the approximate controllability of NSE with homogeneous Dirichlet boundary conditions on a bounded domain using controls acting on an arbitrarily chosen subdomain or subboundary. The approximate controllability (global) for 2D-NSE with Navier-slip boundary conditions was discussed in \cite{jmc5} by using the so-called \emph{return method}, while the small time exact controllability for 2D-NSE has been proved in \cite{jmc6}, where the domain is a manifold without boundary. Likewise, the authors in \cite{avfoyu}, proved small-time global exact null controllability when the control is supported on the whole boundary (see also \cite{oyu,oyu1}). Furthermore, the authors in \cite{JLEZ} demonstrated that in 2D and 3D-NSE exact controllability could be achieved in a bounded domain employing finite-dimesional Galerkin approximations. 
%The authors in \cite{tss1} proved the local exact internal controllability of steady state solutions for the Navier-Stokes equations in three-dimensional bounded domains with the Navier slip boundary conditions by applying the Carleman approximation for the backward Stokes equation (also see \cite{tss2} for the 2D case ). 
By using low-mode forcing, the authors in \cite{agr} proved the $\mathrm{L}^2$-approximate controllability for 2D-NSE, while for 3D-NSE it has been established in \cite{shir}. For  3D-NSE, the author in \cite{khoff} proved the local approximate controllability on the torus. Further in \cite{shir1}, the author proved the exact controllabilty in projections as well as uniform approximate controllability for 3D-NSE. The global approximate null controllability of 3D NSE by means of a boundary control was proved in \cite{SGOY}. 
%The authors in \cite{CF} prove the approximate controllability for \emph{cut-off} NSE by using the fixed argument of the corresponding linearised system (Stokes system) and unique continuation property. 
Recently, the authors in \cite{sKM1} established feedback stabilization  for the CBFeD system preserving the invariance of a given convex set. Furthermore, in \cite{pkmtm}, the authors examined the local exact controllability to trajectories for CBF equations.

\subsubsection{Irreducibility}
Control theory is a useful tool to study the properties of stochastic dynamical systems. One of the important properties of stochastic dynamic systems is the concept of \emph{irreducibility} of transition semigroup associated to the system. An additional convincing reason for the study of the irreducibility property is its relevance in the analysis of the uniqueness and ergodicity of invariant measures. It is well known that the strong Feller property together with irreducibility ensure the uniqueness of invariant measures for the transition semigroup (see \cite{pjz,gdp2}).  See, for instance, the classical work \cite{jLD} and the books \cite{pjz1,gdp2} for stochastic infinite-dimensional systems. The analysis of irreducibility for semilinear SPDEs with additive or multiplicative noises has been widely explored in the literature. The author in \cite{vbd} proved the irreducibility of the transition semigroup associated with the two-phase Stefan problem (perturbed by a colored noise). The proof relies on a general result of approximate controllability for \emph{maximal monotone} systems (see \cite[Chapter 2]{VB2} for a general introduction on maximal monotone operators). Additionally, \cite{vbd1} demonstrated the irreducibility of a stochastic obstacle problem by employing a similar methodology as in \cite{vbd}. Later, the authors in \cite{VB9} extend this result to a more general class of stochastic differential equations by using the theory of \emph{$m$-accretive operators}. For non-linear stochastic partial differential equations (e.g., stochastic Burgers equations, stochastic porous media equations, and stochastic fast diffusion equations), the author in \cite{SQ} provided sufficient conditions (non-degenracy conditions) such that the Markov semigroup is irreducible and strong Feller. In the context of NSE, the author in \cite{ffL} showed that the 3D-NSE with random forcing is irreducible under the full noise assumption (that is, noise interacts with all modes), and initial data belongs to the Sobolev space $\mathbb{H}^1$.  For more details on ergodicity for 2D and 3D NSE perturbed by degenerate and non-degenerate noises, interested readers are referred to see \cite{AD,ZDYX,WEJC,FFBM,NGHJC,MHJC,MTSKSS,gdp2,GDPAD,MR1,MRLX}, etc. and references therein.

 \subsection{Difficulties, novelties and approaches}
 The main advantage for considering the CBFeD system \eqref{1} in a $d$-dimensional torus is as follows: In the torus $\mathbb{T}^d$, the Helmholtz-Hodge projection $\mathcal{P}$ and $-\Delta$ commutes (cf. \cite[Theorem 2.22]{JCR}). So, the equality 
 \begin{align}\label{3}
 	&\int_{\mathbb{T}^d}(-\Delta \boldsymbol{y}(x))\cdot|\boldsymbol{y}(x)|^{r-1}\boldsymbol{y}(x)\d x\nonumber\\&=\int_{\mathbb{T}^d}|\nabla \boldsymbol{y}(x)|^2|\boldsymbol{y}(x)|^{r-1}\d x +4\left[\frac{r-1}{(r+1)^2}\right]\int_{\mathbb{T}^d}|\nabla|\boldsymbol{y}(x)|^{\frac{r+1}{2}}|^2\d x,
 \end{align}
 is quite useful in obtaining regularity results. It is also noticed in the literature that the above equality may not be useful in  domains other than the whole domain or a $d$-dimensional torus (see \cite{KT2},  etc. for a detailed discussion).   Even though the equations are defined on a torus, we are not assuming the zero mean condition for the velocity field unlike the case of NSE, since the damping term $\beta|\y|^{r-1}\y$ and the pumping term  $\gamma|\y|^{q-1}\y$ ($1<q<r$) do not preserve this property (see \cite{MTT}). This yields that $0$ is an eigenvalue of $-\Delta$, and hence $-\Delta$ is not invertible. As a result, we cannot use the well-known Poincar\'e inequality, and we have to work with the full $\mathbb{H}^1$-norm.
     
 As the Navier-Stokes  system is dissipative and non-reversible, it is widely known that one cannot expect exact controllability of NSE with an arbitrary target function \cite{jmc5,JL2}. This also holds true for the CBFeD system. On the other hand, we do not know the answer for the approximate controllability of the system \eqref{1} as long as the control is acting in an interior domain  (\cite{SGOY}). 	We first show the  approximate controllability of the system \eqref{1}  in the energy space (divergence-free $\mathbb{L}^2$-space, denoted by $\H$, see Subsection \ref{sub2.1})  by using a distributed control and adopting ideas from the work \cite{vbd} (also \cite{vbd1,VB9}). The $m$-accretivity of the linear and nonlinear operators play a crucial role in establishing the approximate controllability results.  We use a  density argument, a smooth approximation of  the multi-valued map $\sgn(\cdot)$  and finite time extinction property (see Lemma \ref{exact} and \ref{lem3.3}) to obtain the required results.   To the best of our knowledge, in the context of convective Brinkman-Forchheimer equations, the approximate controllability results are not that much explored in the literature. The only known work is \cite{MT9}, where the author established the approximate controllability result in $\H$ with respect to the initial data (viewed as an initial controller) by using the backward uniqueness.
 
Our second aim of this work is to prove the irreducibility of the transition semigroup associated with the stochastic counterpart of  the CBFeD system \eqref{1} in the energy space $\H$ by an application of the approximate controllability results for the system \eqref{1}. But it is not immediate as established in the works \cite{vbd,vbd1,VB9}, etc., where the irreducibility results are obtained in $\mathrm{H}^{-1}$. Moreover, the authors  in \cite{akmtm} established irreducibility results in $\D(\A^{\alpha}),$ where $\A$ is the Stokes operator (see Subsection \ref{sub2.1}),   for $\alpha\in[\frac{1}{4},\frac{1}{2})$ ($r=1,2$) and  for $\alpha\in[\frac{1}{3},\frac{1}{2})$ ($r=3$)    for 2D stochastic CBF equations perturbed by non-degenerate noise.  But we prove the irreducibility results in $\H$ and in order to establish this, we assume that our noise is non-degenerate (see \eqref{423} and Example \ref{exmp4.1} below) and for $\y_0\in\V$, we use the  regularity result given in \eqref{13} for the system \eqref{1}. This method of proving irreducibility appears to be novel in the literature. Note that our results are true for $d=2,3$ with $r\in(3,\infty)$ and $2\beta\mu>1$ for $r=3$  as well as for critical and supercritical CBF equations. 
 	
 \subsection{Organization of the paper} The rest of the paper is organized as follows: In the upcoming section, we present a functional setting needed for the theory of CBFeD equations \eqref{1}, while also examining commonly employed concepts such as the definitions and characteristics of linear, bilinear, and nonlinear operators. Moreover, we also provide the properties of quasi-$m$-accretivity of single-valued and multi-valued operators (Proposition \ref{prop3.1}-\ref{prop3.2}). Section \ref{apcon} explores the topic of approximate controllability specifically related to the CBFeD system \eqref{pf1}. In order to establish this, our initial step involves demonstrating the finite time extinction property of the system \eqref{pf1.1} (Proposition \ref{exact}), with the initial and final data confined within $\D(\A)$. We then extend this result to the whole of $\H$ by taking advantage of a density argument (Proposition \ref{lem3.3}). Ultimately, through the utilization of the finite time extinction property and another density result (Lemma \ref{dense}), we establish the proof of approximate controllability in $\H$ for the CBFeD system \eqref{pf1}. Next, we discuss the irreducibility of the transition semigroup associated with the CBFeD system, which is another significant component of this work. We first perform the stochastic setup of the CBFeD system in Section \ref{stochastic}, where we also provide the relevant details about the noise, the trace-class operator, and the probabilistic notion of a strong solution (see to Definition \ref{def4.2}). In order to achieve irreducibility, it is imperative to obtain certain regularity results for the stochastic CBFeD system \eqref{scbf}. These results are demonstrated in Theorem \ref{regularity}, employing the It\^o formula and Burkholder-Davis-Gundy inequality. Subsequently, in Proposition \ref{engest}, we ascertain the energy estimates of the random system \eqref{scbf1} through the utilization of the solvability of the Ornstein-Uhlenbeck process satisfying \eqref{oup}. The proof of our major finding (Theorem \ref{mainresult}), which states that the  system \eqref{scbf} is irreducible in $\H$, is the focus of the final section. We eventually conclude the proof of the main result by using a density argument, regularity of the solutions of the stochastic system \eqref{scbf}, and approximate controllability result (see Section \ref{apcon}).   As an immediate consequence of the irreducibility result, we obtain the  accessibility of the stochastic CBFeD system \eqref{scbf} (Corollary \ref{cor5.2}).

\section{Functional Settings and Preliminaries}\label{fn}\setcounter{equation}{0}
The aim of this section is to discuss some basic functional setting and preliminaries needed to establish the main results of this work. As discussed earlier, we consider  $\mathbb{T}^d=\left(\frac{\R}{\mathrm{L}\mathbb{Z}}\right)^d$ where $d\in\{2,3\},$ a $d$-dimensional torus.  
\subsection{Function spaces}\label{sub2.1} Let \ $\C_{\mathrm{p}}^{\infty}(\mathbb{T}^d;\R^d)$ denote the space of all infinitely differentiable functions ($\mathbb{R}^d$-valued) satisfying periodic boundary conditions of the form \eqref{2}. The Sobolev space  $\H_{\mathrm{p}}^s(\mathbb{T}^d):=\mathrm{H}_{\mathrm{p}}^s(\mathbb{T}^d;\mathbb{R}^d)$ is the completion of $\C_{\mathrm{p}}^{\infty}(\mathbb{T}^d;\R^d)$  with respect to the $\H^s$-norm and the norm on the space $\H_{\mathrm{p}}^s(\mathbb{T}^d)$ is given by $$\|\y\|_{{\H}^s_{\mathrm{p}}}:=\left(\sum_{0\leq|\boldsymbol\alpha|\leq s}\|\D^{\boldsymbol\alpha}\y\|_{\mathbb{L}^2(\mathbb{T}^d)}^2\right)^{1/2}.$$ 	
It is known from \cite{JCR1} that the Sobolev space of periodic functions $\H_{\mathrm{p}}^s(\mathbb{T}^d)$ is same as the space
$$\H_{\mathrm{f}}^s(\mathbb{T}^d)=\left\{\y:\y=\sum_{k\in\mathbb{Z}^d}\y_{k}\mathrm{e}^{2\pi i k\cdot x /  \mathrm{L}},\ \overline{\y}_{k}=\y_{-k}, \  \|\y\|_{{\H}^s_\mathrm{f}}:=\left(\sum_{k\in\mathbb{Z}^d}(1+|k|^{2s})|\y_{k}|^2\right)^{1/2}<\infty\right\}.$$ We infer from \cite[Proposition 5.38]{JCR1} that the norms $\|\cdot\|_{{\H}^s_{\mathrm{p}}}$ and $\|\cdot\|_{{\H}^s_{\mathrm{f}}}$ are equivalent. Let us define 
\begin{align*} 
	\mathcal{V}:=\{\y\in\C_{\mathrm{p}}^{\infty}(\mathbb{T}^d;\R^d):\nabla\cdot\y=0\}.
\end{align*}
We define the spaces $\H$ and $\widetilde{\L}^{p}$ as the closure of $\mathcal{V}$ in the Lebesgue spaces $\mathrm{L}^2(\mathbb{T}^d;\R^d)$ and $\mathrm{L}^p(\mathbb{T}^d;\R^d)$ for $p\in(2,\infty)$, respectively. We also define the space $\V$ as the closure of $\mathcal{V}$ in the Sobolev space $\mathrm{H}^1(\mathbb{T}^d;\R^d)$. Then, we characterize the spaces $\H$, $\widetilde{\L}^p$ and $\V$ with the norms  $$\|\y\|_{\H}^2:=\int_{\mathbb{T}^d}|\y(x)|^2\d x,\quad \|\y\|_{\widetilde{\L}^p}^p:=\int_{\mathbb{T}^d}|\y(x)|^p\d x\ \text{ and }\ \|\y\|_{\V}^2:=\int_{\mathbb{T}^d}(|\y(x)|^2+|\nabla\y(x)|^2)\d x,$$ respectively. 
Let $(\cdot,\cdot)$ denote the inner product in the Hilbert space $\H$ and $\langle \cdot,\cdot\rangle $ represent the induced duality between the spaces $\V$  and its dual $\V'$ as well as $\widetilde{\L}^p$ and its dual $\widetilde{\L}^{p'}$, where $\frac{1}{p}+\frac{1}{p'}=1$. Note that $\H$ can be identified with its own dual $\H'$. From \cite[Subsection 2.1]{FKS}, we have that the sum space $\V'+\widetilde{\L}^{p'}$ is well defined and  is a Banach space with respect to the norm (see \cite{MT2} also)
\begin{align*}
	\|\y\|_{\V'+\widetilde{\L}^{p'}}&:=\inf\{\|\y_1\|_{\V'}+\|\y_2\|_{\wi\L^{p'}}:\y=\y_1+\y_2, \y_1\in\V' \ \text{and} \ \y_2\in\wi\L^{p'}\}\nonumber\\&=
	\sup\left\{\frac{|\langle\y_1+\y_2,\f\rangle|}{\|\f\|_{\V\cap\widetilde{\L}^p}}:\boldsymbol{0}\neq\f\in\V\cap\widetilde{\L}^p\right\},
\end{align*}
where $\|\cdot\|_{\V\cap\widetilde{\L}^p}:=\max\{\|\cdot\|_{\V}, \|\cdot\|_{\wi\L^p}\}$ is a norm on the Banach space $\V\cap\widetilde{\L}^p$. Also the norm $\max\{\|\y\|_{\V}, \|\y\|_{\wi\L^p}\}$ is equivalent to the norms  $\|\y\|_{\V}+\|\y\|_{\widetilde{\L}^{p}}$ and $\sqrt{\|\y\|_{\V}^2+\|\y\|_{\widetilde{\L}^{p}}^2}$ on the space $\V\cap\widetilde{\L}^p$. Furthermore, we have
$
(\V'+\widetilde{\L}^{p'})'=	\V\cap\widetilde{\L}^p \  \text{and} \ (\V\cap\widetilde{\L}^p)'=\V'+\widetilde{\L}^{p'},
$
Moreover, we have the continuous embeddings $$\V\cap\widetilde{\L}^p\hookrightarrow\V\hookrightarrow\H\cong\H'\hookrightarrow\V'\hookrightarrow\V'+\widetilde{\L}^{p'},$$ where the embedding $\V\hookrightarrow\H$ is compact. 

\subsection{Linear operator}
Let $\mathcal{P}_p: \L^p(\mathbb{T}^d) \to\wi\L^p,$ $p\in[1,\infty)$ be the Helmholtz-Hodge (or Leray) projection operator (cf.  \cite{JBPCK,DFHM}, etc.).	Note that $\mathcal{P}_p$ is a bounded linear operator and for $p=2$,  $\mathcal{P}:=\mathcal{P}_2$ is an orthogonal projection (see \cite[Section 2.1]{JCR}). We define the Stokes operator 
\begin{equation*}
	\left\{
	\begin{aligned}
		\A\y&:=-\mathcal{P}\Delta\y=-\Delta\y,\;\y\in\D(\A),\\
		\D(\A)&:=\V\cap{\H}^{2}_\mathrm{p}(\mathbb{T}^d). 
	\end{aligned}
	\right.
\end{equation*}
For the Fourier expansion $\y=\sum\limits_{k\in\mathbb{Z}^d} e^{2\pi i k\cdot x} \y_{k} ,$ we calculate by using Parseval's identity
\begin{align*}
	\|\y\|_{\H}^2=\sum\limits_{k\in\mathbb{Z}^d} |\y_{k}|^2 \  \text{and} \ \|\A\y\|_{\H}^2=(2\pi)^4\sum_{k\in\mathbb{Z}^d}|k|^{4}|\y_{k}|^2.
\end{align*}
Therefore, we have 
\begin{align*}
	\|\y\|_{\H^2_\mathrm{p}}^2=\sum_{k\in\mathbb{Z}^d}|\y_{k}|^2+\sum_{k\in\mathbb{Z}^d}|k|^{4}| \y_{k}|^2=\|\y\|_{\H}^2+\frac{1}{(2\pi)^4}\|\A\y\|_{\H}^2\leq\|\y\|_{\H}^2+\|\A\y\|_{\H}^2.
\end{align*}
Moreover, by the definition of $\|\cdot\|_{\H^2_\mathrm{p}}$, we have $	\|\y\|_{\H^2_\mathrm{p}}^2\geq\|\y\|_{\H}^2+\|\A\y\|_{\H}^2$ and hence it is immediate that both the norms are equivalent and  $\D(\I+\A)=\H^2_\mathrm{p}(\mathbb{T}^d)$. 

\subsection{Bilinear operator}
Let us define the \textit{trilinear form} $b(\cdot,\cdot,\cdot):\V\times\V\times\V\to\R$ by $$b(\y,\z,\w)=\int_{\mathbb{T}^d}(\y(x)\cdot\nabla)\z(x)\cdot\w(x)\d x=\sum_{i,j=1}^d\int_{\mathbb{T}^d}\y_i(x)\frac{\partial \z_j(x)}{\partial x_i}\w_j(x)\d x.$$ If $\y, \z$ are such that the linear map $b(\y, \z, \cdot) $ is continuous on $\V$, the corresponding element of $\V'$ is denoted by $\mathcal{B}(\y, \z)$. We also denote $\mathcal{B}(\y) = \mathcal{B}(\y, \y):=\mathcal{P}[(\y\cdot\nabla)\y]$.
An integration by parts yields 
\begin{equation*}
	\left\{
	\begin{aligned}
		b(\y,\z,\w) &=  -b(\y,\w,\z),\ &&\text{ for all }\ \y,\z,\w\in \V,\\
		b(\y,\z,\z) &= 0,\ &&\text{ for all }\ \y,\z \in\V.
	\end{aligned}
	\right.\end{equation*}
By an application of H\"older's inequality, we calculate
\begin{align*}
	|b(\y,\z,\w)|=|b(\y,\w,\z)|\leq \|\y\|_{\widetilde{\L}^{r+1}}\|\z\|_{\widetilde{\L}^{\frac{2(r+1)}{r-1}}}\|\w\|_{\V},
\end{align*}
for all $\y\in\V\cap\widetilde{\L}^{r+1}$, $\z\in\V\cap\widetilde{\L}^{\frac{2(r+1)}{r-1}}$ and $\w\in\V$, so that we get 
\begin{align}\label{2p9}
	\|\mathcal{B}(\y,\z)\|_{\V'}\leq \|\y\|_{\widetilde{\L}^{r+1}}\|\z\|_{\widetilde{\L}^{\frac{2(r+1)}{r-1}}}.
\end{align}
Hence, the trilinear map $b : \V\times\V\times\V\to \R$ has a unique extension to a bounded trilinear map from $(\V\cap\widetilde{\L}^{r+1})\times(\V\cap\widetilde{\L}^{\frac{2(r+1)}{r-1}})\times\V$ to $\R$. It can also be seen that $\mathcal{B}$ maps $ \V\cap\widetilde{\L}^{r+1}$  into $\V'+\widetilde{\L}^{\frac{r+1}{r}}$ and using interpolation inequality, we get for $r>3$
\begin{align*}
	\left|\langle \mathcal{B}(\y,\y),\z\rangle \right|=\left|b(\y,\z,\y)\right|\leq \|\y\|_{\widetilde{\L}^{r+1}}\|\y\|_{\widetilde{\L}^{\frac{2(r+1)}{r-1}}}\|\z\|_{\V}\leq\|\y\|_{\widetilde{\L}^{r+1}}^{\frac{r+1}{r-1}}\|\y\|_{\H}^{\frac{r-3}{r-1}}\|\z\|_{\V},
\end{align*}
for all $\z\in\V\cap\widetilde{\L}^{r+1}$. Thus, from here we deduce that
\begin{align*}
\|\mathcal{B}(\y)\|_{\V'+\wi\L^{\frac{r+1}{r}}}\leq\|\y\|_{\widetilde{\L}^{r+1}}^{\frac{r+1}{r-1}}\|\y\|_{\H}^{\frac{r-3}{r-1}},
\end{align*}
for all $r>3$. By using \eqref{2p9} and interpolation inequality, we calculate for $r>3$ (see \cite[Subsection 2.2]{MT2}) 
\begin{align}\label{bilin}
\|\mathcal{B}(\y)-\mathcal{B}(\z)\|_{\V'+\widetilde{\L}^{\frac{r+1}{r}}}\leq \left(\|\y\|_{\H}^{\frac{r-3}{r-1}}\|\y\|_{\widetilde{\L}^{r+1}}^{\frac{2}{r-1}}+\|\z\|_{\H}^{\frac{r-3}{r-1}}\|\z\|_{\widetilde{\L}^{r+1}}^{\frac{2}{r-1}}\right)\|\y-\z\|_{\widetilde{\L}^{r+1}},
\end{align}
for all $\y,\z\in\V\cap\widetilde{\L}^{r+1}$. Hence, the map $\mathcal{B}(\cdot):\V\cap\widetilde{\L}^{r+1}\to\V'+\widetilde{\L}^{\frac{r+1}{r}}$ is a locally Lipschitz operator.     

\subsection{Nonlinear operator}\label{nonlin}
Let us define an operator $\mathcal{C}_1(\y):=\mathcal{P}(|\y|^{r-1}\y)$ for $\y\in\V\cap\L^{r+1}$.  Since the projection operator $\mathcal{P}$ is bounded from $\H^1$ into itself (cf. \cite[Remark 1.6]{Te}), the operator $\mathcal{C}_1(\cdot):\V\cap\widetilde{\L}^{r+1}\to\V'+\widetilde{\L}^{\frac{r+1}{r}}$ is well-defined and we have $\langle\mathcal{C}_1(\y),\y\rangle =\|\y\|_{\widetilde{\L}^{r+1}}^{r+1}.$ 
For all $\y,\z\in\V\cap{\wi \L}^{r+1}$, we have (see \cite[Subsection 2.4]{MT4})
\begin{align}\label{C1}
	\langle\mathcal{C}_1(\y)-\mathcal{C}_1(\z),\y-\z\rangle&\geq \frac{1}{2}\||\y|^{\frac{r-1}{2}}(\y-\z)\|_{\H}^2+\frac{1}{2}\||\z|^{\frac{r-1}{2}}(\y-\z)\|_{\H}^2\geq \frac{1}{2^{r-1}}\|\y-\z\|_{\wi\L^{r+1}}^{r+1},
\end{align}
for all $r\in[1,\infty)$. Moreover, we define $\mathcal{C}_2(\y):=\mathcal{P}(|\y|^{q-1}\y)$ for $1\leq q<r$. The nonlinear operator $\mathcal{C}_2(\cdot)$ satisfies similar properties as $\mathcal{C}_1(\cdot)$. By the application of Taylor's formula, we find 
\begin{align*}
	|\langle\mathcal{C}_1(\y)-\mathcal{C}_1(\z),\w\rangle|&\leq r\sup\limits_{0<\theta<1}\|\theta\y+(1-\theta)\z\|_{\wi\L^{r+1}}^{r-1}\|\y-\z\|_{\wi\L^{r+1}}\|\w\|_{\wi\L^{r+1}}\nonumber\\&\leq r \left(\|\y\|_{\wi\L^{r+1}}+\|\z\|_{\wi\L^{r+1}}\right)^{r-1} \|\y-\z\|_{\wi\L^{r+1}}\|\w\|_{\wi\L^{r+1}},
\end{align*}
for all $\y,\z\in\V\cap\wi\L^{r+1}$ and for $r\geq1$. From here, we conclude that
\begin{align}\label{clin1}
	\|\mathcal{C}_1(\y)-\mathcal{C}_1(\z)\|_{\V'+\wi\L^{\frac{r+1}{r}}}\leq r2^{r-2}\left(\|\y\|_{\wi\L^{r+1}}^{r-1}+\|\z\|_{\wi\L^{r+1}}^{r-1}\right)\|\y-\z\|_{\wi\L^{r+1}},
\end{align}
so that the operator $\mathcal{C}_1(\cdot)$ is locally Lipschitz. 

The following properties of linear and nonlinear operators will be exploited in the rest of the paper. 
\begin{proposition}\label{prop3.1}
	For $d=2,3$ with $r>3$, define the operator $\Gamma(\cdot):\mathrm{D}(\mathcal{M})\to\H$ by
	\begin{align*}
		\Gamma(\cdot)=\mu\A+\alpha\mathrm{I}+\mathcal{B}(\cdot)+\beta\mathcal{C}_1(\cdot)+\gamma\mathcal{C}_2(\cdot), 
	\end{align*}
	where $\mathrm{D}(\Gamma)=\{\y\in\V\cap\wi\L^{r+1}:\A\y\in\H\}.$ Then, $\D(\Gamma)=\D(\A)$ and $\Gamma$ is quasi-$m$-accretive in $\H$ for sufficiently large $\kappa\geq\eta_1+\eta_2+\eta_3$ such that 
	\begin{align}\label{mono}
	((\Gamma+\kappa\I)(\y_1)-(\Gamma+\kappa\I)(\y_2),\y_1-\y_2)\geq0,
	\end{align} 
 for all $\y_1,\y_2\in\D(\A)$, where $\eta_1:=\frac{r-q}{r-1}\left[\frac{2^{q}q|\gamma|(q-1)}{\beta(r-1)}\right]^{\frac{q-1}{r-q}},$ $\eta_2:=\frac{r-q}{r-1}\left[\frac{2^{q+1}q|\gamma|(q-1)}{\beta(r-1)}\right]^{\frac{q-1}{r-q}}$ and $\eta_3:=\frac{r-3}{2\mu(r-1)}\left[\frac{4}{\varepsilon\beta\mu (r-1)}\right]^{\frac{2}{r-3}}.$
\end{proposition}

\begin{proof}
	See \cite[Proposition A.1]{sKM1}.
\end{proof}

\begin{proposition}\label{prop3.2}
	Define the multivalued operator  $\mathfrak{G}:\mathrm{D}(\mathfrak{A})\to\H$ by 
	\begin{equation*}
	\mathfrak{G}(\cdot) = \mu\mathrm{A}+\alpha\I +\mathcal{B}(\cdot)+\beta\mathcal{C}_1(\cdot)+\gamma\mathcal{C}_2(\cdot)+\rho\mathrm{sgn}(\cdot-\y_1)+
	\kappa\I,
	\end{equation*}
	with the domain $\mathrm{D}(\mathfrak{G})=\{\y\in\H: \varnothing \neq \mathfrak{G}(\y)\subset\H\}$,  where  $\kappa$ is the same as in Proposition \ref{prop3.1}. Then $\mathrm{D}(\mathfrak{G}) = \mathrm{D}(\mathrm{A})\cap\mathrm{D}(\Phi)$, where $\Phi(\y)=\mathrm{sgn}(\y-\y_1)=\partial(\|\y-\y_1\|_{\H})$, $\partial$ stands for the subdifferential, and $\mathfrak{G}$ is a maximal monotone operator in $\H\times\H,$ 
\end{proposition}
\begin{proof}
	See \cite[Proposition A.3]{sKM1}.
\end{proof}
\section{Approximate controllability}\label{apcon}\setcounter{equation}{0}
The aim of this section is to prove the approximate controllability result in $\H$ for CBFeD equations.  The problem of approximate controllability is a relaxed version of the exact controllability. Instead of requiring that the state of the system reaches the target exactly, one may wonder if, at least, one can get arbitrarily close to the target (cf. \cite{fab}). Mathematically speaking, this can be written as follows: we are given $T>0$ fixed and we are given an element $\y^T$ of a (suitable) function space. Think of $\y^T$ as being the target. Can we find a control $\v$ such that at time $T$ there is a state $\y(\cdot,T,\v)$ which is as ``close'' as we wish of $\y^T$?

Let us consider the following CBFeD equations with control:
\begin{equation}\label{pf1}
	\left\{
	\begin{aligned}
		\frac{\d \y(t)}{\d t}+\mu\A\y(t)+ \alpha\y(t)+\mathcal{B}(\y(t))+ \beta\mathcal{C}_1(\y(t))+\gamma\mathcal{C}_2(\y(t))&= \mathrm{B}\u(t),  \ \text{for a.e.} \ t>0,\\ 
		\y(0)&=\y_0,
	\end{aligned}
	\right.
\end{equation}	
where the initial data $\y_0\in\H$ and $\mathrm{B}:\U\to\H$ is a bounded linear operator from Hilbert space $\U$ to $\H$ and $\u\in\mathrm{L}^2(0,T;\U)$ is a distributed control. Let  $\y_1\in\H$ be fixed.  Our aim is to find an appropriate controller such that in a given time, we can steer the system \eqref{pf1}, from the initial state $\y_0$ to a neighborhood of the target $\y_1$, that is, to prove that the system is \emph{approximately controllable}.  Mathematically speaking, the system \eqref{pf1} is said to be \emph{approximately controllable}, if for arbitrary  $\y_0,\y_1\in\H$, there exists a control $\u\in\mathrm{L}^2(0,T;\U)$ such that for a given $\varepsilon>0, \ T>0$, we have 
\begin{align*}
	\|\y(T)-\y_1\|_{\H}\leq\varepsilon.
\end{align*}
Therefore, our goal is to show that 
\begin{proposition}\label{prop31}
Let $\y_0,\y_1\in\H$. Then, for all $\e>0$, there exists a control $\u\in\C([0,T];\U)$ such that $$\|\y(T)-\y_1\|_{\H}\leq \e.$$ 	That is, the system \eqref{pf1} is approximately controllable.
\end{proposition}
We divide the proof of Proposition \ref{prop31} into the following parts: 
\begin{lemma}\label{exact}
	Assume that $\y_0,\y_1\in\D(\A)$. Then there exists a control $\v\in\mathrm{L}^{\infty}(0,T;\H)$ such that the solution $\y(\cdot)$ to the problem 
	\begin{equation}\label{pf1.1}
		\left\{
		\begin{aligned}
\frac{\d \y(t)}{\d t}+ \mu\A\y(t)+\alpha\y(t)+\mathcal{B}(\y(t))+\beta\mathcal{C}_1(\y(t))+ \gamma\mathcal{C}_2(\y(t))&= \v(t),  \ \text{for a.e.} \ t\in[0,T],\\ 
			\y(0)&=\y_0,
		\end{aligned}
		\right.
	\end{equation}	
satisfies $\y\in\mathrm{W}^{1,\infty}(0,T;\H)\cap\mathrm{L}^{\infty}(0,T;\D(\A))\cap\mathrm{C}([0,T];\V)$ and $\y(T)=\y_1$.
\end{lemma}

\begin{proof}
We first consider the following  Cauchy problem: 
\begin{equation}\label{pf1.2}
	\left\{
	\begin{aligned}
		\frac{\d \y(t)}{\d t}+\mu\A\y(t)+\alpha\y(t)+\mathcal{B}(\y(t))+\beta\mathcal{C}_1(\y(t))+ \gamma\mathcal{C}_2(\y(t))+\rho(\mathrm{sgn}(\y(t)-\y_1))&\ni\0, \ \\
		\y(0)&=\y_0,
	\end{aligned}
	\right.
\end{equation}
 for a.e. $ t\in[0,T]$, where $\mathrm{sgn}(\cdot)$ is a multivalued mapping defined as 
\begin{align}\label{310}
	\mathrm{sgn}(\y)=
	\begin{cases}
		\frac{\y}{\|\y\|_{\H}}, &\text{ if } \ \y\neq\mathbf{0},\\
		\{\y\in\H:\|\y\|_{\H}\leq 1\}, &\text{ if } \ \y=\mathbf{0}.
	\end{cases}
\end{align}
	From Propositions \ref{prop3.1} and \ref{prop3.2}, it is clear that the operator $\Gamma(\cdot)+\rho(\mathrm{sgn}(\cdot-\y_1))$ is \emph{quasi $m$-accretive} in $\H\times\H$. Then, from \cite[Theorem 1.3]{sKM}, we infer that the Cauchy problem \eqref{pf1.2} has a \emph{unique strong solution}
\begin{align*}
	\y\in\W^{1,\infty}(0,T;\H)\cap\mathrm{L}^{\infty}(0,T;\D(\A))\cap \C([0,T];\V).
\end{align*}
Since, a strong solution exists, $\|\cdot\|_{\H}^2$ is absolutely continuous, and multiplying \eqref{pf1.2} scalarly by $\y(\cdot)-\y_1$, we get 
\begin{equation}\label{pf1.3}
	\frac{1}{2}\frac{\d }{\d t}\|\y(t)-\y_1\|_{\H}^2 + (\Gamma(\y(t)),\y(t)-\y_1)+\rho\|\y(t)-\y_1\|_{\H}=0,
\end{equation}
for a.e. $t\in[0,T]$. One can justify the above calculation by using a smooth approximation of the multi-valued map $\sgn(\cdot)$  (\cite[pp. 110, 213]{VB2}, see \eqref{smap} below). Using the quasi $m$-accretivity of $\Gamma(\cdot)$ (see \eqref{mono}), one gets
\begin{align*}
	(\Gamma(\y),\y-\y_1)&=(\Gamma(\y)-\Gamma(\y_1),\y-\y_1)+(\Gamma(\y_1),\y-\y_1)\no\\&\geq-\kappa\|\y_1-\y_1\|_{\H}^2+(\Gamma(\y_1),\y-\y_1),
\end{align*}
where $\kappa>0$ is a sufficiently large constant specified in Proposition \ref{prop3.1}. This gives together with \eqref{pf1.3}
\begin{align*}
	\frac{1}{2}\frac{\d }{\d t}\|\y(t)-\y_1\|_{\H}^2+ \rho\|\y(t)-\y_1\|_{\H} \leq\kappa\|\y_1(t)-\y_1\|_{\H}^2+\|\Gamma(\y_1)\|_{\H}\|\y(t)-\y_1\|_{\H},
\end{align*}
which implies 
\begin{align*}
	\frac{\d}{\d t}(e^{-\kappa t}\|\y(t)-\y_1\|_{\H})+ (\rho-\|\Gamma(\y_1)\|_{\H})e^{-\kappa t}\leq0,
\end{align*}
for a.e. $t\in[0,T]$. Integration  from $0$ to $t$ yields 
\begin{align*}
	e^{-\kappa t}\|\y(t)-\y_1\|_{\H}+(\rho-\|\Gamma(\y_1)\|_{\H})\left(\frac{1-e^{-\kappa t}}{\kappa} \right)\leq\|\y_0-\y_1\|_{\H},
\end{align*}
for all $t\in[0,T].$ This shows that for $\rho>\|\Gamma(\y_1)\|_{\H},$ we have $\|\y(t)-\y_1\|_{\H}=0$ for all
\begin{align*}
	t\geq T_0:=-\frac{1}{\kappa} \log\left(1-\frac{\kappa\|\y_0-\y_1\|_{\H}}{\rho-\|\Gamma(\y_1)\|_{\H}}\right).
\end{align*}
For sufficiently large $\rho$, one can take $T_0$ just equal to an a-priori fixed $T$. Thus $\v(\cdot):=-\rho(\mathrm{sgn}(\y(\cdot)-\y_1))$ (precisely, the single valued section) is the desired controller for which the solution $\y(\cdot)$ of 
\begin{equation}\label{pf1.4}
	\left\{
	\begin{aligned}
		\frac{\d \y(t)}{\d t}+\mu\A\y(t)+\alpha\y(t)+\mathcal{B}(\y(t))+\beta\mathcal{C}_1(\y(t))+ \gamma\mathcal{C}_2(\y(t))&=\v(t), \ \text{a.e.} \ t\in[0,T], \\
		\y(0)&=\y_0,
	\end{aligned}
	\right.
\end{equation}
satisfies $\y(T)=\y_1$. Moreover, it is clear that this control belongs to the constraint set $$\{\v\in\mathrm{L}^{\infty}(0,T;\H): \ \|\v\|_{\mathrm{L}^{\infty}(0,T;\H)}\leq\rho\},$$ which completes the proof. 
\end{proof}

\begin{lemma}\label{lem3.3}
	Assume that $\y_0,\y_1\in\H$. Then there exists a control $\v\in\mathrm{L}^{\infty}(0,T;\H)$ such that the  solution $\y(\cdot)$ to the problem \eqref{pf1.1}  satisfies  $\y(T)=\y_1$.
\end{lemma}

	\begin{proof}
		We prove this result in the following steps:
		\vskip 0.1 cm
		\noindent
		\textbf{Step 1:} $\y_0\in\H$ \emph{and} $\y_1\in\D(\A)$.
		Let us first assume that $\y_0\in\H$ and $\y_1\in\D(\A)$. Then there exists a sequence $\{\y_0^j\}_{j\in\mathbb{N}}\in\D(\A)$ such that $\y_0^j\to\y_0$ in $\H$ as $j\to\infty$. For each $j$, the solution $\y^j(\cdot)$ of the equation
		\begin{equation}\label{sgn}
			\left\{
			\begin{aligned}
				\frac{\d \y^{j}(t)}{\d t}+\Gamma(\y^j(t))+\rho \ \mathrm{sgn}(\y^{j}(t)-\y_1)&\ni\boldsymbol{0}, \ \text{a.e.} \ t\in[0,T], \\
				\y^{j}(0)&=\y_0^{j},
			\end{aligned}
			\right.
		\end{equation}
		where $\Gamma(\cdot):=\mu\A+\alpha\I+\mathcal{B}(\cdot)+\beta\mathcal{C}_1(\cdot)+\gamma\mathcal{C}_2(\cdot)$, satisfies $\y^j(T)=\y_1$. From Lemma \ref{exact}, the above Cauchy problem has unique strong solution $$\y^j\in\W^{1,\infty}(0,T;\H)\cap\mathrm{L}^{\infty}(0,T;\D(\A))\cap \C([0,T];\V).$$On taking inner product with $\y^j(\cdot)-\y_1$ in \eqref{sgn}, we obtain 
		\begin{align*}
			\frac{1}{2}\frac{\d}{\d t}\|\y^j(t)-\y_1\|_{\H}^2+(\Gamma(\y^j(t))-\Gamma(\y_1),\y^j(t)-\y_1) &= -\rho(\mathrm{sgn}(\y^{j}(t)-\y_1),\y^j(t)-\y_1)\nonumber\\&\quad+(\Gamma(\y_1),\y^j(t)-\y_1),
		\end{align*}
		for a.e. $t\in[0,T]$. From Proposition \ref{prop3.1}(cf. \eqref{mono}) and by using the definition of $\mathrm{sgn}(\cdot)$ function, we write
		\begin{align*}
			\frac{1}{2}\frac{\d}{\d t}\|\y^j(t)-\y_1\|_{\H}^2-\kappa\|\y^j(t)-\y_1\|_{\H}^2&\leq -\rho\|\y^{j}(t)-\y_1\|_{\H}+(\Gamma(\y_1),\y^j(t)-\y_1)\nonumber\\&\leq
			-\rho\|\y^{j}(t)-\y_1\|_{\H}+\|\Gamma(\y_1)\|_{\H}\|\y^j(t)-\y_1)\|_{\H},
		\end{align*}
		for a.e. $t\in[0,T]$. Since $\rho$ is sufficiently large, so, we can write
		\begin{align*}
			\frac{1}{2}\frac{\d}{\d t}\|\y^j(t)-\y_1\|_{\H}^2+(\rho-\|\Gamma(\y_1)\|_{\H})\|\y^{j}(t)-\y_1\|_{\H} \leq\kappa\|\y^j(t)-\y_1\|_{\H}^2,
		\end{align*}
		for a.e. $t\in[0,T]$. Thus, by Gronwall's inequality, we get
		\begin{align*}
			\|\y^j(t)-\y_1\|_{\H}^2\leq e^{-\kappa(t-T)}\|\y^j(T)-\y_1\|_{\H}^2,
		\end{align*}
		for all $t\in[0,T]$. This yields that $\y^j(t)=\y_1$ for all $t\geq T$. Let us now set $\v^j(\cdot):=-\rho \ \mathrm{sgn}(\y^j(\cdot)-\y_1)$. Let $\y(\cdot)$ be the solution of the following system:
		\begin{equation}\label{sgn2}
			\left\{
			\begin{aligned}
				\frac{\d \y(t)}{\d t}+\Gamma(\y(t))+\rho \ \mathrm{sgn}(\y(t)-\y_1)&\ni\boldsymbol{0}, \ \text{ for a.e. } t\in[0,T],\\
				\y(0)&=\y_0\in\H.
			\end{aligned}
			\right.
		\end{equation}
		Note that for any $\y_0\in\H$, the system \eqref{sgn2} has unique weak solution (cf. \cite[Theorem 3.3]{sKM1}). On subtracting \eqref{sgn2} from \eqref{sgn} and taking the inner product with $\y^j(\cdot)-\y$, we get 
		\begin{align*}
			&\frac{1}{2}\frac{\d}{\d t}\|\y^j(t)-\y(t)\|_{\H}^2+ (\Gamma(\y^j(t))-\Gamma(\y(t)),\y^j(t)-\y(t))\leq0,
		\end{align*}
		for a.e. $t\in[0,T]$, where we use the monotonicity of $\mathrm{sgn}(\cdot-\y_1)$. 
		From Proposition \ref{prop3.1} (cf. \eqref{mono}), we write
		\begin{align*}
			\frac{\d}{\d t}\|\y^j(t)-\y(t)\|_{\H}^2\leq2\kappa\|\y^j(t)-\y(t)\|_{\H}^2,
		\end{align*}
		for a.e. $t\in[0,T]$. On applying Gronwall's inequality and using the fact that $\y_0^j\to\y_0$ in $\H$ as $j\to\infty$, we deduce 
		\begin{align}\label{315}
			\y^j\to\y \text{ in } \mathrm{C}([0,T];\H).
		\end{align}
		Since the sequence $\{\v^j(\cdot)\}_{k\in\mathbb{N}}$ is bounded in $\mathrm{L}^{\infty}(0,T;\H)$, an application of the Banach-Alaoglu theorem yields the existence of a subsequence (still denoted by the same) such that $\v^j\xrightarrow{w^*}\xi$ in $\mathrm{L}^{\infty}(0,T;\H)$. We consider  the following smooth approximation of the signum function $\mathrm{sgn}(\y)$ defined in \eqref{310} (\cite[pp. 213]{VB2}) for the feedback controller $\v^j(\cdot)=-\rho(\mathrm{sgn}(\y^j(\cdot)-\y_1))$:
		\begin{align}\label{smap}
			\vartheta_{\eps}(\y):=\begin{cases}
				\frac{1}{\e}\y , \  &\text{for }  \ \|\y\|_{\H}\leq\e,\\
				\frac{\y}{\|\y\|_{\H}}, \  &\text{for } \ \|\y\|_{\H}>\e.
			\end{cases}
		\end{align}
		It is clear that as $\eps\to 0$, $\vartheta_{\eps}(\y)\to \zeta$ in $\mathbb{L}^{\infty}(\mathbb{T}^d)$, where $\zeta\in \mathrm{sgn}(\y(x))$ for a.e $x\in\mathbb{T}^d$. 	We now show that $\v^j\to\v$ in $\mathrm{L}^2(0,T;\H)$, as $j\to\infty$, where $\v(\cdot)=-\rho(\mathrm{sgn}(\y(\cdot)-\y_1))$. By using \eqref{pf1.4}, we calculate 
		\begin{align}\label{sigapp}
			\|\v^j-\v\|_{\H}^2&=(\v^j-\v,\v^j-\v)
			\no\\&=-\rho\left\langle\mathrm{sgn}(\y^j-\y_1)-\mathrm{sgn}(\y-\y_1),\frac{\d}{\d t}(\y^j-\y)\right\rangle \no\\&\quad -\rho\left<\mathrm{sgn}(\y^j-\y_1)-\mathrm{sgn}(\y-\y_1), \Gamma(\y^j)-\Gamma(\y)\right>\no\\&=-\rho\lim\limits_{\eps\to0} \underbrace{\left<\vartheta_{\eps}(\y^j-\y_1)-\vartheta_{\eps}(\y-\y_1),\frac{\d}{\d t} (\y^j-\y)\right>}_{I_{\e}^j}\no\\&\quad-\rho\lim\limits_{\eps\to0}\underbrace{\left<\vartheta_{\eps}(\y^j-\y_1) -\vartheta_{\eps}(\y-\y_1), \Gamma(\y^j)-\Gamma(\y)\right>}_{J_{\e}^j}.
		\end{align}
	We consider the following two cases:\\
	\vskip 0.1mm
	\noindent
	\emph{When $\|\y^j(\cdot)-\y_1\|_{\H}\leq\e$:} The convergence given in  \eqref{315} implies that if $\|\y^j(\cdot)-\y_1\|_{\H}\leq\e$,  for sufficiently large $j\in\N$, then $\|\y(\cdot)-\y_1\|_{\H}\leq\e$. Therefore,  for sufficiently large $j\in\N$, we have 
		\begin{align*}
			-\rho I_{\e}^j-\rho J_{\e}^j&=-\frac{\rho}{\e}\left<\y^j-\y,\frac{\d}{\d t}(\y^j-\y)\right> -\frac{\rho}{\e}\left<\y^j-\y,\Gamma(\y^j)-\Gamma(\y)\right>\nonumber\\&\leq-\frac{\rho}{2\e}\frac{\d}{\d t}\|\y^j-\y\|_{\H}^2,
		\end{align*}
	where we have used the  monotonicity of $\Gamma(\cdot)$. 	Integrating over $t\in[0,T]$,  we get
		\begin{align*}
			-&\rho\int_0^TI_{\e}^j(t)\d t-\rho\int_0^TJ_{\e}^j(t)\d t\nonumber\\&\leq-\frac{\rho}{2\e}\|\y^j(T)-\y(T)\|_{\H}^2 +\frac{\rho}{2\e}\|\y^j(0)-\y(0)\|_{\H}^2\nonumber\\&\leq \frac{\rho}{2\e}\|\y^j(0)-\y(0)\|_{\H}^2\leq\frac{\rho}{\e}\left(\|\y^j(0)-\y_1\|^2_{\H}+\|\y_1-\y(0)\|_{\H}^2\right)\nonumber\\&\leq 2\rho\e,
		\end{align*}
so that the existence of the limit in \eqref{sigapp} provides 
	\begin{align*}
		-\rho \lim\limits_{\e\to 0}\left[\int_0^TI_{\e}^j(t)\d t+\int_0^TJ_{\e}^j(t)\d t\right]\leq 0\Rightarrow  \int_0^T \|\v^j(t)-\v(t)\|_{\H}^2\d t\leq 0, 
	\end{align*}
	for sufficiently large $j\in\N$. 
		Thus, one can conclude that $\v^j\to\v$ in $\mathrm{L}^2(0,T;\H)$ as $j\to\infty$. 
	\\
	\vskip 0.1mm
	\noindent
	\emph{When $\|\y^j(\cdot)-\y_1\|_{\H}>\e$:}
		For the case $\|\y^j(\cdot)-\y_1\|_{\H}>\e$, one can calculate from \eqref{sigapp}
		\begin{align}\label{d1d2}
			-\rho I_{\e}^j-\rho J_{\e}^j&=-\rho\underbrace{\left<\frac{\y^j-\y_1}{\|\y^j-\y_1\|_{\H}} -\frac{\y-\y_1}{\|\y-\y_1\|_{\H}},\frac{\d}{\d t}(\y^j-\y)\right>}_{D^j_1} \nonumber\\&\quad-\rho\underbrace{\left<\frac{\y^j-\y_1}{\|\y^j-\y_1\|_{\H}} -\frac{\y-\y_1}{\|\y-\y_1\|_{\H}},\frac{\d}{\d t}(\y^j-\y)\right>}_{D^j_2}.
		\end{align}
		Let us now calculate $D^j_1$ as
		\begin{align*}
			D^j_1&=\left<\frac{\y^j-\y_1}{\|\y^j-\y_1\|_{\H}},\frac{\d}{\d t}(\y^j-\y)\right> -\left<\frac{\y-\y_1}{\|\y-\y_1\|_{\H}},\frac{\d}{\d t}(\y^j-\y)\right>\nonumber\\&=
			\frac{1}{\|\y^j-\y_1\|_{\H}}\left<\y^j-\y,\frac{\d}{\d t}(\y^j-\y)\right>+ \frac{1}{\|\y^j-\y_1\|_{\H}}\left<\y-\y_1,\frac{\d}{\d t}(\y^j-\y)\right>\nonumber\\&\quad+ \frac{1}{\|\y-\y_1\|_{\H}}\left<\y^j-\y,\frac{\d}{\d t}(\y^j-\y)\right>-
			\frac{1}{\|\y-\y_1\|_{\H}}\left<\y^j-\y_1,\frac{\d}{\d t}(\y^j-\y)\right>%\nonumber\\&=	\frac{1}{2\|\y^j-\y_1\|_{\H}}\frac{\d}{\d t}\|\y^j-\y\|_{\H}^2+ \frac{1}{2\|\y-\y_1\|_{\H}}\frac{\d}{\d t}\|\y^j-\y\|_{\H}^2\nonumber\\&\quad+ \frac{1}{\|\y^j-\y_1\|_{\H}}\left<\y-\y_1,\frac{\d}{\d t}(\y^j-\y)\right>-		\frac{1}{\|\y-\y_1\|_{\H}}\left<\y^j-\y_1,\frac{\d}{\d t}(\y^j-\y)\right>
			\nonumber\\&=
			\frac{1}{2}\left(\frac{1}{\|\y-\y_1\|_{\H}}+ \frac{1}{\|\y^j-\y_1\|_{\H}}\right)\frac{\d}{\d t} \|\y^j-\y\|_{\H}^2\nonumber\\&\quad+\left<\frac{\y-\y_1}{\|\y^j-\y_1\|_{\H}}-\frac{\y^j-\y_1}{\|\y-\y_1\|_{\H}},\frac{\d}{\d t}(\y^j-\y)\right>.
		\end{align*}
		Similarly, we estimate 
		\begin{align*}
			D^j_2&=\left(\frac{1}{\|\y-\y_1\|_{\H}}+ \frac{1}{\|\y^j-\y_1\|_{\H}}\right) \langle\y^j-\y,\Gamma(\y^j)-\Gamma(\y)\rangle\nonumber\\&\quad+ \left<\frac{\y-\y_1}{\|\y^j-\y_1\|_{\H}}-\frac{\y^j-\y_1}{\|\y-\y_1\|_{\H}},\Gamma(\y^j)-\Gamma(\y)\right>.
		\end{align*} 
	Therefore, combining the above estimates, we have 
		\begin{align}\label{3.12}
			D^j_1+D^j_2&=\bigg(\frac{1}{\|\y-\y_1\|_{\H}}+\frac{1}{\|\y^j-\y_1\|_{\H}}\bigg) \bigg[\frac{1}{2}\frac{\d}{\d t}\|\y^j-\y_1\|_{\H}^2+ \langle\y^j-\y,\Gamma(\y^j)-\Gamma(\y)\rangle\bigg]\nonumber\\&\quad
             +\left<\frac{\y-\y_1}{\|\y^j-\y_1\|_{\H}}-\frac{\y^j-\y_1}{\|\y-\y_1\|_{\H}},\frac{\d}{\d t} (\y^j-\y)+\Gamma(\y^j)-\Gamma(\y)\right>.
		\end{align}
	Now, for $\|\y^j-\y_1\|_{\H}>\e$, we  infer from \eqref{sgn} and \eqref{sgn2} that  for a.e. $t\in[0,T],$ 
	\begin{align*}
		\frac{\d\y^j}{\d t}+\Gamma(\y^j)+\rho\frac{\y^j-\y_1}{\|\y^j-\y_1\|_{\H}}=\boldsymbol{0} \ \text{ and }\ 
 	\frac{\d\y}{\d t}+\Gamma(\y)+\rho\frac{\y-\y_1}{\|\y-\y_1\|_{\H}}=\boldsymbol{0},
 \end{align*}
in $\V'$. Subtracting the above two  systems, we obtain for a.e. $ t\in[0,T],$ 
\begin{align}\label{dd1}
	\frac{\d}{\d t}(\y^j-\y)+\Gamma(\y^j)-\Gamma(\y)= -\rho\bigg[\frac{\y^j-\y_1}{\|\y^j-\y_1\|_{\H}}-\frac{\y-\y_1}{\|\y-\y_1\|_{\H}}\bigg],
\end{align}
in $\V'$. 	On taking the  inner product with $\y^j-\y$, we deduce 
	\begin{align}\label{dd2}
\frac{1}{2}\frac{\d}{\d t}\|\y^j-\y\|_{\H}^2+\langle\y^j-\y,\Gamma(\y^j)-\Gamma(\y)\rangle=-\rho  \left(\frac{\y^j-\y_1}{\|\y^j-\y_1\|_{\H}}-\frac{\y-\y_1}{\|\y-\y_1\|_{\H}},\y^j-\y\right).
	\end{align}	 
for a.e. $t\in[0,T]$. Substituting  \eqref{dd1}-\eqref{dd2} in \eqref{3.12}, we arrive at
\begin{align*}
	D^j_1+D^j_2=&-\rho\bigg(\frac{1}{\|\y-\y_1\|_{\H}}+\frac{1}{\|\y^j-\y_1\|_{\H}}\bigg) \bigg[ \left(\frac{\y^j-\y_1}{\|\y^j-\y_1\|_{\H}}-\frac{\y-\y_1}{\|\y-\y_1\|_{\H}},\y^j-\y\right)\bigg]\nonumber\\&\quad
	-\rho\left(\frac{\y^j-\y_1}{\|\y^j-\y_1\|_{\H}}-\frac{\y-\y_1}{\|\y-\y_1\|_{\H}},\frac{\y-\y_1}{\|\y^j-\y_1\|_{\H}}-\frac{\y^j-\y_1}{\|\y-\y_1\|_{\H}}\right).
\end{align*}
Thus, from \eqref{d1d2}, we have
\begin{align*}
-\rho I_{\e}^j-\rho J_{\e}^j&=\rho^2\bigg(\frac{1}{\|\y-\y_1\|_{\H}}+\frac{1}{\|\y^j-\y_1\|_{\H}}\bigg) \bigg[\left(\frac{\y^j-\y_1}{\|\y^j-\y_1\|_{\H}}-\frac{\y-\y_1}{\|\y-\y_1\|_{\H}},\y^j-\y\right)\bigg]\nonumber\\&\quad
+\rho^2\left(\frac{\y^j-\y_1}{\|\y^j-\y_1\|_{\H}}-\frac{\y-\y_1}{\|\y-\y_1\|_{\H}},\frac{\y-\y_1}{\|\y^j-\y_1\|_{\H}}-\frac{\y^j-\y_1}{\|\y-\y_1\|_{\H}}\right)
%\nonumber\\&=\rho^2\left(\frac{\y^j-\y_1}{\|\y^j-\y_1\|_{\H}}-\frac{\y-\y_1}{\|\y-\y_1\|_{\H}},\frac{\y^j-\y}{\|\y-\y_1\|_{\H}}+\frac{\y^j-\y}{\|\y^j-\y_1\|_{\H}}\right)\nonumber\\&\quad+\rho^2\left(\frac{\y^j-\y_1}{\|\y^j-\y_1\|_{\H}}-\frac{\y-\y_1}{\|\y-\y_1\|_{\H}},\frac{\y-\y_1}{\|\y^j-\y_1\|_{\H}}-\frac{\y^j-\y_1}{\|\y-\y_1\|_{\H}}\right)
\nonumber\\&=
\rho^2\left(\frac{\y^j-\y_1}{\|\y^j-\y_1\|_{\H}}-\frac{\y-\y_1}{\|\y-\y_1\|_{\H}},\frac{\y^j-\y_1}{\|\y^j-\y_1\|_{\H}}-\frac{\y-\y_1}{\|\y-\y_1\|_{\H}}\right)\nonumber\\&=
2\rho^2\bigg(1-\left(\frac{\y^j-\y_1}{\|\y^j-\y_1\|_{\H}},\frac{\y-\y_1}{\|\y-\y_1\|_{\H}}\right)\bigg).
\end{align*}
Integrating the above expression over $[0,T]$ and then taking limit $\varepsilon\to 0$ in the above expression, we find 
  \begin{align*}
-\rho\lim\limits_{\e\to0} \int_0^T( I_{\e}^j(t)+J_{\e}^j(t))\d t=& 2\rho^2\int_0^T\bigg(1-\left(\frac{\y^j(t)-\y_1}{\|\y^j(t)-\y_1\|_{\H}},\frac{\y(t)-\y_1}{\|\y(t)-\y_1\|_{\H}}\right)\bigg)\d t.
  \end{align*}
    Using \eqref{sigapp} and  making the use of the strong convergence given in \eqref{315}, we have 
    \begin{align*}
    	\lim\limits_{j\to\infty}\int_0^T \|\v^j(t)-\v(t)\|_{\H}^2\d t= 0. 
    	\end{align*}
 Thus in both cases, we infer that $\v^j\to\v$ in $\mathrm{L}^2(0,T;\H)$ as $j\to\infty$. Finally, by the uniqueness of the weak limit, we conclude that $\xi=\v$.
  
  We have shown that (see Lemma \ref{exact}) the solution $\y^j(\cdot):=\y^j(\cdot,\y_0^j)$ of the system 
		\begin{equation}\label{pf1.5}
			\left\{
			\begin{aligned}
				\frac{\d \y^j(t)}{\d t}+\mu\A\y^j(t)+\alpha\y^j(t)+ \mathcal{B}(\y^j(t))+\beta\mathcal{C}_1(\y^j(t))+ \gamma\mathcal{C}_2(\y^j(t))&= \v^j(t), \ \text{a.e.} \ t\in[0,T], \\
				\y^j(0)&=\y_0^j,
			\end{aligned}
			\right.
		\end{equation}
		with $\y_0^j\in\D(\A)$ and $\y_1\in\D(\A)$ satisfies 
		\begin{align}\label{317}
			\y^j(T,\y_0^j)-\y_1=0. 
		\end{align} 
		Let $\y(\cdot,\y_0)$ denote the solution of \eqref{pf1.4} corresponding to the initial data $\y_0\in\H$ and forcing $\v\in\mathrm{L}^{\infty}(0,T;\H)$. Then, by using \eqref{315} and \eqref{317}, we have 
		\begin{align*}
			\|\y(T,\y_0)-\y_1\|_{\H}\leq\|\y(T,\y_0)-\y^j(T,\y_0^j)\|_{\H}+\|\y^j(T,\y_0^j)-\y_1\|_{\H}\to 0,
		\end{align*}
		as $j\to\infty$, 	which implies that 
		\begin{align}\label{finite}
			\y(T,\y_0)=\y_1 \  \text{ in } \   \H. 
		\end{align}
		Thus, by the continuity of the map $t\mapsto\y^j(t)=\y^j(t,\y_0^j)$ (from the strong convergence \eqref{315}), we are able to extend the exact controllability result to all $\y_0\in\H$. 
		\vskip 0.1 cm
		\noindent
		\textbf{Step 2:} \emph{$\y_0\in\H$ and $\y_1\in\H$.} Let us now suppose that $\y_0\in\H$ and $\y_1\in\H$.
		Further, for $\y_1\in\H$, there exists $\y_1^{j}\in\D(\A)$ such that $\y_1^j\to\y_1$ in $\H$ as $j\to\infty$.	
		Now, as $\y_1^j\in\D(\A)$, then for a given $\y_0\in\H,$ there exists a control $\v$ such that $\y(T)=\y_1^j$ in $\H$, where $\y(\cdot)$ is the weak solution of \eqref{sgn2}. Then, we conclude that
		\begin{align*}
			\|\y(T)-\y_1\|_{\H}\leq\|\y(T)-\y_1^j\|_{\H}+\|\y_1^j-\y_1\|_{\H}\to 0\ \text{ as }\ j\to\infty,
		\end{align*}
		which implies that $\y(T)=\y_1$ in $\H$ for $\y_1\in\H$. Finally, we are able to extend the exact controllability result to all $\y_1\in\H$.
	\end{proof}

\begin{lemma}\label{dense}
	The set $\{\mathrm{B}\u: \u\in\mathrm{C}([0,T];\U)\}$ is dense in $\mathrm{L}^2(0,T;\H)$.
\end{lemma}
\begin{proof}
	See \cite[Lemma 2.4]{vbd}. 
\end{proof}

\begin{proof}[Proof of Proposition \ref{prop31}]\label{apcon1}
	Since the operator  $\Gamma(\cdot):=\mu\A+\alpha\I+\mathcal{B}(\cdot)+\beta\mathcal{C}_1(\cdot)+\gamma\mathcal{C}_2(\cdot)$ is a quasi $m$-accretive operator in $\H$, by Lemma \ref{lem3.3}, there exists a control $\v\in\mathrm{L}^{\infty}(0,T;\H)\subset\mathrm{L}^{2}(0,T;\H)$ such that $\y(T)=\y_1$ and 
\begin{equation}\label{pf1.1111}
	\left\{
	\begin{aligned}
		\frac{\d \y(t)}{\d t} +\Gamma(\y(t))&= \v(t), \text{ for a.e. } \ t\in[0,T],\\ 
		\y(0)&=\y_0\in\H.
	\end{aligned}
	\right.
\end{equation}	
By Lemma \ref{dense}, for any $\varepsilon>0$, there exists $\u_{\e}\in\C([0,T];\U)$ such that $$\|\mathrm{B}\u_{\e}-\v\|_{\mathrm{L}^2(0,T;\H)}\leq \e.$$  Let $\y_\e\in\mathrm{C}([0,T;\H])\cap\mathrm{L}^2(0,T;\V)\cap\mathrm{L}^{r+1}(0,T;\wi\L^{r+1})$ be the solution of 
\begin{equation}\label{pf1.11}
	\left\{
	\begin{aligned}
		\frac{\d \y_\e(t)}{\d t} +\Gamma(\y_\e(t))&= \mathrm{B}\u_\e(t), \text{ for a.e. } \ t\in[0,T],\\ 
		\y_\e(0)&=\y_0\in\H.
	\end{aligned}
	\right.
\end{equation}	
Subtracting \eqref{pf1.11} from \eqref{pf1}, we obtain  for a.e. $t\in[0,T]$ in $\V'$, 
\begin{align*}
	\frac{\d}{\d t}(\y_\e(t)-\y(t))+\Gamma(\y_\e(t))-\Gamma(\y(t))=\mathrm{B}\u_\e(t)-\v(t). 
\end{align*}
Taking the inner product with $\y_\e(\cdot)-\y(\cdot)$ and using the $m$-accretivity of $\Gamma(\cdot)$, we deduce 
\begin{align*}
	\frac{1}{2}\frac{\d}{\d t}\|\y_\e(t)-\y(t)\|_{\H}^2\leq \kappa\|\y_\e(t)-\y(t)\|_{\H}^2+(\mathrm{B}\u_\e(t)-\v(t),\y_\e(t)-\y(t)),
\end{align*}
for a.e. $t\in[0,T]$ and it implies that 
\begin{align*}
	\frac{\d}{\d t}\|\y_\e(t)-\y(t)\|_{\H}\leq\kappa\|\y_\e(t)-\y(t)\|_{\H}+\|\mathrm{B}\u_\e(t)-\v(t)\|_{\H}.
\end{align*}
An application of Gronwall's inequality yields 
\begin{align*}
	\|\y_\e(T)-\y(T)\|_{\H}\leq\int_0^T e^{\kappa(T-t)}\|\mathrm{B}\u_\e(t)-\v(t)\|_{\H}\d t\leq\frac{e^{2\kappa T}}{2\kappa}\sqrt{\e}. 
\end{align*}
The above estimate implies 
\begin{align*}
	\|\y_\e(T)-\y_1\|_{\H}\leq\|\y_\e(T)-\y(T)\|_{\H}+\|\y(T)-\y_1\|_{\H}\leq\frac{e^{2\kappa T}}{2\kappa}\sqrt{\e}, 
\end{align*}
which completes the proof of approximate controllability of the system \eqref{pf1}. 
\end{proof}

\section{Stochastic Convective Brinkmann-Forchheiemer Extended Darcy Equations}\label{stochastic}\setcounter{equation}{0}
Let $(\Omega,\mathscr{F},\P)$ be a complete probability space equipped with an increasing family of sub-sigma fields $\{\mathscr{F}_t\}_{t\geq0}$ of $\mathscr{F}$ satisfying the usual conditions (that is, a normal filtration). The noise term on the  stochastic basis $(\Omega,\mathscr{F},\{\mathscr{F}_t\}_{t\geq0},\P)$ is described by  a cylindrical Wiener process $\{\W(t)\}_{t\geq 0}$ on $\H$,  and a covariance operator $\Q$. Let $\mathcal{L}(\H,\H)$ be the space of all bounded linear operators on $\H$. Let the covariance operator  $\mathrm{Q}\in\mathcal{L}(\H,\H)$ be  such that $\mathrm{Q}$ is positive, symmetric and trace class operator with ker $\mathrm{Q}=\{\boldsymbol{0}\}$. We assume that there exists a complete orthonormal system $\{\boldsymbol{e}_k\}_{k\in\mathbb{N}}$ in $\mathbb{H}$ of the covariance operator $\Q$ and a bounded sequence $\{\mu_k\}_{k\in\mathbb{N}}$ of positive real numbers such that $\Q \boldsymbol{e}_k=\mu_k \boldsymbol{e}_k,\ k\in\N$. Here $\mu_k$ is an eigenvalue corresponding to the eigenfunction $\boldsymbol{e}_k$ such that following holds:
\begin{align*}
	\Tr\Q=\sum\limits_{k=1}^{\infty} \mu_k<\infty \ \text{ and }\  \sqrt{\Q}\y=\sum\limits_{k=1}^{\infty}\sqrt{\mu_k}(\y,\boldsymbol{e}_k)\boldsymbol{e}_k,  \ \text{ for }\  \y\in\H.
\end{align*}
We know that the stochastic process $\{\W(t)\}_{t\geq0}$ is an $\H$-valued cylindrical Wiener process with respect to a filtration $\{\mathscr{F}_t\}_{t\geq0}$ if and only if for any $t\geq0,$ the process $\{\W(t)\}_{t\geq 0}$ can be expressed as $\W(t)=\sum\limits_{k=1}^{\infty} \beta_k(t)\boldsymbol{e}_k$ (cf. \cite{gdp}), where $\{\beta_k(\cdot), \ k\in\N\}$, is a family of real-valued independent Brownian motions on the probability space $\left(\Omega,\mathscr{F},\{\mathscr{F}_t\}_{t\geq 0},\P\right)$. In this section, we consider the following stochastic CBFeD equations subjected to an external deterministic force $\f$, driven by an additive Gaussian noise in $\V'+\wi\L^{\frac{r+1}{r}}$:
\begin{equation}\label{scbf}
	\left\{
	\begin{aligned}
		\d\X(t)+\{\mu\A\X(t)+\mathcal{B}(\X(t))+\alpha\X(t)+\beta\mathcal{C}_1(\X(t))+\gamma\mathcal{C}_2(\X(t))\}\d t&=\f\d t+ \sqrt{\mathrm{Q}}\d\W(t),\\ 
		\X(0)&=\x,
	\end{aligned}
\right.
\end{equation}
for a.e. $t\in[0,T]$, where $\x\in\H$, $\sqrt{\mathrm{Q}}\d\W$ is a colored noise defined on $(\Omega,\mathscr{F},\{\mathscr{F}_t\}_{t\geq 0},\P)$ with values in $\H$ such that $\sqrt{\Q}\W(t)=\sum\limits_{k=1}^{\infty}\sqrt{\mu_k}\boldsymbol{e}_k\beta_k(t)$. We shall assume further that \begin{align}\label{423}\sqrt{\Q}\in\mathcal{L}(\mathbb{U},\mathbb{V}),\end{align} where $\mathbb{U}$ is a Hilbert space, $\mathbb{H}\subset\mathbb{U}$ and the injection of $\H$ into $\mathbb{U}$ is Hilbert-Schmidt, that is, there exits a one-to-one map $\J:\H\to\U$ with $\y\mapsto\J\y=\y$ such that the Hilbert-Schmidt norm $\|\J\|_{\mathrm{L}_2(\H,\U)}<\infty$. Note that Hilbert-Schmidt operators are compact, so one cannot expect the Hilbert-Schmidt operators from the same Hilbert space to themselves in an infinite-dimensional setup.

\begin{example}\label{exmp4.1}
	For $\e>\frac{d}{2}+1$, one can take $\Q=(\I+\A)^{-\e},$ $\{\mu_k\}_{k\in\N}=\{(1+\lambda_k)^{-\e}\}_{k\in\N}$ and $\mathbb{U}=\V_{-(\e-1)}=\D((\I+\A)^{-\left(\frac{\e-1}{2}\right)})$. Note that the asymptotic growth of  $\lambda_k$ are given by $\lambda_k\sim(k)^{\frac{2}{d}}$, for $k=1,2,\ldots$ (cf. \cite{fmrt}). Then, $\sqrt{\Q}\in\mathcal{L}(\U,\V)$. Indeed for any $\y\in\U$, we find
	\begin{align*}
		\|\sqrt{\Q}\y\|_{\V}^2= \|(\I+\A)^{\frac{1}{2}}\sqrt{\Q}\y\|_{\H}^2= \|(\I+\A)^{\frac{1}{2}-\frac{\e}{2}}\y\|_{\H}^2=\|\y\|_{\D((\I+\A)^{-\left(\frac{\e-1}{2}\right)})}^2<\infty.
	\end{align*} 
Note that, since $\V\hookrightarrow\H$, therefore $\sqrt{\Q}\in\mathcal{L}(\U,\H)$ also. Moreover, $\Q$ is a trace class operator, since 
	\begin{align*}
		\Tr\Q&=\sum\limits_{k=1}^{\infty} (\Q \boldsymbol{e}_k,\boldsymbol{e}_k)=\sum\limits_{k=1}^{\infty} ((\I+\A)^{-\e} \boldsymbol{e}_k,\boldsymbol{e}_k)=\sum\limits_{k=1}^{\infty} (1+\lambda_k)^{-\e}\sim\sum\limits_{k=1}^{\infty} \frac{1}{(1+k^{\frac{2}{d}})^{\e}}\nonumber\\&\leq\sum\limits_{k=1}^{\infty} \frac{1}{k^{\frac{2\e}{d}}}<\infty,
	\end{align*}
provided $\frac{2\e}{d}>1$ or $\e>\frac{d}{2}$. Furthermore, for $\e>\frac{d}{2}+1$, the embedding $\H\hookrightarrow\V_{-(\e-1)}$ is Hilbert-Schmidt, that is, the map $\mathrm{J}:\V_{-(\e-1)}\to\H$ is a Hilbert-Schmidt operator. Indeed
\begin{align*}
	\|\mathrm{J}\|_{\mathrm{L}_2(\H,\U)}^2=\sum\limits_{k=1}^{\infty} \|\mathrm{J}\boldsymbol{e}_k\|_{\U}^2=\sum\limits_{k=1}^{\infty} \|\boldsymbol{e}_k\|_{\U}^2=\sum\limits_{k=1}^{\infty} ((\I+\A)^{-\e+1} \boldsymbol{e}_k,\boldsymbol{e}_k)\leq\sum\limits_{k=1}^{\infty}\frac{1}{k^{\frac{2}{d}(\e-1)}} <\infty,
\end{align*}
provided $\e>\frac{d}{2}+1$. Let us define $\wi\Q:=\A\Q$ and we calculate  
\begin{align*}
	\Tr\wi\Q=\sum\limits_{k=1}^{\infty} (\wi\Q\boldsymbol{e}_k,\boldsymbol{e}_k)_{\H}&= \sum\limits_{k=1}^{\infty} (\A(\I+\A)^{-\e} \boldsymbol{e}_k,\boldsymbol{e}_k)= \sum\limits_{k=1}^{\infty} ((1+\lambda_k)^{-\e}\A \boldsymbol{e}_k,\boldsymbol{e}_k)\nonumber\\&=\sum\limits_{k=1}^{\infty} \lambda_k(1+\lambda_k)^{-\e}\leq\sum\limits_{k=1}^{\infty} \lambda_k^{-\e+1}\sim\sum\limits_{k=1}^{\infty} k^{-\frac{2}{d}(\e-1)}<\infty,
\end{align*}
provided $\frac{2}{d}(\e-1)>1$ or $\e>\frac{d}{2}+1$. Thus the covariance operator $\Q=(\I+\A)^{-\e}$ has the following two properties:
\begin{equation*}
	\left\{
	\begin{aligned}
		\Tr\Q&<\infty, \ \text{ for } \e>\frac{d}{2}, \\
		\Tr(\A\Q)&<\infty, \ \text{ for } \e>\frac{d}{2}+1.
	\end{aligned}
	\right.
\end{equation*}	
\end{example}

	\begin{definition}[Global strong solution]\label{def4.2}
		An $\H$-valued $(\mathscr{F}_t)_{t\geq 0}$-adapted stochastic process $\X(\cdot)$ is called a strong solution to the system \eqref{scbf} if the following conditions are satisfied:
		\begin{enumerate}
			\item [(i)] the process $\X\in\mathrm{L}^2(\Omega;\mathrm{L}^{\infty}(0,T;\H)\cap\mathrm{L}^2(0,T;\V))\cap\mathrm{L}^{r+1}(\Omega;\mathrm{L}^{r+1}(0,T;\widetilde{\L}^{r+1}))$ and $\X(\cdot)$ has a $\V\cap\widetilde{\L}^{r+1}$-valued  modification, which is progressively measurable with continuous paths in $\H$ and $\X\in\C([0,T];\H)\cap\mathrm{L}^2(0,T;\V)\cap\mathrm{L}^{r+1}(0,T;\widetilde{\L}^{r+1})$, $\mathbb{P}$-a.s.,
			\item [(ii)] the following equality holds for every $t\in [0, T ]$, as an element of $\V'+\wi\L^{\frac{r+1}{r}},$ 
			\begin{align*}
				\X(t)&=\x-\int_0^t\left[\mu \A\X(s)+\alpha\X(s)+\B(\X(s))+\beta\mathcal{C}_1(\X(s))+ \gamma\mathcal{C}_1(\X(s))+\f(s)\right]\d s\nonumber\\&\quad+\int_0^t\sqrt{\Q}\d \W(s),\ \mbox{$\mathbb{P}$-a.s.}
			\end{align*}
		\end{enumerate}
	\end{definition}
An alternative version of the condition (ii) is to require that for any  $\v\in\V\cap\widetilde{\L}^{r+1}$:
\begin{align}\label{4.5}
	(\X(t),\v)&=(\x,\v)-\int_0^t\langle\mu\A\X(s)+\alpha\X(s)+\B(\X(s))+\beta\mathcal{C}_1(\X(s)) +\gamma\mathcal{C}_2(\X(s)),\v\rangle\d s\no\\&\quad+\int_0^t\left(\sqrt\Q\d \W(s),\v\right),\ \mathbb{P}\text{-a.s.},
\end{align}
for all $t\in[0,T]$.
\begin{definition}[Pathwise uniqueness]
		A strong solution $\X(\cdot)$ to \eqref{scbf} is called \emph{pathwise  unique} if $\widetilde{\X}(\cdot)$ is an another strong solution, then $$\mathbb{P}\Big\{\omega\in\Omega:\X(t)=\widetilde{\X}(t),\ \text{ for all }\ t\in[0,T]\Big\}=1.$$
\end{definition}
The following result on the existence and uniqueness of strong solutions for the system \eqref{scbf} is proved in \cite{MT4}.

\begin{theorem}\label{regularity}
	Let $\x\in\H$ and $\f\in\mathrm{L}^2(0,T;\V')$ be given and $\Tr(\Q)<\infty$. Then there exists a \emph{unique  strong solution} to the system \eqref{scbf} in the sense of Definition \ref{def4.2} satisfying 
	\begin{align}\label{429}
		&\E\left[\sup_{t\in[0,T]}\|\X(t)\|_{\H}^2+2\mu \int_0^T\|\nabla\X(t)\|_{\H}^2\d t+2\alpha \int_0^T\|\X(t)\|_{\H}^2\d t+
		2\beta\int_0^T\|\X(t)\|_{\widetilde{\L}^{r+1}}^{r+1}\d t\right]\nonumber\\&\leq 2\left[\|\x\|_{\H}^2+\left(\frac{1}{\alpha}+\frac{1}{\mu}\right)\int_0^T\|\f(t)\|_{\V'}^2\d t+\varkappa(2|\gamma|)^{\frac{r+1}{r-q}}|\mathbb{T}^d|T+7\Tr(\Q)T\right],
	\end{align}
	where $\varkappa=\left(\frac{\beta(r+1)}{q+1}\right)^{\frac{q+1}{r-q}}\left(\frac{r-q}{r+1}\right)$. In addition, the following It\^o formula holds true: 
	\begin{align*}
		&\|\X(t)\|_{\H}^2+2\alpha\int_0^t \|\X(s)\|_{\H}^2\d s+2\mu \int_0^t\|\X(s)\|_{\V}^2\d s+ 
		2\beta \int_0^t\|\X(s)\|_{\widetilde{\L}^{r+1}}^{r+1}\d s\nonumber\\&= \|\x\|_{\H}^2-2\gamma\int_0^t\|\X(s)\|_{\widetilde{\L}^{q+1}}^{q+1}\d s+2\int_0^t \langle\f(s),\X(s)\rangle\d s+\text{\emph{Tr}}(\Q)t+2\int_0^t (\sqrt{\Q}\d\W(s),\X(s)),
	\end{align*}
	for all $t\in[0,T]$, $\mathbb{P}$-a.s.

Moreover, if $\x\in\V$, $\f\in\mathrm{L}^2(0,T;\H)$ and $\Tr(\A\Q)<\infty$, then the unique strong solution $\X$ to the system \eqref{scbf} belongs to  $\mathrm{L}^2(\Omega;\mathrm{L}^{\infty}(0,T;\V)\cap\mathrm{L}^2(0,T;\D(\A)))\cap\mathrm{L}^{r+1}(\Omega;\mathrm{L}^{r+1}(0,T;\widetilde{\L}^{3(r+1)}))$ such that 
\begin{align}\label{430}
	&\E\left[\sup_{t\in[0,T]}\|\nabla\X(t)\|_{\H}^2+2\mu \int_0^T\|\A\X(t)\|_{\H}^2\d t+4\alpha\int_0^T \|\nabla\X(t)\|_{\H}^2\d t+2\beta \int_0^T\||\X(t)|^{\frac{r-1}{2}}\nabla\X(t)\|_{\H}^2\d t\right] \nonumber\\&\leq C\left\{\|\nabla\x\|_{\H}^2+\frac{1}{\mu}\int_0^T \|\f(t)\|_{\H}^2\d t +T\Tr(\A\Q)\right\}e^{CT},
\end{align}
and $\X$ has a modification with paths in $ \C([0,T];\V)\cap\mathrm{L}^2(0,T;\D(\A))\cap\mathrm{L}^{r+1}(0,T;\widetilde{\L}^{3(r+1)})$, $\mathbb{P}$-a.s.
	\end{theorem}

	\begin{proof}
	The existence and uniqueness of a strong solution to the system \eqref{scbf} is available in \cite{MT4}. Therefore, we are providing here proofs  of the energy estimates \eqref{429} and \eqref{430} only. 	On applying the infinite-dimensional It\^o formula to the function $\|\cdot\|_{\H}^2$ and to the process $\X(\cdot)$, we get
		\begin{align*}
			&\|\X(t)\|_{\H}^2+2\mu \int_0^t\|\nabla\X(s)\|_{\H}^2\d s+2\alpha \int_0^t\|\X(s)\|_{\H}^2\d s+ 2\beta\int_0^t\|\X(s)\|_{\widetilde{\L}^{r+1}}^{r+1}\d s\nonumber\\&=\|\x\|_{\H}^2+2\int_0^t \langle\f(s),\X(s)\rangle\d s-2\gamma\int_0^t\|\X(s)\|_{\widetilde{\L}^{q+1}}^{q+1}\d s \nonumber\\&\quad+\int_0^{t}\Tr(\Q)\d s+2\int_0^{t}\left(\sqrt{\Q}\d\W(s),\X(s)\right) \nonumber\\&\leq\|\x\|_{\H}^2+2\int_0^t \langle\f(s),\X(s)\rangle\d s+ 2|\gamma||\mathbb{T}^d|^{\frac{r-q}{r+1}}\int_0^t\|\X(t)\|_{\widetilde{\L}^{r+1}}^{q+1}\d s \nonumber\\&\quad+ \int_0^{t}\Tr(\Q)\d s+2\int_0^{t}\left(\sqrt{\Q}\d\W(s),\X(s)\right) \nonumber\\&\leq\|\x\|_{\H}^2+\left(\frac{1}{\alpha}+\frac{1}{\mu}\right)\int_0^t \|\f(s)\|_{\V'}^2+\alpha\int_0^t \|\X(s)\|_{\H}^2\d s+\mu\int_0^t \|\nabla\X(s)\|_{\H}^2\d s+ \varkappa(2|\gamma|)^{\frac{r+1}{r-q}}|\mathbb{T}^d|t \nonumber\\&\quad +\beta\int_0^t\|\X(t)\|_{\widetilde{\L}^{r+1}}^{r+1}\d s+ \int_0^{t}\Tr(\Q)\d s+2\int_0^{t}\left(\sqrt{\Q}\d\W(s),\X(s)\right), \ \P\text{-a.s.},
		\end{align*}
for all $t\in[0,T]$ and where $\varkappa=\left(\frac{\beta(r+1)}{q+1}\right)^{\frac{q+1}{r-q}}\left(\frac{r-q}{r+1}\right)$. 
 This implies that
 \begin{align*}
&\|\X(t)\|_{\H}^2+\mu \int_0^t\|\nabla\X(s)\|_{\H}^2\d s+\alpha \int_0^t\|\X(s)\|_{\H}^2\d s+ \beta\int_0^t\|\X(s)\|_{\widetilde{\L}^{r+1}}^{r+1}\d s\nonumber\\&\leq	\|\x\|_{\H}^2+\left(\frac{1}{\alpha}+\frac{1}{\mu}\right)\int_0^t \|\f(s)\|_{\V'}^2+ \varkappa(2|\gamma|)^{\frac{r+1}{r-q}}|\mathbb{T}^d|t +\Tr(\Q)t+ 2\int_0^{t}\left(\sqrt{\Q}\d\W(s),\X(s)\right).
 \end{align*}
On taking supremum over $t\in[0,T]$ and then expectation, we obtain
		\begin{align}\label{vreg0}
			&\E\left[\sup\limits_{t\in[0,T]}\|\X(t)\|_{\H}^2+\mu \int_0^T\|\nabla\X(s)\|_{\H}^2\d s+ \alpha \int_0^T\|\X(s)\|_{\H}^2\d s+\beta\int_0^T\|\X(s)\|_{\widetilde{\L}^{r+1}}^{r+1}\d s\right] \nonumber\\&\leq\E\left[\|\x\|_{\H}^2\right]+\left(\frac{1}{\alpha}+\frac{1}{\mu}\right)\int_0^T \|\f(s)\|_{\V'}^2\d s+ \varkappa(2|\gamma|)^{\frac{r+1}{r-q}}|\mathbb{T}^d|T +\Tr(\Q)T\nonumber\\&\quad+2\E\left[\sup\limits_{t\in[0,T]} \left|\int_0^{t}\left(\sqrt{\Q}\d\W(s),\X(s)\right)\right|\right].
		\end{align} 
We estimate the last term in \eqref{vreg0} by using the Burkholder-Davis-Gundy inequality (cf. \cite{bd,IA}) and Young's inequality as
\begin{align}\label{vreg01}
\E\left[\sup\limits_{t\in[0,T]}\left|\int_0^{t}\left(\sqrt{\Q}\d\W(s),\X(s)\right)\right|\right]
&\leq\sqrt{3}\E\left[\int_0^{T}\|\sqrt{\Q}\|_{\mathrm{L}_2(\U,\H)}^2\|\X(t)\|_{\H}^2\d t \right]^{\frac{1}{2}}\nonumber\\&=\sqrt{3}\E\left[\int_0^{T}\Tr(\Q)\|\X(s)\|_{\H}^2\d t \right]^{\frac{1}{2}}\nonumber\\&\leq\sqrt{3}\E\left[\left(\sup\limits_{t\in[0,T]}\|\X(t)\|_{\H}^2\right)^{\frac{1}{2}}\left(\int_0^{T}\Tr(\Q)\d t\right)^{\frac{1}{2}}\right]\nonumber\\&\leq\frac{1}{4}\E\left[\sup\limits_{t\in[0,T]}\|\X(t)\|_{\H}^2\right]+3\Tr(\Q)T.
\end{align}
On substituting \eqref{vreg01} in \eqref{vreg0}, we obtain
\begin{align*}
	&\E\left[\frac{1}{2}\sup\limits_{t\in[0,T]}\|\X(t)\|_{\H}^2+\mu \int_0^T\|\nabla\X(s)\|_{\H}^2\d s+ \alpha \int_0^T\|\X(s)\|_{\H}^2\d s+\beta\int_0^T\|\X(t)\|_{\widetilde{\L}^{r+1}}^{r+1}\d s\right] \nonumber\\&\leq\|\x\|_{\H}^2+\left(\frac{1}{\alpha}+\frac{1}{\mu}\right)\int_0^T \|\f(s)\|_{\V'}^2\d s+ \varkappa(2|\gamma|)^{\frac{r+1}{r-q}}|\mathbb{T}^d|T +7\Tr(\Q)T,
\end{align*}
and \eqref{429} follows.

Let us operate by $\A^{\frac{1}{2}}$ in \eqref{scbf} to obtain the following stochastic differential satisfied by the stochastic process $\A^{\frac{1}{2}}\X$:
	\begin{equation*}
		\left\{
		\begin{aligned}
			&	\d\A^{\frac{1}{2}}\X(t)+\A^{\frac{1}{2}}\{\mu\A\X(t)+\mathcal{B}(\X(t))+\alpha\X(t)+\beta\mathcal{C}_1(\X(t))+\gamma\mathcal{C}_2(\X(t))\}\d t\\&\hspace{1.1cm}=\A^{\frac{1}{2}}\f\d t+ \A^{\frac{1}{2}}\sqrt{\mathrm{Q}}\d\W(t),
		\end{aligned}
		\right.
	\end{equation*}
	for a.e. $t\in[0,T]$. On applying the infinite-dimensional It\^o formula to the function $\|\cdot\|_{\H}^2$ and to the process  $\A^{\frac{1}{2}}\X$, we get
	\begin{align}\label{vreg1}
		&	\|\nabla\X(t)\|_{\H}^2+2\mu \int_0^t\|\A\X(s)\|_{\H}^2\d s+2\alpha \int_0^t\|\nabla\X(s)\|_{\H}^2\d s+2\beta\int_0^t(\mathcal{C}_1(\X(s)),\A\X(s))\d s\nonumber\\&= \|\nabla\x\|_{\H}^2+2\int_0^t (\f(s),\A\X(s))\d s-2\int_0^t(\B(\X(s)),\A\X(s))\d s- 2\gamma\int_0^t(\mathcal{C}_2(\X(s)),\A\X(s))\d s \nonumber\\&\quad+\int_0^{t}\Tr(\A\Q)\d s + 2\int_0^{t}\left(\A^{1/2}\sqrt{\Q}\d\W(s),\A^{1/2}\X(s)\right), \  \P\text{-a.s.,}
	\end{align}
for all $t\in[0,T]$. We consider the following two cases:
		\vskip 2mm
	\noindent
	\textbf{Case-I} \emph{(For $r>3$).} By using H\"older's and Young's inequalities, we estimate $|(\B(\X),\A\X)|$ as 
	\begin{align}\label{vreg1.0}
		|(\B(\X),\A\X)|\leq\frac{\mu}{4}\|\A\X\|_{\H}^2+\frac{\beta}{4} \||\X|^{\frac{r-1}{2}}\nabla\X\|_{\H}^2+\eta_4\|\nabla\X\|_{\H}^2,
	\end{align}
	where $\eta_4:=\frac{r-3}{2\mu(r-1)}\left[\frac{8}{\beta\mu(r-1)}\right]^{\frac{2}{r-3}}.$ Using  H\"older's inequality, we estimate
	\begin{align}\label{vreg2}
		\int_{\mathbb{T}^d}|\X(x)|^{q-1}|\nabla\X(x)|^2\d x\leq
		\left(\int_{\mathbb{T}^d}|\X(x)|^{r-1}|\nabla\X(x)|^2\d x\right)^\frac{q-1}{r-1}
		\left(\int_{\mathbb{T}^d}|\nabla\X(x)|^2 \d x\right)^\frac{r-q}{r-1}.
	\end{align}
	Using \eqref{vreg2}, the identity $\nabla|\y|^k=k \sum\limits_{j=1}^dy_j\nabla y_j|\y|^{k-2}$, and H\"older's and Young's inequalities, we calculate 
	\begin{align}\label{vreg3}
		|\gamma(\mathcal{C}_2(\X),\mathrm{A}\X)|&\leq
		|\gamma|\||\X|^{\frac{q-1}{2}}\nabla\X\|_{\H}^{2} +4|\gamma|\left[\frac{q-1}{(q+1)^2}\right]\|\nabla|\X|^{\frac{q+1}{2}}\|_{\H}^{2}
		\nonumber\\&\leq q|\gamma|\||\X|^{\frac{r-1}{2}}\nabla(\X)\|_{\H}^{\frac{2(q-1)}{r-1}}\|\nabla\X\|_{\H}^{\frac{2(r-q)}{r-1}}\nonumber\\&\leq %C(q)|\gamma|\||\X|^{\frac{r-1}{2}}\nabla\X\|_{\H}^{\frac{2(q-1)}{r-1}}\|\X\|_{\V}^{\frac{2(r-q)}{r-1}}\nonumber\\&\leq
		\frac{\beta}{4}\||\X|^{\frac{r-1}{2}}\nabla(\X)\|_{\H}^2+\eta_5\|\nabla\X\|_{\H}^2,
	\end{align}
	where $\eta_5:=(q|\gamma|)^{\frac{r-1}{r-q}}\left[\frac{4}{\beta}\left(\frac{q-1}{r-1}\right)\right]^ {\frac{q-1}{r-q}}\left({\frac{r-q}{r-1}}\right).$ Combining \eqref{vreg1}  with \eqref{vreg1.0}-\eqref{vreg3} yield
	\begin{align*}
		&	\|\nabla\X(t)\|_{\H}^2+\mu\int_0^t\|\A\X(s)\|_{\H}^2\d s+2\alpha \int_0^t\|\nabla\X(s)\|_{\H}^2\d s+\beta\int_0^t\||\X(s)|^{\frac{r-1}{2}}\nabla\X(s)\|_{\H}^2\d s \nonumber\\&\leq \|\nabla\x\|_{\H}^2+\frac{1}{\mu}\int_0^t \|\f(s)\|_{\H}^2\d s+ (\eta_4+\eta_5)\int_0^t\|\nabla\X(s)\|_{\H}^2\d s +\int_0^{t}\Tr(\A\Q)\d s \nonumber\\&\quad+ 2\int_0^{t}\left(\A^{1/2}\sqrt{\Q}\d\W(s),\A^{1/2}\X(s)\right), \  \P\text{-a.s.},
	\end{align*}
for all $t\in[0,T]$. On taking supremum over time $t\in[0,T]$ followed by an expectation, we deduce
\begin{align}\label{vreg4}
	&	\E\left[\sup\limits_{t\in[0,T]}\|\nabla\X(t)\|_{\H}^2+\mu \int_0^T\|\A\X(s)\|_{\H}^2\d s+2\alpha \int_0^T\|\nabla\X(s)\|_{\H}^2\d s+\beta\int_0^T\||\X(s)|^{\frac{r-1}{2}}\nabla\X(s)\|_{\H}^2\d s\right]\nonumber\\&\leq\|\nabla\x\|_{\H}^2+\frac{1}{\mu}\int_0^t \|\f(s)\|_{\H}^2\d s+(\eta_4+\eta_5) \E\left[\int_0^T\|\nabla\X(s)\|_{\H}^2\d s \right]+\E\left[\int_0^{T}\Tr(\A\Q)\d s \right] \nonumber\\&\quad+2\E\left[\sup\limits_{t\in[0,T]}\left|\int_0^{t}\left(\sqrt{\A\Q}\d\W(s),\A^{1/2}\X(s)\right)\right|\right].
\end{align}
We calculate the last term in \eqref{vreg4} by making use of the Burkholder-Davis-Gundy inequality (cf. \cite{bd,IA}) and Young's inequality as
\begin{align}\label{vreg5} \E\left[\sup\limits_{t\in[0,T]}\left|\int_0^{t}\left(\sqrt{\A\Q}\d\W(s),\A^{1/2}\X(s)\right)\right|\right]
&\leq\sqrt{3}\E\left[\int_0^{T}\|\sqrt{\A\Q}\|_{\mathrm{L}_2(\U,\H)}^2\|\A^{1/2}\X(t)\|_{\H}^2\d t \right]^{\frac{1}{2}}\nonumber\\&=\sqrt{3}\E\left[\int_0^{T}\Tr(\A\Q)\|\nabla\X(t)\|_{\H}^2\d t \right]^{\frac{1}{2}}\nonumber\\&\leq\sqrt{3}\E\left[\left(\sup\limits_{t\in[0,T]}\|\nabla\X(t)\|_{\H}^2\right)^{\frac{1}{2}}\left(\int_0^{T}\Tr(\A\Q)\d t\right)^{\frac{1}{2}}\right] \nonumber\\&\leq\frac{1}{4}\E\left[\sup\limits_{t\in[0,T]}\|\nabla\X(t)\|_{\H}^2\right]+3\Tr(\A\Q)T.
\end{align}
On substituting \eqref{vreg5} in \eqref{vreg4}, we obatin
\begin{align}\label{vreg6}
	&\E\bigg[\frac{1}{2}\sup\limits_{t\in[0,T]}\|\nabla\X(t)\|_{\H}^2+\mu \int_0^T\|\A\X(s)\|_{\H}^2\d s +2\alpha \int_0^T\|\nabla\X(s)\|_{\H}^2\d s +\beta\int_0^T\||\X(s)|^{\frac{r-1}{2}}\nabla\X(s)\|_{\H}^2\d s \bigg]%\nonumber\\&\leq\E\left[\|\nabla\x\|_{\H}^2\right]+\frac{1}{\mu}\int_0^T \|\f(s)\|_{\H}^2\d s+(\eta_6+\eta_7) \E\left[\int_0^T\|\nabla\X(s)\|_{\H}^2\d s \right]+7\E\left[\int_0^{T}\Tr(\A\Q)\d s \right]  
	\nonumber\\&\leq\|\nabla\x\|_{\H}^2+\frac{1}{\mu}\int_0^T \|\f(s)\|_{\H}^2\d s +(\eta_4+\eta_5) \E\left[\int_0^T\|\nabla\X(s)\|_{\H}^2\d s \right]+7T\Tr(\A\Q).
\end{align}
	By the application of Gronwall inequality, we arrive at 
%	\begin{align}\label{vreg7}
%		 \E\left[\sup\limits_{t\in[0,T]}\|\nabla\X(t)\|_{\H}^2\right]\leq\left\{\E\left[2\|\nabla\x\|_{\H}^2\right]+\frac{2}{\mu}\int_0^T \|\f(s)\|_{\H}^2\d s +14T\Tr(\A\Q)\right\}e^{2(\eta_4+\eta_5)T}.
%	\end{align}
%From \eqref{vreg6} and \eqref{vreg7}, it implies that
\begin{align}\label{vreg8}
		&\E\left[\sup\limits_{t\in[0,T]}\|\nabla\X(t)\|_{\H}^2+2\mu \int_0^T\|\A\X(s)\|_{\H}^2\d s+ +4\alpha \int_0^T\|\nabla\X(s)\|_{\H}^2\d s+ 2\beta\int_0^T\||\X(s)|^{\frac{r-1}{2}}\nabla\X(s)\|_{\H}^2\d s \right]\nonumber\\&\leq 2\left\{\|\nabla\x\|_{\H}^2+\frac{1}{\mu}\int_0^T \|\f(s)\|_{\H}^2\d s +7T\Tr(\A\Q)\right\}e^{2(\eta_4+\eta_5)T}.
\end{align}

\vskip 2mm
\noindent
\textbf{Case-II} \emph{(For $d=r=3$ with $2\beta\mu>1$).} By using H\"older's and Young's inequalities, we calculate following:
\begin{align}
	|(\B(\X),\A\X)|&\leq\||\X|\nabla\X\|_{\H}\|\A\X\|_{\H}\leq\theta\mu\|\A\X\|_{\H}^2+\frac{1}{4\theta\mu}\||\X|\nabla\X\|_{\H}^2,\label{vreg9}\\
	|\gamma(\mathcal{C}_2(\X),\mathrm{A}\X)|&\leq\frac{1}{4\theta\mu} \||\X|\nabla\X\|_{\H}^2+\eta_6\|\nabla\X\|_{\H}^2,\label{vreg10.1}\\
|(\f,\A\X)|&\leq\frac{1}{2\mu(1-\theta)}\|\f\|_{\H}^2+\frac{\mu(1-\theta)}{2}\|\A\X\|_{\H}^2,\label{vreg10}
\end{align}
for some $0<\theta<1$ and where $\eta_6=\left(2\theta\mu|\gamma|q(q-1)\right)^{\frac{3-q}{q-1}}\left({\frac{3-q}{2}}\right)$. Moreover, for $r=3$, we write
\begin{align}\label{vreg11}
	(\mathcal{C}_1(\X),\mathrm{A}\X)&=\||\X|\nabla\X\|_{\H}^{2} +\frac{1}{2}\|\nabla|\X|^2\|_{\H}^{2}.
\end{align}
 Using \eqref{vreg9}-\eqref{vreg11} in \eqref{vreg1} and performing the similar calculations as we did for $r>3$, we finally obtain 
\begin{align}\label{vreg12}
	&\E\bigg[\sup\limits_{t\in[0,T]}\|\nabla\X(t)\|_{\H}^2+\mu(1-\theta)\int_0^T\|\A\X(s)\|_{\H}^2\d s+2\left(\beta-\frac{1}{2\theta\mu}\right) \int_0^T\||\X(s)|\nabla\X(s)\|_{\H}^{2}\d s\bigg]\nonumber\\&\leq 2\left\{\|\nabla\x\|_{\H}^2+\frac{1}{2\mu(1-\theta)}\int_0^T\|\f(s)\|_{\H}^2 \d s+7T\Tr(\A\Q)\right\}e^{2\eta_6T},
\end{align}
which completes the proof.
\end{proof}

%Let us take $\mathbb{U}=\D((\I+\A)^{-\left(\frac{\e-1}{2}\right)})$. 
% Further, note that the embedding $\H\hookrightarrow\U$ is Hilbert-Schmidt, that is, the map $\mathrm{J}:\U\to\H$ is a Hilbert-Schmidt operator, since
%\begin{align*}
%	\|\mathrm{J}\|_{\mathrm{L}_2(\H,\U)}^2=\sum\limits_{k=1}^{\infty} \|\mathrm{J}\boldsymbol{e}_k\|_{\U}^2=\sum\limits_{k=1}^{\infty} \|\boldsymbol{e}_k\|_{\U}^2=\sum\limits_{k=1}^{\infty} ((\I+\A)^{-\e+1} \boldsymbol{e}_k,\boldsymbol{e}_k)\leq\sum\limits_{k=1}^{\infty}\frac{1}{k^{\frac{2}{d}(\e-1)}} <\infty,
%\end{align*}
%for $\e>\frac{d}{2}+1$. 

We need the following results to establish our main theorem. The next result shows  $\mathbb{P}$-a.s. bound for the solutions of the system \eqref{scbf}. 
	\begin{proposition}\label{engest}
		For arbitarary $\x\in\H$, the unique solution $\X(\cdot,\cdot)$ to the  stochastic system \eqref{scbf}  satisfies 
		$	\X(\cdot,\omega)\in\mathrm{C}([0,T];\H)\cap\mathrm{L}^2(0,T;\V')\cap\mathrm{L}^{r+1}(0,T;\wi\L^{r+1}),$ $\P$-a.s.,  and 
		\begin{align}\label{est}
			&\sup\limits_{t\in[0,T]}\|\X(t)\|_{\H}^2+\int_0^T\|\nabla\X(t)\|_{\H}^2\d t+ \int_0^T\|\X(t)\|_{\H}^2\d t+ \int_0^T \|\X(t)\|_{\wi\L^{r+1}}^{r+1}\d t\nonumber\\&\leq C\left(\mu,\beta,\gamma,r,T,\Q,|\mathbb{T}^d|,  \|\f\|_{\mathrm{L}^2(0,T;\V')}\right), \ \P\text{-a.s. }
		\end{align}
	\end{proposition}
	\begin{proof}
		Let $\Y(\cdot)$ be the Ornstein-Uhlenback process satisfying
		\begin{equation}\label{oup}
			\left\{
			\begin{aligned}
				\d\Y+(\mu\A\Y(t)+\alpha\Y(t))\d t&=\sqrt{\Q}\d\W(t),\\
				\Y(0)&=\boldsymbol{0},
			\end{aligned}
			\right.
		\end{equation}
		for a.e. $t\in[0,T]$. The unique solution of \eqref{oup} can be represented as (see \cite[Chapter 5] {gdp})
		\begin{align}
			\Y(t)=\int_0^t\mathrm{S}(t-s)\sqrt{\Q}\d\W(s), 
		\end{align}
		where $\mathrm{S}$ is the analytic semigroup generated by the operator $\mu\A+\alpha\I$ (cf. \cite[Chapter 7]{ivb}) and we denote it as $\mathrm{S}(t):=e^{-(\mu\A+\alpha\I)t}$, for all $t\geq 0$.  Since $\mathrm{Tr}(\Q)=\sum\limits_{k=1}^{\infty}\mu_k<+\infty$ and $\sum\limits_{k=1}^{\infty}\frac{\mu_k}{(\mu\lambda_k+\alpha)^{1-\gamma}}\leq \frac{1}{\alpha^{1-\gamma}}\sum\limits_{k=1}^{\infty}\mu_k<\infty$, for any $\gamma\in(0,1)$, by \cite[Theorem 5.22]{gdp}, the process $\Y$ has a $\C(\mathbb{T}^d)$-valued version with continuous paths. Moreover, we have 
		\begin{align}\label{bddest}
			\sup\limits_{t\in[0,T]}\sup\limits_{x\in\mathbb{T}^d}|\Y(t,x)|\leq C(T,\Q)<\infty, \ \mathbb{P}\text{-a.s.} 
		\end{align}
		
		Let us set $\mathcal{Z}:=\X-\Y$, where $\X$ satisfies the stochastic CBFeD \eqref{scbf} and $\Y$ is the solution of the Ornstein-Uhlenback equation \eqref{oup}. On subtracting \eqref{oup} from \eqref{scbf}, we obtain the following system:
		\begin{equation}\label{scbf1}
			\left\{
			\begin{aligned}
				\frac{	\d\mathcal{Z}(t)}{\d t}+\mu\A\mathcal{Z}(t)+\mathcal{B}(\mathcal{Z}(t)+\Y(t))+\alpha\mathcal{Z}(t)+\beta\mathcal{C}_1
				(\mathcal{Z}(t)+\Y(t))+\gamma\mathcal{C}_2(\mathcal{Z}(t)+\Y(t))&=\f,\\  
				\mathcal{Z}(0)&=\x,
			\end{aligned}
			\right.
		\end{equation}
		which is a  deterministic system for each fixed $\omega\in\Omega$. Taking the inner product with $\mathcal{Z}(\cdot)$ in \eqref{scbf1}, we get
		\begin{align}\label{inn1}
			&\frac{1}{2}\frac{\d}{\d t}\|\mathcal{Z}(t)\|_{\H}^2+\mu\|\nabla\mathcal{Z}(t)\|_{\H}^2+ \alpha\|\mathcal{Z}(t)\|_{\H}^2+\beta\|\mathcal{Z}(t)+\Y(t)\|_{\wi\L^{r+1}}^{r+1}\nonumber\\& =\langle\f(t),\mathcal{Z}(t)\rangle+\langle\mathcal{B}(\mathcal{Z}(t)+\Y(t)),\Y(t)\rangle+\beta\langle\mathcal{C}_1(\mathcal{Z}(t)+\Y(t)),\Y(t)\rangle\nonumber\\&\quad+\gamma\langle\mathcal{C}_2(\mathcal{Z}(t)+\Y(t)),\Y(t)\rangle,
		\end{align}
		for a.e. $t\in[0,T]$.  For $d=2,3$ and $r\in[3,\infty)$, let us estimate $|\langle\mathcal{B}(\mathcal{Z}+\Y),\Y\rangle|$  by using H\"older's and Young's inequalities as
		%\begin{align}\label{cal1}
		%&|\langle\mathcal{B}(\mathcal{Z}(t)+\Y(t)),\Y(t)\rangle|\nonumber\\&=|\langle\mathcal{B}(\mathcal{Z}(t)+\Y(t),\mathcal{Z}(t)+\Y(t)),\Y(t)\rangle|
		%\nonumber\\&\leq|\langle\mathcal{B}(\mathcal{Z}(t),\mathcal{Z}(t)),\Y(t)\rangle|+\langle\mathcal{B}(\Y(t),\mathcal{Z}(t)),\Y(t)
		%\rangle|\nonumber\\&\leq\|\nabla\mathcal{Z}(t)\|_{\H}\|\mathcal{Z}(t)\Y(t)\|_{\H}+\|\nabla\mathcal{Z}(t)\|_{\H}\|\Y(t)\|_{\wi\L^4}^2
		%\nonumber\\&\leq\frac{\mu}{2}\|\nabla\mathcal{Z}(t)\|_{\H}^2+\frac{1}{\mu}\|\mathcal{Z}(t)\Y(t)\|_{\H}^2
		%+\frac{1}{\mu}\|\Y(t)\|_{\wi\L^4}^4\nonumber\\&\leq\frac{\mu}{2}\|\nabla\mathcal{Z}(t)\|_{\H}^2+\frac{1}{\mu}\|\Y(t)\|_{\wi\L^4}^4+\frac{1}{\mu}\|\mathcal{Z}(t)\|_{\wi\L^{r+1}}^2\|\Y(t)\|_{\wi\L^{\frac{2(r+1)}{r-1}}}^2\nonumber\\&\leq\frac{\mu}{2}\|\nabla\mathcal{Z}(t)\|_{\H}^2+\frac{1}{\mu}\|\Y(t)\|_{\wi\L^4}^4+\frac{\beta}{4}\|\mathcal{Z}(t)\|_{\wi\L^{r+1}}^{r+1}+\frac{1}{\beta\mu^{\frac{r+1}{r-1}}}\|\Y(t)\|_{\wi\L^{\frac{2(r+1)}{r-1}}}^{\frac{2(r+1)}{r-1}}.
		%\end{align}
		\begin{align}\label{cal1}
			|\langle\mathcal{B}(\mathcal{Z}+\Y),\Y\rangle|&=|\langle\mathcal{B}(\mathcal{Z}+\Y,\mathcal{Z}+\Y),\Y\rangle|\nonumber\\&\leq |\langle\mathcal{B}(\mathcal{Z}+\Y,\mathcal{Z}),\Y\rangle|+|\langle\mathcal{B}(\mathcal{Z}+\Y,\Y),\Y\rangle|\nonumber\\&\leq\|\nabla\mathcal{Z}\|_{\H}\|(\mathcal{Z}+\Y)\Y\|_{\H}\nonumber\\&\leq\frac{\mu}{4}\|\nabla\mathcal{Z}\|_{\H}^2+\frac{1}{\mu}\|(\mathcal{Z}+\Y)\Y\|_{\H}^2
			\nonumber\\&\leq\frac{\mu}{4}\|\nabla\mathcal{Z}\|_{\H}^2+\frac{1}{\mu}\|\mathcal{Z}+\Y\|_{\wi\L^{r+1}}^2\|\Y\|_{\wi\L^{\frac{2(r+1)}{r-1}}}^2\nonumber\\&\leq\frac{\mu}{4}\|\nabla\mathcal{Z}\|_{\H}^2+\frac{\beta}{4}\|\mathcal{Z}+\Y\|_{\wi\L^{r+1}}^{r+1}+\frac{1}{\beta\mu^{\frac{r+1}{r-1}}}\|\Y\|_{\wi\L^{\frac{2(r+1)}{r-1}}}^{\frac{2(r+1)}{r-1}}\nonumber\\&\leq\frac{\mu}{4}\|\nabla\mathcal{Z}\|_{\H}^2+\frac{\beta}{4}\|\mathcal{Z}+\Y\|_{\wi\L^{r+1}}^{r+1}+\frac{1}{\beta\mu^{\frac{r+1}{r-1}}}|\mathbb{T}^d|^{\frac{r-3}{r-1}}\|\Y\|_{\wi\L^{r+1}}^{\frac{2(r+1)}{r-1}},
		\end{align}
		where in the last inequality, we used $\frac{2(r+1)}{r-1}\leq r+1$ for $r\geq3$. Similarly, we calculate
		\begin{align}
			|\langle\mathcal{C}_1(\mathcal{Z}+\Y),\Y\rangle|&\leq\|\mathcal{Z}+\Y\|_{\wi\L^{r+1}}^r\|\Y\|_{\wi\L^{r+1}}\leq\frac{1}{8}\|\mathcal{Z}+\Y\|_{\wi\L^{r+1}}^{r+1}+\frac{1}{(r+1)\left(\frac{r+1}{8r}\right)^r}\|\Y\|_{\wi\L^{r+1}}^{r+1},\label{cal2}\\
			|\langle\mathcal{C}_2(\mathcal{Z}+\Y),\Y\rangle|&\leq\|\mathcal{Z}+\Y\|_{\wi\L^{q+1}}^q\|\Y\|_{\wi\L^{q+1}}\leq|\mathbb{T}^d|^{\frac{q(r-q)}{(r+1)(q+1)}}\|\mathcal{Z}+\Y\|_{\wi\L^{r+1}}^{q}\|\Y\|_{\wi\L^{q+1}} \nonumber\\&\leq\frac{\beta}{8\gamma}\|\mathcal{Z}+\Y\|_{\wi\L^{r+1}}^{r+1} +\frac{(r+1-q)|\mathbb{T}^d|^{\frac{q(r-q)}{(q+1)(r+1-q)}}}{(r+1)\left(\frac{\beta(r+1)}{8\gamma q}\right)^{\frac{q}{r+1-q}}} \|\Y\|_{\wi\L^{q+1}}^{\frac{r+1}{r+1-q}},\label{cal2.0}\\
			|\langle\f,\mathcal{Z}\rangle|&\leq\|\f\|_{\V'}\|\mathcal{Z}\|_{\V}\leq\left(\frac{1}{2\alpha}+\frac{1}{\mu}\right)\|\f\|_{\V'}^2+\frac{\alpha}{2}\|\Z\|_{\H}^2+\frac{\mu}{4}\|\nabla\Z\|_{\H}^2.\label{cal3}
		\end{align}
		%\begin{align}\label{cal3}
		%	|\langle\f,\mathcal{Z}\rangle|\leq\frac{1}{2\alpha}\|\f\|_{\H}^2+\frac{\alpha}{2}\|\mathcal{Z}\|_{\H}^2.
		%\end{align}
		Combining \eqref{cal1}-\eqref{cal3} in \eqref{inn1}, we obtain
			\begin{align}\label{cal4}
				&\frac{1}{2}\frac{\d}{\d t}\|\mathcal{Z}(t)\|_{\H}^2+\frac{\mu}{2}\|\nabla\mathcal{Z}(t)\|_{\H}^2+\frac{\alpha}{2}\|\mathcal{Z}(t)\|_{\H}^2+ \frac{\beta}{2}\|\mathcal{Z}(t)+\Y(t)\|_{\wi\L^{r+1}}^{r+1}\nonumber\\&\leq\left(\frac{1}{2\alpha}+\frac{1}{\mu}\right)\|\f\|_{\V'}^2+\frac{1}{\beta\mu^{\frac{r+1}{r-1}}}|\mathbb{T}^d|^{\frac{r-3}{r-1}}\|\Y(t)\|_{\wi\L^{r+1}}^{\frac{2(r+1)}{r-1}}+\frac{\beta}{(r+1)\left(\frac{r+1}{8r}\right)^r}\|\Y(t)\|_{\wi\L^{r+1}}^{r+1}\nonumber\\&\quad+\frac{\gamma(r+1-q)|\mathbb{T}^d|^{\frac{q(r-q)}{(q+1)(r+1-q)}}}{(r+1)\left(\frac{\beta(r+1)}{8\gamma q}\right)^{\frac{q}{r+1-q}}} \|\Y\|_{\wi\L^{q+1}}^{\frac{r+1}{r+1-q}},
			\end{align}
			for a.e. $t\in[0,T]$. On integrating \eqref{cal4} over $t\in[0,T]$ and using H\"older's inequality and \eqref{bddest}, we finally obtain
		\begin{align*}
			&\sup\limits_{t\in[0,T]}\|\mathcal{Z}(t)\|_{\H}^2+\mu\int_0^T\|\nabla\mathcal{Z}(t)\|_{\H}^2\d t +\alpha\int_0^T\|\mathcal{Z}(t)\|_{\H}^2 \d t+\beta\int_0^T \|\mathcal{Z}(t)+\Y(t)\|_{\wi\L^{r+1}}^{r+1}\d t \nonumber\\&\leq\|\x\|_{\H}^2+\left(\frac{1}{\alpha}+\frac{2}{\mu}\right)
			\int_0^T\|\f(t)\|_{\V'}^2\d t+ \frac{2}{\beta\mu^{\frac{r+1}{r-1}}}(T|\mathbb{T}^d|)^{\frac{r-3}{r-1}}\int_0^T\|\Y(t)\|_{\wi\L^{r+1}}^{r+1}\d t\nonumber\\&\quad+ \frac{2\beta}{(r+1)\left(\frac{r+1}{8r}\right)^r}\int_0^T \|\Y(t)\|_{\wi\L^{r+1}}^{r+1}\d t+\frac{2\gamma(r+1-q)|\mathbb{T}^d|^{\frac{q(r-q)}{(q+1)(r+1-q)}}}{(r+1)\left(\frac{\beta(r+1)}{8\gamma q}\right)^{\frac{q}{r+1-q}}}\int_0^T\|\Y(t)\|_{\wi\L^{q+1}}^{\frac{r+1}{r+1-q}}\d t\nonumber\\& \leq C\left(\mu,\beta,\gamma,r,T,\Q,|\mathbb{T}^d|, \|\f\|_{\mathrm{L}^2(0,T;\V')}\right), \ \P\text{ -a.s.}
		\end{align*}
		Since $\X=\mathcal{Z}+\Y$, therefore one can finally  deduce  \eqref{est}. 
			%	\textcolor{Violet}{For $d=2$ and $r\in[1,3]$, one can estimate $|\langle\mathcal{B}(\mathcal{Z}+\Y),\Y\rangle|$ as} 
		\end{proof}

\section{Irreducibility}\setcounter{equation}{0}\label{sec5}
In this section, as an application of the approximate controllability results established in Section \ref{apcon}, we prove that the transition semigroup $\{\mathrm{P}_t\}_{t\geq0}$ associated to the system  \eqref{scbf} is irreducible. For that, we assume that $r\in[3,\infty)$ for $d=2$ and $r\in[3,\infty)$ for $d=3$ ($2\beta\mu> 1$ for $d=r=3$). 

For any $t\geq 0$ and $\x\in\H$, let $\X(\cdot,\x),$  be the unique strong solution (in the probability sense) to the problem \eqref{scbf}. Let $\{\mathrm{P}_t\}_{t\geq0}$ be the transition semigroup (see \cite[Chapter 2]{gdp2}) associated to the system \eqref{scbf} and is given by 
\begin{align}\label{irrE}
	\mathrm{P}_t(\varphi)=\mathbb{E}\left[\varphi(\X(t,\cdot))\right],\ t\geq 0,\ \varphi\in\C_b(\H),
	\end{align}
where $\C_b(\H)$ is the space of all bounded continuous functions on $\H$.  A probability measure $\nu$ having support in $\V$ is said to be \emph{invariant} for \eqref{scbf}, with respect to semigroup $\{\mathrm{P}_t\}_{t\geq0}$, if 
\begin{align}\label{inv}
	\int_{\H}\mathrm{P}_t\varphi(\x)\nu(\d\x)=\int_{\H}\varphi(\x)\nu(\d\x),\ \varphi\in\C_b(\H).
\end{align}  
In other words, the measure $\nu$ is invariant if $\mathrm{P}_t^*\nu=\nu$, where $\{\mathrm{P}_t^*\}_{t\geq 0}$ is the dual semigroup of $\{\mathrm{P}_t\}_{t\geq 0}$.  We say that the transition semigroup $\{\mathrm{P}_t\}_{t\geq 0}$ on $\H$ is \emph{irreducible}  if the transition probabilities $\mathrm{P}_t(\x,\mathscr{U}):=\mathrm{P}_t\mathds{1}_{\mathscr{U}}(\x)=\mathbb{P}\left\{\X(t,\x)\in\mathscr{U}\right\}$ are strictly positive for all $t>0$, $\x\in\H$ and all open sets $\mathscr{U}\subseteq\H$, where  $\mathds{1}_{\mathscr{U}}(\cdot)$ is the characteristic function on $\mathscr{U}$ (cf. \cite{pjz}). Particularly, for all $r>0$, $\x_0,\x_1\in\H$ and $\mathscr{U}:=\B(\x_1,r)=\{\x\in\H:\|\x-\x_1\|_{\H}<r\}$, then the transition semigroup $\{\mathrm{P}_t\}_{t\geq 0}$ is \emph{irreducible} if 
\begin{align*}
	P_t\mathds{1}_{\B(\x_1,r)}(\x_0)>0,
\end{align*}
for all $t>0$. This means that, starting from any point, the process reaches all points immediately. From \eqref{irrE}, the definition of irreducibility is equivalent to
\begin{align*}
	\mathbb{P}\left\{\|\X(t;\x_0)-\x_1\|_{\H}\geq r\right\}<1,\ \text{ for all }\ t>0, \ r>0,\ \x_0,\x_1\in\H.
\end{align*}
Irreducibility is an important property for the transition semigroup $\{\mathrm{P}_t\}_{t\geq0}$ since it implies that the invariant measure $\nu$ associated with $\mathrm{P}_t$ is \emph{full}, that is, $\nu(\B(\x,r))>0$ for any ball $\B(\x,r)$ of center $\x\in\H$ and $r>0$. In fact from \eqref{inv}, it follows that
$$\nu(\B(\x,r))=\int_{\H}\mathrm{P}_t\mathds{1}_{\B(\x,r)}(\x)\nu(\d \x)>0.$$

The main result of this section is the following: 
\begin{theorem}\label{mainresult}
Let  $r\in[3,\infty)$ for $d=2$ and $r\in[3,\infty)$ for $d=3$ ($2\beta\mu>1$ for $d=r=3$). Then, under the assumption 
\begin{align}\label{54}
	\sqrt{\Q}\in\mathcal{L}(\mathbb{U},\mathbb{V}),
\end{align}the	transition semigroup $\{\mathrm{P}_t\}_{t\geq0}$ associated to the system  \eqref{scbf} is irreducible.
\end{theorem}
\begin{proof}
	The proof is divided into the following number of steps:
	\vskip 2mm
	\noindent
	\textbf{Step 1.} \emph{A density argument.}
Let $\x\in\H$. Since the embedding $\V\hookrightarrow\H$ is dense, so there exists a sequence $\{\x_n\}_{n\in\N}\in\V$ such that
	\begin{align}\label{density}
		\x_n\to\x \ \text{ in } \ \H \ \text{ as }\  n\to\infty.
\end{align}
Further, for any $\f\in\mathrm{L}^2(0,T;\V')$, we find a sequence $\{\f^n\}_{n\in\N}$ in $\mathrm{L}^2(0,T;\H)$ such that $\f^n\to\f$ in $\mathrm{L}^2(0,T;\V')$.
Let $\X^n(\cdot):=\X^n(\cdot,\x_n)$ denote the solution corresponding to $\x_n$ such that for each fixed $\omega\in\Omega$, the following holds: 
\begin{equation}\label{scbfden}
	\left\{
	\begin{aligned}
	&	\d\X^n(t)+\{\mu\A\X^n(t)+\mathcal{B}(\X^n(t))+\alpha\X^n(t)+\beta\mathcal{C}_1(\X^n(t))+\gamma\mathcal{C}_2(\X^n(t))\}\d t\\&\hspace{1.1cm}=\f^n(t)\d t+\sqrt{\mathrm{Q}}\d\W(t),\\ 
	&	\X^n(0)=\x_n,
	\end{aligned}
	\right.
\end{equation}
for a.e. $t\in[0,T]$. One can rewrite the stochastic CBFeD equations \eqref{scbf} in $\H$ as 
\begin{align}\label{0.1}
	\X^n(t)+\int_0^t \Gamma(\X^n(s))\d s=\x_n+ \int_0^t \f^n(s)\d s+\B\W(t),
\end{align}
for all $t\in[0,T]$, where $\B=\sqrt{\Q}$. Similarly, let $\y^n(\cdot,\x_n):=\y^n(\cdot)$ be the solution corresponding to the following controlled problem:
\begin{equation}\label{ccbf}
	\left\{
	\begin{aligned}
		\frac{\d \y^n(t)}{\d t}+\mu\A\y^n(t)+\alpha\y^n(t)+\mathcal{B}(\y^n(t))+\beta\mathcal{C}_1(\y^n(t)) +\gamma\mathcal{C}_2(\y^n(t))&=\f^n(t)+\mathrm{B}\u(t),  \\ 
		\y^n(0)&=\x_n,
	\end{aligned}
	\right. 
\end{equation}	
for a.e. $t\in[0,T]$ and $\u\in\mathrm{L}^2(0,T;\U)$. From \cite{KWH}, we know that the above controlled CBF equations has unique weak solution $\y^n\in\mathrm{C}([0,T];\H)\cap\mathrm{L}^2(0,T;\V)\cap\mathrm{L}^{r+1}(0,T;\wi\L^{r+1})$ satisfying the energy equality. Since $\x_n\in\V$, and $\f^n, \mathrm{B}\u\in\mathrm{L}^2(0,T;\H)$, the solution is strong and has the regularity $\y^n\in\mathrm{C}([0,T];\V)\cap\mathrm{L}^2(0,T;\D(\A))\cap\mathrm{L}^{r+1}(0,T;\wi\L^{3(r+1)})$ and the equation \eqref{ccbf} is satisfied in $\H$ for a.e. $t\in[0,T]$.  Taking the inner product with $\y^n(\cdot)$ to the first equation in \eqref{ccbf}, we find 
\begin{align*}
	&\|\y^n(t)\|_{\H}^2+2\mu\int_0^t\|\nabla\y^n(s)\|_{\H}^2\d s+2\alpha\int_0^t\|\y^n(s)\|_{\H}^2\d s+ 2\beta\int_0^t\|\y^n(s)\|_{\wi\L^{r+1}}^{r+1}\d s\nonumber\\&= \|\x_n\|_{\H}^2-2\gamma\int_0^t\|\y^n(s)\|_{\wi\L^{q+1}}^{q+1}\d s+2\int_0^t(\B\u(s),\y^n(s))\d s+2\int_0^t(\f^n(s),\y^n(s))\d s 
	%\nonumber\\&\leq \|\x_n\|_{\H}^2+2|\gamma||\mathbb{T}^d|^{\frac{r-q}{r+1}}\int_0^t \|\y^n(s)\|_{\wi\L^{r+1}}^{q+1}\d s+2\alpha\int_0^t\|\y^n(s)\|_{\H}^2 + \frac{1}{\alpha}\int_0^t\|\B\u(s)\|_{\H}^2\d s\nonumber\\&\quad+\frac{1}{\alpha}\int_0^t\|\f^n(s)\|_{\H}^2\d s 
	\nonumber\\&\leq\|\x_n\|_{\H}^2+ \varkappa(2|\gamma|)^{\frac{r+1}{r-q}}|\mathbb{T}^d|t+\beta\int_0^t \|\y^n(s)\|_{\wi\L^{r+1}}^{r+1}\d s+ 2\alpha\int_0^t\|\y^n(s)\|_{\H}^2 + \frac{1}{\alpha}\int_0^t\|\B\u(s)\|_{\H}^2\d s\nonumber\\&\quad+\frac{1}{\alpha}\int_0^t\|\f^n(s)\|_{\H}^2\d s,
\end{align*}
for all $t\in[0,T]$ and where $\varkappa=\left(\frac{\beta(r+1)}{q+1}\right)^{\frac{q+1}{r-q}}\left(\frac{r-q}{r+1}\right)$. Thus, it is immediate that 
\begin{align}\label{53}
	&\sup_{t\in[0,T]}\|\y^n(t)\|_{\H}^2+2\mu\int_0^T\|\nabla\y^n(s)\|_{\H}^2\d s+\beta\int_0^T\|\y^n(s)\|_{\wi\L^{r+1}}^{r+1}\d s\nonumber\\& \leq \|\x_n\|_{\H}^2+\varkappa(2|\gamma|)^{\frac{r+1}{r-q}}|\mathbb{T}^d|T+ \frac{1}{\alpha}\int_0^T\|\B\u(s)\|_{\H}^2\d s+\frac{1}{\alpha}\int_0^T\|\f^n(s)\|_{\H}^2\d s. 
\end{align}
Since $\x_n\to\x$ in $\H$, we know that there exists an $\eps>0$ such that $\|\x_n\|_{\H}\leq \varepsilon +\|\x\|_{\H}$ for all $n>N$ and $\|\x_{n}\|_{\H}\leq\max\limits_{1\leq k\leq N}\|\x_k\|_{\H}$. Similarly, one can show that $\|\f^n\|_{\mathrm{L}^2(0,T;\H)}$ is bounded independent of $n$. Therefore, the right hand side of \eqref{53} is independent of $n$ and we denote the constant by $C_T$. Therefore, by the Banach-Alaoglu theorem, we have the following limits: 
\begin{equation}\label{5p9}
	\left\{
\begin{aligned}
	&\y^n\xrightharpoonup{w^*}\y\ \text{ in }\ \mathrm{L}^{\infty}(0,T;\H),\\
	&\y^n\xrightharpoonup{w}\y\ \text{ in }\ \mathrm{L}^2(0,T;\V)\cap\mathrm{L}^{r+1}(0,T;\wi\L^{r+1}),\\
	&\frac{\d\y^n}{\d t}\xrightharpoonup{w}\frac{\d\y}{\d t}\ \text{ in }\ \mathrm{L}^{\frac{r+1}{r}}(0,T;\V'+\L^{\frac{r+1}{r}}),\\
	&\y^n\to\y\ \text{ in }\ \mathrm{L}^2(0,T;\H), 
\end{aligned}
\right. 
\end{equation}
and the final convergence is along a subsequence  due an application of the Aubin-Lions compactness lemma. Moreover, one can pass to the limit $n\to\infty$ in \eqref{ccbf} and obtain that  the limit  $\y(\cdot)$ is the unique weak solution of the following system in $\V'$: 
\begin{equation}\label{5.9}
	\left\{
	\begin{aligned}
		\frac{\d \y(t)}{\d t}+\mu\A\y(t)+\alpha\y(t)+\mathcal{B}(\y(t))+\beta\mathcal{C}_1(\y(t)) +\gamma\mathcal{C}_2(\y(t))&=\f(t)+\mathrm{B}\u(t),  \\ 
		\y(0)&=\x,
	\end{aligned}
	\right. 
\end{equation}	
for a.e. $t\in[0,T]$. 
\vskip 2mm
\noindent
\textbf{Step 2.} \emph{Strong convergence}.
For our purpose, the final strong convergence in \eqref{5p9} is not enough. We need to show that $\y^n\to\y\ \text{ in }\ \mathrm{C}([0,T];\H),$ along a subsequence.

 Let $\{\x_{n_k}\}_{k\in\mathbb{N}}$ and $\{\x_{n_j}\}_{j\in\mathbb{N}}$ be two subsequences of $\{\x_n\}_{n\in\mathbb{N}}$ such that \eqref{ccbf} holds for $\y^{n_k}(\cdot)$ and $\y^{n_j}(\cdot)$ with the initial data $\y^{n_k}(0)=\x^{n_k}$ and $\y^{n_j}(0)=\x^{n_j},$  and the forcing $\{\f^{n_k}\}_{k\in\mathbb{N}}$ and $\{\f^{n_j}\}_{j\in\mathbb{N}}$ in $\mathrm{L}^{2}(0,T;\H)$, respectively. Then, we have
\begin{align}\label{density1}
	\x^{n_k}_0, \x^{n_j}_0\to\x \ \text{ in } \ \H\ \text{ and }\ 
\f^{n_k},	\f^{n_j}\to \f \ \text{ in }\ \mathrm{L}^2(0,T;\V'),
\end{align}
and the completeness of the spaces $\H$ and $\mathrm{L}^2(0,T;\V')$ implies that the sequences  $\{\x_n\}_{n\in\mathbb{N}}$  and $\{\f^{n}\}_{n\in\mathbb{N}}$  are Cauchy in $\H$ and $\mathrm{L}^2(0,T;\V')$, respectively.  Moreover for a.e. $t\in[0,T]$, we also have
\begin{align}\label{512}
	&\frac{\d}{\d t}(\y^{n_k}(t)-\y^{n_j}(t))+ \mu(\A\y^{n_k}(t)-\A\y^{n_j}(t))+
	\beta(\mathcal{C}_1(\y^{n_k}(t))-\mathcal {C}_1(\y^{n_j}(t)))\nonumber\\& = \f^{n_k}(t)-\f^{n_j}(t)-(\mathcal{B}(\y^{n_k}(t)) -\mathcal{B}(\y^{n_j}(t)))-\gamma(\mathcal{C}_2(\y^{n_k}(t))-\mathcal{C}_2(\y^{n_j}(t))),
\end{align} 
in $\H$. Taking the inner product with $\y^{n_k}(\cdot)-\y^{n_j}(\cdot)$ and using the calculations performed  in \cite[Propositions 4.1, 4.3]{sKM}, we obtain for a.e. $t\in[0,T]$
\begin{align}\label{eng5}
	&\frac{1}{2}\frac{\d}{\d t}\|\y^{n_k}(t)-\y^{n_j}(t)\|_{\H}^2 +\frac{\mu}{2} \|\nabla(\y^{n_k}(t)-\y^{n_j}(t))\|_{\H}^2+
	\frac{\beta}{4}\||\y^{n_k}(t)|^{\frac{r-1}{2}}(\y^{n_k}(t)-\y^{n_j}(t))\|_{\H}^2\nonumber\\&\quad+\frac{\beta}{4}\||\y^{n_j}(t)|^{\frac{r-1}{2}}(\y^{n_k}(t)-\y^{n_j}(t))\|_{\H}^2\nonumber\\&
	\leq\eta\|\y^{n_k}(t)-\y^{n_j}(t)\|_{\H}^2+\left(\frac{1}{2\alpha}+\frac{1}{\mu}\right)\|\f^{n_k}(t)-\f^{n_j}(t)\|_{\V'},
\end{align}
where $\eta=\eta_1+\eta_2+\eta_3+\frac{1}{2}$. Using \eqref{C1} in \eqref{eng5},  we arrive at
\begin{align}\label{eng6}
	&\frac{1}{2}\frac{\d}{\d t}\|\y^{n_k}(t)-\y^{n_j}(t)\|_{\H}^2 + \frac{\mu}{2} \|\nabla(\y^{n_k}(t)-\y^{n_j}(t))\|_{\H}^2+ \frac{\beta}{2^{r}}\|\y^{n_k}(t)-\y^{n_j}(t)\|_{\wi\L^{r+1}}^{r+1}\nonumber\\&\leq
	2\eta\|\y^{n_k}(t)-\y^{n_j}(t)\|_{\H}^2+\left(\frac{1}{2\alpha}+\frac{1}{\mu}\right)\|\f^{n_k}(t)-\f^{n_j}(t)\|_{\V'},
\end{align}
for a.e. $t\in[0,T]$. Now integrating \eqref{eng6} over $(0,t)$,  we obtain
\begin{align*}
	& \|\y^{n_k}(t)-\y^{n_j}(t)\|_{\H}^2+\mu\int_0^t \|\y^{n_k}(s)-\y^{n_j}(s)\|_{\V}^2\d s+\frac{\beta}{2^{r-1}} \int_0^t \|\y^{n_k}(s)-\y^{n_j}(s)\|_{\wi\L^{r+1}}^{r+1} \d s\nonumber
	\\&\leq\|\y^{n_k}(0)-\y^{n_j}(0)\|_{\H}^2+2\eta\int_0^t \|\y^{n_k}(s)-\y^{n_j}(s)\|_{\H}^2\d s+\left(\frac{1}{\alpha}+\frac{2}{\mu}\right)\int_0^t\|\f^{n_k}(s)-\f^{n_j}(s)\|_{\V'}\d s, 
\end{align*}
for all $t\in[0,T]$. Applying Gronwall's inequality,  we deduce 
\begin{align}\label{ctsdep}
	& \|\y^{n_k}(t)-\y^{n_j}(t)\|_{\H}^2+\mu\int_0^t \|\y^{n_k}(s)-\y^{n_j}(s)\|_{\V}^2\d s+\frac{\beta}{2^{r-1}} \int_0^t \|\y^{n_k}(s)-\y^{n_j}(s)\|_{\wi\L^{r+1}}^{r+1} \d s\nonumber\\&\leq
	e^{2\eta T}\left\{\|\x^{n_k}-\x^{n_j}\|_{\H}^2+\left(\frac{1}{\alpha}+\frac{2}{\mu}\right)\int_0^T\|\f^{n_k}(t)-\f^{n_j}(t)\|_{\V'}\d t\right\},
\end{align}
for all $t\in[0,T]$ and for all $j,k\in\N.$  This shows that the sequence $\{\y^{n_k}(t)\}_{k\in\N}$ for all $t\in[0,T]$  is a uniformly Cauchy sequence in $\H$ and since $\H$ is complete, we obtain the following uniform convergence:
\begin{align}\label{stc1}
	\y^{n_k}(t)\to\y (t)\   \text{ in } \  \H \  \mbox{ for all }\ t\in[0,T]. 
\end{align}
Since $\y^{n_k}\in\C([0,T];\H)$, the uniform convergence implies $\y\in\C([0,T];\H)$.  By a similar reasoning as above, we obtain
\begin{align}\label{stc2}
	\y^{n_k}\to\y \  \ \text{in} \  \mathrm{L}^2(0,T;\V) \ \text{ and } \  \y^{n_k}\to\y \  \ \text{in} \  \mathrm{L}^{r+1}(0,T;\wi\L^{r+1}).
\end{align}
Note that from \eqref{stc1}, by continuity in time,  we have following result:
\begin{align}\label{cty}
 \y^{n_k}(T,\x^{n_k})\to\y(T,\x) \  \text{ in } \ \H.
\end{align}

 Let $\{\x_{n_k}\}_{k\in\mathbb{N}}$ and $\{\x_{n_j}\}_{j\in\mathbb{N}}$ be two subsequences of $\{\x_n\}_{n\in\mathbb{N}}$ such that \eqref{scbfden} holds for $\X^{n_k}(\cdot)$ and $\X^{n_j}(\cdot)$ with the initial data $\X^{n_k}(0)=\x^{n_k}$ and $\X^{n_j}(0)=\x^{n_j},$ in $\V$,   and the forcing $\{\f^{n_k}\}_{k\in\mathbb{N}}$ and $\{\f^{n_j}\}_{j\in\mathbb{N}}$ in $\mathrm{L}^{2}(0,T;\H)$, respectively. 
%Then, we have
%\begin{align}\label{density2}
%	\x^{n_k}_0, \x^{n_j}_0\to\x \ \text{ in } \ \H\ \text{ and }\ 
%	\f^{n_k},	\f^{n_j}\to \f \ \text{ in }\ \mathrm{L}^2(0,T;\V'),
%\end{align}
%and the completeness of the spaces $\H$ and $\mathrm{L}^2(0,T;\V')$ implies that the sequences  $\{\x_n\}_{n\in\mathbb{N}}$  and $\{\f^{n}\}_{n\in\mathbb{N}}$  are Cauchy in $\H$ and $\mathrm{L}^2(0,T;\V')$, respectively.  
Then \eqref{density1} holds true.  Moreover for a.e. $t\in[0,T]$, we also have 
\begin{align}\label{519}
	&\frac{\d}{\d t}(\X^{n_k}(t)-\X^{n_j}(t))+ \mu(\A\X^{n_k}(t)-\A\X^{n_j}(t))+
	\beta(\mathcal{C}_1(\X^{n_k}(t))-\mathcal {C}_1(\X^{n_j}(t)))\nonumber\\& = \f^{n_k}(t)-\f^{n_j}(t)-(\mathcal{B}(\X^{n_k}(t)) -\mathcal{B}(\X^{n_j}(t)))-\gamma(\mathcal{C}_2(\X^{n_k}(t))-\mathcal{C}_2(\X^{n_j}(t))),
\end{align} 
in $\H$ (a random PDE).  
Since the systems \eqref{512} and \eqref{519} are the same,  we obtain the following uniform convergence:
\begin{align}\label{stc3}
	\X^{n_k}(t)\to\X (t)\   \text{ in } \  \H \  \mbox{ for all }\ t\in[0,T], \ \mathbb{P}\text{-a.s.}
\end{align}
Since $\X^{n_k}\in\C([0,T];\H)$, $\mathbb{P}\text{-a.s.,}$ the uniform convergence implies $\X\in\C([0,T];\H)$, $\mathbb{P}\text{-a.s.}$.

\vskip 2mm
\noindent
\textbf{Step 3:} \emph{Some useful estimates.} We denote $n_k$ as $n$ in the rest of the discussion. 
Let us first write  the system \eqref{ccbf} in the integral form
\begin{align}\label{0.2}
	\y^n(t)+\int_0^t \Gamma(\y^n(s))\d s =\x_n+\int_0^t \f(s)\d s+\B\mathrm{V}(t),\ t\in[0,T],
\end{align}
in $\H$, where $\mathrm{V}(t)=\int_0^t \u(s)\d s$. On subtracting \eqref{0.2} from \eqref{0.1}, we get for all $t\in[0,T]$
\begin{align}\label{0.22}
	\X^n(t)-\y^n(t)+\Z^n(t)= \B\W(t)-\B\mathrm{V}(t),\ \text{ in }\ \H, 
\end{align}
where $\Z^n(t):=\int_0^t (\Gamma(\X^n(s))-\Gamma(\y^n(s)))\d s$. Then, $\dot{\Z}^n(t)=\Gamma(\X^n(t))-\Gamma(\y^n(t))$, for a.e. $t\in[0,T]$. Therefore, on taking the inner product with $\Gamma(\X^n(\cdot))-\Gamma(\y^n(\cdot))$ in \eqref{0.22}, we find
\begin{align}\label{0.3}
	&\int_0^t \langle\X^n(s)-\y^n(s),\Gamma(\X^n(s))-\Gamma(\y^n(s))\rangle\d s+\int_0^t (\Z^n(s),\dot{\Z}^n(s))d s \nonumber\\&
	=\int_0^t \langle\B(\W(s)-\mathrm{V}(s)),\Gamma(\X^n(s))-\Gamma(\y^n(s))\rangle\d s,
\end{align}
for all $t\in[0,T]$. Moreover, we have the following monotone property (\cite{sKM}):
%\begin{align*}
%	&\langle(\F+k_1\I)(\X-\y),\X-\y\rangle\nonumber\\&\geq\frac{\mu}{2} \|\nabla(\X-\y)\|_{\H}^2+(\alpha+k_1-\rho)\|\X-\y\|_{\H}^2
%	+\frac{\beta}{2^r}\|\X-\y\|_{\wi\L^{r+1}}^{r+1},
%\end{align*}
%for sufficiently large $k_1\geq\rho$. So, for $k_1=\rho$, we obtain
\begin{align}\label{0.4}
	&\langle(\Gamma+\rho\I)(\X^n-\y^n),\X^n-\y^n\rangle\geq\frac{\mu}{2} \|\nabla(\X^n-\y^n)\|_{\H}^2+\alpha\|\X^n-\y^n\|_{\H}^2
	+\frac{\beta}{2^r}\|\X^n-\y^n\|_{\wi\L^{r+1}}^{r+1},
\end{align}
where $\eta_7:=\frac{r-3}{2\mu(r-1)}\left(\frac{2}{\beta\mu(r-1)}\right)^{\frac{2}{r-3}}$. Using \eqref{0.4} in \eqref{0.3}, we obtain
\begin{align}\label{0.6}
	&\frac{\mu}{2}\int_0^t\|\nabla(\X^n(s)-\y^n(s))\|_{\H}^2\d s+\alpha\int_0^t \|\X^n(s)-\y^n(s)\|_{\H}^2\d s
	\nonumber\\&\quad+\frac{\beta}{2^r}\int_0^t \|\X^n(s)-\y^n(s)\|_{\wi\L^{r+1}}^{r+1}\d s+ \frac{1}{2}\|\Z^n(t)\|_{\H}^2\nonumber\\&\leq\eta_7\int_0^t \|\X^n(s)-\y^n(s)\|_{\H}^2\d s+\int_0^t \|\B(\W(s)-\mathrm{V}(s))\|_{\V}\|\Gamma(\X^n(s))-\Gamma(\y^n(s))\|_{\V'}\d s\nonumber\\&=:I_1+I_2,
\end{align}
for all $t\in[0,T]$. Let us calculate $I_1$ by using \eqref{0.22} and the Cauchy-Schwarz inequality as
\begin{align*}
&\eta_7\int_0^t \|\X^n(s)-\y^n(s)\|_{\H}^2\d s\nonumber\\&=
	\eta_7\int_0^t (\X^n(s)-\y^n(s),-\Z^n(s)+\B\W(s)-\B\mathrm{V}(s))\d s \nonumber\\&\leq\eta_7\int_0^t \|\X^n(s)-\y^n(s)\|_{\H}\|\Z^n(s)\|_{\H}\d s+\eta_7\int_0^t \|\X^n(s)-\y^n(s)\|_{\H}\|\B\W(s)-\B\mathrm{V}(s)\|_{\H}\d s \nonumber\\&\leq\frac{\alpha}{2}\int_0^t \|\X^n(s)-\y^n(s)\|_{\H}^2\d s+ \frac{\eta_7^2}{2\alpha}\int_0^t \|\Z^n(s)\|_{\H}^2\d s\no\\&\quad+\eta_7\int_0^t \|\X^n(s)-\y^n(s)\|_{\H}\|\B\W(s)-\B\mathrm{V}(s)\|_{\H}\d s,
\end{align*}
for all $t\in[0,T]$. Combining the above estimate together with \eqref{0.6}, we get
\begin{align*}
		&\frac{\mu}{4}\int_0^t\|\nabla(\X^n(s)-\y^n(s))\|_{\H}^2\d s+\frac{\alpha}{2}\int_0^t \|\X^n(s)-\y^n(s)\|_{\H}^2\d s\nonumber\\&\quad
		+\frac{\beta}{2^r}\int_0^t \|\X^n(s)-\y^n(s)\|_{\wi\L^{r+1}}^{r+1}\d s+\frac{1}{2}\|\Z^n(t)\|_{\H}^2 \nonumber\\&\leq\frac{\eta_7^2}{2\alpha}\int_0^t \|\Z^n(s)\|_{\H}^2\d s+\eta_7\int_0^t \|\X^n(s)-\y^n(s)\|_{\H} \|\B\W(s)-\B\mathrm{V}(s)\|_{\H} \d s\nonumber\\&\quad+\int_0^t
	\|\B(\W(s)-\mathrm{V}(s))\|_{\V}\|\Gamma(\X^n(s))-\Gamma(\y^n(s))\|_{\V'}\d s, 
\end{align*}
for all $t\in[0,T]$. On applying Gronwall's inequality, we obtain
\begin{align}\label{67}
&\frac{\mu}{2}\int_0^t\|\nabla(\X^n(s)-\y^n(s))\|_{\H}^2\d s+\frac{\alpha}{2}\int_0^t \|\X^n(s)-\y^n(s)\|_{\H}^2\d s\nonumber\\&\quad+
\frac{\beta}{2^r}\int_0^t \|\X^n(s)-\y^n(s)\|_{\wi\L^{r+1}}^{r+1}\d s+\frac{1}{2}\|\Z^n(t)\|_{\H}^2 \nonumber\\&\leq\bigg\{\int_0^T \|\X^n(s)-\y^n(s)\|_{\H}\|\B\W(s)-\B\mathrm{V}(s)\|_{\H}\d s \nonumber\\&\quad+\int_0^T
\|\B(\W(s)-\mathrm{V}(s))\|_{\V}\|\Gamma(\X^n(s))-\Gamma(\y^n(s))\|_{\V'}\d s\bigg\} e^{\frac{\eta_7^2}{\alpha}T}\nonumber\\&\leq \bigg\{\sup_{t\in[0,T]}\|\B\W(s)-\B\mathrm{V}(s)\|_{\H}\int_0^T\|\X^n(s)-\y^n(s)\|_{\H}\d s \nonumber\\&\quad+\sup_{t\in[0,T]}\|\B(\W(s)-\mathrm{V}(s))\|_{\V}\int_0^T
\|\Gamma(\X^n(s))-\Gamma(\y^n(s))\|_{\V'}\d s\bigg\}e^{\frac{\eta_7^2}{\alpha}T} \nonumber\\&\leq \|\B\W-\B\mathrm{V}\|_{\mathrm{C}([0,T];\V)} \bigg\{\int_0^T\|\X^n(s)-\y^n(s)\|_{\H}\d s +\int_0^T\|\Gamma(\X^n(s))-\Gamma(\y^n(s))\|_{\V'}\d s\bigg\} e^{\frac{\eta_7^2}{\alpha}T},
\end{align}
for all $t\in[0,T]$.   
%Since \textcolor{teal}{$\B:=\sqrt{\Q}\in\mathrm{L}^2(0,T;\H)\cap\mathrm{C}(0,T;\U)$,} where $\U\subset\H$. Thus from above, we have 
%\begin{align}\label{0.7}
%&\frac{\mu}{2}\int_0^T\|\nabla(\X(s)-\y(s))\|_{\H}^2\d s+\frac{\alpha}{2}\int_0^T \|\X(s)-\y(s)\|_{\H}^2\d s+
%\frac{\beta}{2^r}\int_0^T \|\X(s)-\y(s)\|_{\wi\L^{r+1}}^{r+1}\d s+\frac{1}{2}\|\Z(T)\|_{\H}^2 \nonumber\\&\leq\left(\frac{\alpha}{\rho^2}\right)e^{\frac{\rho^2}{\alpha}T}\rho C_1\int_0^T \|\X(s)-\y(s)\|_{\H}\d s+\left(\frac{\alpha}{\rho^2}\right)e^{\frac{\rho^2}{\alpha}T}\rho C_2
%\int_0^T\|\F(\X(s))-\F(\y(s))\|_{\V'+\wi\L^{\frac{r+1}{r}}}\nonumber\\&\leq\left(\frac{\alpha}{\rho}\right) e^{\frac{\rho^2}{\alpha}T}\max\{C_1,C_2\}\left(\int_0^T \|\X(s)-\y(s)\|_{\H}\d s+ \int_0^T\|\F(\X(s))-\F(\y(s))\|_{\V'+\wi\L^{\frac{r+1}{r}}}\right)\nonumber\\&\leq\left(\frac{\alpha}{\rho}\right) e^{\frac{\rho^2}{\alpha}T}\max\{C_1,C_2\}\left(T^{\frac{1}{2}}\left(\int_0^T \|\X(s)-\y(s)\|_{\H}^2\d s\right)^{\frac{1}{2}}+ \int_0^T\|\F(\X(s))-\F(\y(s))\|_{\V'+\wi\L^{\frac{r+1}{r}}}\right),
%\end{align}
%where \textcolor{teal}{$C_1:=\sup\limits_{t\in[0,T]}\|\B\W(t)-\B\mathrm{V}(t)\|_{\H}$} and  \textcolor{teal}{$C_2:=\sup\limits_{t\in[0,T]}\|\B\W(t)-\B\mathrm{V}(t)\|_{\V\cap\wi\L^{r+1}}.$} 
We now calculate the final integral in \eqref{67} as
\begin{align}\label{0.8}
&\int_0^T\|\Gamma(\X^n(t))-\Gamma(\y^n(t))\|_{\V'}\d t\nonumber\\&\leq\mu
\int_0^T\|\A\X^n(t)-\A\y^n(t)\|_{\V'}\d t+\int_0^T
\|\mathcal{B}(\X^n(t))-\mathcal{B}(\y^n(t))\|_{\V'}\d t \nonumber\\&\quad+\beta\int_0^T \|\mathcal{C}_1(\X^n(t))-\mathcal{C}_1(\y^n(t))\|_{\V'}\d t+\gamma\int_0^T \|\mathcal{C}_2(\X^n(t))-\mathcal{C}_2(\y^n(t))\|_{\V'}\d t.
\end{align}
From \cite[Subsection 2.3]{MT2}, we infer by using H\"older's inequality, the estimates \eqref{est} and \eqref{53}  that
\begin{align}\label{0.9}
&\int_0^T\|\mathcal{B}(\X^n(t))-\mathcal{B}(\y^n(t))\|_{\V'}\d t \nonumber\\&\leq\int_0^T \left(\|\X^n(t)\|_{\H}^{\frac{r-3}{r-1}}\|\X^n(t)\|_{\wi\L^{r+1}}^{\frac{2}{r-1}} +\|\y^n(t)\|_{\H}^{\frac{r-3}{r-1}}\|\y^n(t)\|_{\wi\L^{r+1}}^{\frac{2}{r-1}}\right)\|\X^n(t)-\y^n(t)\|_{\wi\L^{r+1}}\d t \nonumber\\&\leq\bigg[\left(\int_0^T\|\X^n(t)\|_{\H}^{\frac{r-3}{r-2}}\d t\right)^{\frac{r-2}{r-1}} \left(\int_0^T\|\X^n(t)\|_{\wi\L^{r+1}}^{r+1}\d t\right)^{\frac{2}{r^2-1}}+ \left(\int_0^T\|\y^n(t)\|_{\H}^{\frac{r-3}{r-2}}\d t\right)^{\frac{r-2}{r-1}} \nonumber\\&\quad\times\left(\int_0^T\|\y^n(t)\|_{\wi\L^{r+1}}^{r+1}\d t\right)^{\frac{2}{r^2-1}}\bigg] \times\left(\int_0^T\|\X^n(t)-\y^n(t)\|_{\wi\L^{r+1}}^{r+1}\d t\right)^{\frac{1}{r+1}}
\nonumber\\&\leq T^{\frac{r-2}{r-1}}\left[\sup\limits_{t\in[0,T]}\|\X^n(t)\|_{\H}^{\frac{r-3}{r-2}} \left(\int_0^T\|\X^n(t)\|_{\wi\L^{r+1}}^{r+1}\d t\right)^{\frac{2}{r^2-1}}+ \sup\limits_{t\in[0,T]}\|\y^n(t)\|_{\H}^{\frac{r-3}{r-2}}\left(\int_0^T\|\y^n(t)\|_{\wi\L^{r+1}}^{r+1}\d t\right)^{\frac{2}{r^2-1}}\right]\nonumber\\&\quad\times \left(\int_0^T\|\X^n(t)-\y^n(t)\|_{\wi\L^{r+1}}^{r+1}\d t\right)^{\frac{1}{r+1}}
\nonumber\\&\leq
C_T\left(\int_0^T\|\X^n(t)-\y^n(t)\|_{\wi\L^{r+1}}^{r+1}\d t\right)^{\frac{1}{r+1}},
\end{align}
for $r\geq 3$. By using  the estimates \eqref{est} and \eqref{53},  we estimate 
\begin{align}\label{0.10}
	&\int_0^T\|\mathcal{C}_1(\X^n(t))-\mathcal{C}_1(\y^n(t))\|_{\V'}\d t \nonumber\\&\leq C\int_0^T \left(\|\X^n(t)\|_{\wi\L^{r+1}}^{r-1} + \|\y^n(t)\|_{\wi\L^{r+1}}^{r-1}\right)\|\X^n(t)-\y^n(t)\|_{\wi\L^{r+1}}\d t \nonumber\\&\leq T^{\frac{1}{r+1}}
	\left[\left(\int_0^T\|\X^n(t)\|_{\wi\L^{r+1}}^{r+1}\d t\right)^{\frac{r-1}{r+1}} +\left(\int_0^T\|\y^n(t)\|_{\wi\L^{r+1}}^{r+1}\d t\right)^{\frac{r-1}{r+1}}\right] \left(\int_0^T\|\X^n(t)-\y^n(t)\|_{\wi\L^{r+1}}^{r+1}\d t\right)^{\frac{1}{r+1}}\nonumber\\&\leq C_T\left(\int_0^T\|\X^n(t)-\y^n(t)\|_{\wi\L^{r+1}}^{r+1}\d t\right)^{\frac{1}{r+1}},
\end{align}
and 
\begin{align*}
	&\int_0^T\|\mathcal{C}_2(\X^n(t))-\mathcal{C}_2(\y^n(t))\|_{\V'}\d t\leq C_T\left(\int_0^T\|\X^n(t)-\y^n(t)\|_{\wi\L^{q+1}}^{q+1}\right)^{\frac{1}{q+1}}.
\end{align*}
Therefore, from \eqref{0.8}, we conclude 
\begin{align}\label{0.11}
&\int_0^T\|\Gamma(\X^n(s))-\Gamma(\y^n(s))\|_{\V'}\d t\nonumber\\&\leq\mu T^{\frac{1}{2}} \left(\int_0^T \|\nabla(\X^n(t)-\y^n(t))\|_{\H}^2\d t\right)^{\frac{1}{2}}+C_T \left(\int_0^T\|\X^n(t)-\y^n(t)\|_{\wi\L^{r+1}}^{r+1}\right)^{\frac{1}{r+1}}\nonumber\\&\quad+C_T\left(\int_0^T\|\X^n(t)-\y^n(t)\|_{\wi\L^{q+1}}^{q+1}\right)^{\frac{1}{q+1}}.
\end{align}
On substituting \eqref{0.11} in \eqref{67}, we obtain
\begin{align}\label{0.12}
	&\frac{\mu}{2}\int_0^T\|\nabla(\X^n(s)-\y^n(s))\|_{\H}^2\d s+\frac{\alpha}{2}\int_0^T \|\X^n(s)-\y^n(s)\|_{\H}^2\d s\nonumber\\&\quad+\frac{\beta}{2^r}\int_0^T \|\X^n(s)-\y^n(s)\|_{\wi\L^{r+1}}^{r+1}\d s+ \frac{1}{2}\|\Z^n(T)\|_{\H}^2\nonumber\\&\leq \|\B\W-\B\mathrm{V}\|_{\mathrm{C}([0,T];\V)}\bigg\{T^{\frac{1}{2}}\left(\int_0^T \|\X^n(s)-\y^n(s)\|_{\H}^2\d s\right)^{\frac{1}{2}}+\mu T^{\frac{1}{2}} \left(\int_0^T \|\nabla(\X^n(t)-\y^n(t))\|_{\H}^2\d t\right)^{\frac{1}{2}}\nonumber\\&\quad+ C_T\int_0^T\|\X^n(t)-\y^n(t)\|_{\wi\L^{r+1}}^{r+1}\bigg\}^{\frac{1}{r+1}}e^{\frac{\eta_7^2}{\alpha} T}\nonumber\\&\leq C_T\|\B\W-\B\mathrm{V}\|_{\mathrm{C}([0,T];\V)} \bigg[\left(\int_0^T \|\X^n(s)-\y^n(s)\|_{\H}^2\d s\right)^{\frac{1}{2}} +\left(\int_0^T \|\nabla(\X^n(t)-\y^n(t))\|_{\H}^2\d t\right)^{\frac{1}{2}}\nonumber\\&\quad+ \left(\int_0^T\|\X^n(t)-\y^n(t)\|_{\wi\L^{r+1}}^{r+1}\right)^{\frac{1}{r+1}}+\left(\int_0^T\|\X^n(t)-\y^n(t)\|_{\wi\L^{q+1}}^{q+1}\right)^{\frac{1}{q+1}}\bigg].
\end{align}
By using  the estimates \eqref{est} and \eqref{53}, we find for $r\geq1$ 
\begin{align*}
\left(\int_0^T \|\X^n(s)-\y^n(s)\|_{\H}^2\d s\right)^{\frac{1}{2}}&=\left(\int_0^T \|\X^n(s)-\y^n(s)\|_{\H}^2\d s\right)^{\frac{1}{2}-\frac{1}{r+1}}\left(\int_0^T \|\X^n(s)-\y^n(s)\|_{\H}^2\d s\right)^{\frac{1}{r+1}}
\nonumber\\&\leq C_T\left(\int_0^T \|\X^n(s)-\y^n(s)\|_{\H}^2\d s\right)^{\frac{1}{r+1}},
\end{align*}
and for $r>q$ and an application of H\"older's inequality yields 
\begin{align*}
\left(\int_0^T\|\X^n(t)-\y^n(t)\|_{\wi\L^{q+1}}^{q+1}\d t\right)^{\frac{1}{q+1}}\leq \left(T|\mathbb{T}^d|\right)^{\frac{r-q}{(r+1)(q+1)}}\left(\int_0^T\|\X^n(t)-\y^n(t)\|_{\wi\L^{r+1}}^{r+1}\d t\right)^{\frac{1}{r+1}}.
\end{align*}
Similarly, one can show that 
\begin{align*}
	\left(\int_0^T \|\nabla(\X^n(s)-\y^n(s))\|_{\H}^2\d s\right)^{\frac{1}{2}}\leq C_T\left(\int_0^T \|\nabla(\X^n(s)-\y^n(s))\|_{\H}^2\d s\right)^{\frac{1}{r+1}}.
\end{align*}
Using the above estimates in \eqref{0.12}, we obtain 
\begin{align*}
&\min\left\{\frac{\mu}{2},\frac{\alpha}{2},\frac{\beta}{2^r}\right\}\bigg(\int_0^T\|\nabla(\X^n(s)-\y^n(s))\|_{\H}^2\d s+\int_0^T \|\X^n(s)-\y^n(s)\|_{\H}^2\d s\quad\\&\quad+\int_0^T \|\X^n(s)-\y^n(s)\|_{\wi\L^{r+1}}^{r+1}\d s\bigg) \nonumber\\&\leq C_T\|\B\W-\B\mathrm{V}\|_{\mathrm{C}([0,T];\V)}\bigg[\left(\int_0^T \|\X^n(s)-\y^n(s)\|_{\H}^2\d s\right)^{\frac{1}{r+1}} +\left(\int_0^T \|\nabla(\X^n(t)-\y^n(t))\|_{\H}^2\d t\right) ^{\frac{1}{r+1}}\nonumber\\&\quad+\left(\int_0^T\|\X^n(t)-\y^n(t)\|_{\wi\L^{r+1}}^{r+1}\right) ^{\frac{1}{r+1}}\bigg]\nonumber\\&\leq C_T\|\B\W-\B\mathrm{V}\|_{\mathrm{C}([0,T];\V)} 	\bigg[\int_0^T\|\nabla(\X^n(s)-\y^n(s))\|_{\H}^2\d s+\int_0^T \|\X^n(s)-\y^n(s)\|_{\H}^2\d s\nonumber\\&\quad+\int_0^T \|\X^n(s)-\y^n(s)\|_{\wi\L^{r+1}}^{r+1}\d s\bigg]^{\frac{1}{r+1}}.
\end{align*}
Therefore, it is immediate that 
\begin{align*}
&\left(\int_0^T\|\nabla(\X^n(s)-\y^n(s))\|_{\H}^2\d s+\int_0^T \|\X^n(s)-\y^n(s)\|_{\H}^2\d s+\int_0^T \|\X^n(s)-\y^n(s)\|_{\wi\L^{r+1}}^{r+1}\d s \right)^{\frac{r}{r+1}} \nonumber\\&\leq \frac{C_T}{\min\left\{\frac{\mu}{2},\frac{\alpha}{2},\frac{\beta}{2^r}\right\}}\|\B\W-\B\mathrm{V}\|_{\mathrm{C}([0,T];\V)}. 
\end{align*}
%or one can write 
%\begin{align*}
%	&\int_0^T\|\nabla(\X(s)-\y(s))\|_{\H}^2\d s+\int_0^T \|\X(s)-\y(s)\|_{\H}^2\d s+\int_0^T \|\X(s)-\y(s)\|_{\wi\L^{r+1}}^{r+1}\d s\nonumber\\&\leq \left(\frac{C_T}{\min\left\{\frac{\mu}{2},\frac{\alpha}{2},\frac{\beta}{2^r}\right\}}\right)^{\frac{r+1}{r}}\|\B\W-\B\mathrm{V}\|_{\mathrm{C}([0,T];\V)}^{\frac{r+1}{r}}.
%\end{align*}
Thus from \eqref{0.12}, we finally arrive at 
\begin{align*}
	&\int_0^T\|\nabla(\X^n(s)-\y^n(s))\|_{\H}^2\d s+\int_0^T \|\X^n(s)-\y^n(s)\|_{\H}^2\d s+ \int_0^T \|\X^n(s)-\y^n(s)\|_{\wi\L^{r+1}}^{r+1}\d s\nonumber\\&+\|\Z^n(T)\|_{\H}^2\leq C_T\|\B\W-\B\mathrm{V}\|_{\mathrm{C}([0,T];\V)}^{\frac{r+1}{r}}.
\end{align*}
Therefore, from \eqref{0.22} at $t=T$,  we have 
\begin{align}\label{0.13}
	\|\X^n(T)-\y^n(T)\|_{\H}&\leq\|\Z^n(T)\|_{\H}+\|\B\W(T)-\B\mathrm{V}(T)\|_{\H}
	\nonumber\\&\leq C_T\|\B\W-\B\mathrm{V}\|_{\mathrm{C}([0,T];\V)}^{\frac{r+1}{2r}} +\|\B\W(T)-\B\mathrm{V}(T)\|_{\H}.
\end{align}

\vskip 2mm
\noindent
\textbf{Step 4:} \emph{Irreducibility.}
We calculate
\begin{align}\label{0.14}
	&\|\X(T,\x)-\x_1\|_{\H}\nonumber\\&\leq\|\X(T,\x)-\X^n(T,\x_n)\|_{\H}+\|\X^n(T,\x_n)-\y^n(T,\x_n)\|_{\H}+\|\y^n(T,\x_n)-\y(T,\x)\|_{\H}\nonumber\\&\quad+\|\y(T,\x)-\x_1\|_{\H}.
\end{align}
 By the approximate controllability result (see Theorem \ref{prop31}), for each $\eps>0$, $T>0$ and all $\x,\x_1\in\H$, there exists $\v\in\mathrm{C}([0,T];\U)$ such that 
 \begin{align}\label{approx}
 	\|\y(T,\x)-\x_1\|_{\H}\leq C_T\eps.
 \end{align}
% Moreover, from \eqref{density}, for any $\eps>0$, there exists $n_0\in\N$ such that 
% \begin{align}\label{denlim}
% 	\|\x_n-\x\|_{\H}<\frac{\eps}{2} \  \text{ for all } n\geq n_0.
% \end{align}
Moreover, from the strong convergences \eqref{stc1} and \eqref{stc3}, we find for any $\eps>0$, there exist $n_1,n_2\in\N$ such that 
\begin{equation}\label{stc4}
	\left\{
\begin{aligned}
	\|\y^n(T,\x_n)-\y(T,\x)\|_{\H}&<\frac{\eps}{2}, \   \text{ for all } n\geq n_1 \ \text{ and }\\	\|\X(T,\x)-\X^n(T,\x_n)\|_{\H}&<\frac{\eps}{2}, \  \text{ for all } n\geq n_2.
\end{aligned}
\right.
\end{equation}
 Using \eqref{approx} and \eqref{stc4} in \eqref{0.14}, we deduce 
\begin{align*}
	\|\X(T,\x)-\x_1\|_{\H}\leq C_T\|\B\W-\B\mathrm{V}\|_{\mathrm{C}([0,T];\V)}^{\frac{r+1}{2r}}+\|\B\W(T)-\B\mathrm{V}(T)\|_{\H}+(C_T+1)\eps.
\end{align*}
 Thus, for any $k>0$, $0<\eps<\frac{k}{C_T+1}$ and $n>\max\{n_1,n_2\}$, we have 
\begin{align}\label{irred}
&	\P\left(\|\X(T,\x)-\x_1\|_{\H}\geq k\right)\nonumber\\&\leq \P\left(C_T\|\B\W-\B\mathrm{V}\|_{\mathrm{C}([0,T];\V)}^{\frac{r+1}{2r}}+\|\B\W(T)-\B\mathrm{V}(T)\|_{\H}+(C_T+1)\eps\geq k\right)\nonumber\\&= \P\left(C_T\|\B\W-\B\mathrm{V}\|_{\mathrm{C}([0,T];\V)}^{\frac{r+1}{2r}}+\|\B\W(T)-\B\mathrm{V}(T)\|_{\H}\geq k-(C_T+1)\eps\right)\nonumber\\&\leq \P\left(\|\B\W-\B\mathrm{V}\|_{\mathrm{C}([0,T];\V)}^{\frac{r+1}{2r}}+\|\B\W(T)-\B\mathrm{V}(T)\|_{\H}\geq\frac{k-(C_T+1)\eps}{\max\{C_T,1\}}\right).
\end{align}
From the representation $\sqrt{\Q}\W(t)=\sum\limits_{k=1}^{\infty}\sqrt{\mu_k}\boldsymbol{e}_k\beta_k(t)$, we conclude that $(\sqrt{\Q}\W(\cdot),\sqrt{\Q}\W(T))$ is a Gaussian random variable in $\Lambda:=\mathrm{C}([0,T];\V)\times\H$ (see \cite[Propositions 2.1.6 and 2.1.10]{wlm}). Now, if $(\sqrt{\Q}\W(\cdot),\sqrt{\Q}\W(T))$ is non-degenrate, then one can immediately conclude from \eqref{irred} that $\P\left(\|\X(T,\x_0)-\x_1\|_{\H}\geq k\right)<1$, and the  assertion of the theorem follows  if  we show that $(\sqrt{\Q}\W(\cdot),\sqrt{\Q}\W(T))$ is non-degenerate. 
\vskip 2mm
\noindent
\textbf{Step 5:} \emph{Non-degeneracy}.
Let us take an arbitrary element $\begin{pmatrix} \boldsymbol{\psi} \\ \x\end{pmatrix}$ in $\Lambda$ and let $\mathcal{Q}$ be the covariance operator of the Gaussian random variable $\begin{pmatrix} \sqrt{\Q}\W(t) \\ \sqrt{\Q}\W(T)\end{pmatrix}$. Since, $\E\begin{pmatrix} \sqrt{\Q}\W(t) \\ \sqrt{\Q}\W(T)\end{pmatrix}=0$ for all $t\in[0,T]$, then for any $\begin{pmatrix} \boldsymbol{\psi} \\ \x\end{pmatrix}\in\Lambda$ and $t\in[0,T]$, we calculate
\begin{align}\label{nondeg}
&	\left\langle\mathcal{Q}\begin{pmatrix} \boldsymbol{\psi} \\ \x\end{pmatrix},\begin{pmatrix} \boldsymbol{\psi} \\ \x\end{pmatrix}\right\rangle\nonumber\\&=\E\left[\left|\left\langle\begin{pmatrix} \sqrt{\Q}\W(t) \\ \sqrt{\Q}\W(T)\end{pmatrix},\begin{pmatrix} \boldsymbol{\psi} \\ \x\end{pmatrix}\right\rangle\right|^2\right] \nonumber\\&=\E\left\{\left[\left(\sqrt{\Q}\W(t),\boldsymbol{\psi}(t)\right)+\left(\sqrt{\Q}\W(T),\x\right)\right]^2\right\}\nonumber\\&=\E\left[\left(\sqrt{\Q}\W(t),\boldsymbol{\psi}(t)\right)^2+\left(\sqrt{\Q}\W(T),\x\right)^2+2\left(\sqrt{\Q}\W(t),\boldsymbol{\psi}(t)\right)\left(\sqrt{\Q}\W(T),\x\right)\right].
\end{align}
Since, $\E\left[\left(\sqrt{\Q}\W(t),\boldsymbol{\psi}(t)\right)^2\right]=t\left(\Q\boldsymbol{\psi}(t),\boldsymbol{\psi}(t)\right)$, $\E\left[\left(\sqrt{\Q}\W(T),\x\right)^2\right]=T\left(\Q\x,\x\right)$ and 
\begin{align*}
	\E\left[\left(\sqrt{\Q}\W(t),\boldsymbol{\psi}(t)\right)\left(\sqrt{\Q}\W(T),\x\right)\right]=(t\land T)\left(\Q\boldsymbol{\psi}(t),\x\right).
\end{align*}
Therefore, it implies from \eqref{nondeg} that 
\begin{align}\label{nondeg1}
 \left\langle\mathcal{Q}\begin{pmatrix} \boldsymbol{\psi} \\ \x\end{pmatrix},\begin{pmatrix} \boldsymbol{\psi} \\ \x\end{pmatrix}\right\rangle&=t\left(\Q\boldsymbol{\psi}(t),\boldsymbol{\psi}(t)\right)+T\left(\Q\x,\x\right)+2t\left(\Q\boldsymbol{\psi}(t),\x\right)\nonumber\\&=\left\langle\begin{pmatrix} t\Q\boldsymbol{\psi}(t)+2t\Q\x \\ T\Q\x\end{pmatrix},\begin{pmatrix} \boldsymbol{\psi} \\ \x\end{pmatrix}\right\rangle,
\end{align}
for all $\begin{pmatrix} \boldsymbol{\psi} \\ \x\end{pmatrix}\in\Lambda$. 
Let $\begin{pmatrix} \tilde{\boldsymbol{\psi}}\\ \tilde{\x}\end{pmatrix}\in\Lambda$ be such that $\mathcal{Q}\begin{pmatrix} \tilde{\boldsymbol{\psi}}\\ \tilde{\x}\end{pmatrix}=0.$ Then, from \eqref{nondeg1}, it implies for all $t\in[0,T]$
\begin{align*}
	t\Q\tilde{\boldsymbol{\psi}}(t)+2t\Q\tilde{\x}=0\ \text{ and } \ T\Q\tilde{\x}=0.
\end{align*}
Since $\Q$ is non-degenerate, the last equation in the above gives $\tilde{\x}=0$.  Consequently the first equation in above implies that $\tilde{\boldsymbol{\psi}}(t)=0$ for all $t\in [0,T]$. This, completes the proof that the random variable $(\sqrt{\Q}\W(\cdot),\sqrt{\Q}\W(T))$ is non-degenerate. 
\end{proof}
\subsection{Accessibility}
Let us define $\mathbb{X}:=\left\{\X(t,\x):t\geq 0, \x\in\H\right\}$, where $\X(\cdot,\x),$  is the unique strong solution (in the probability sense) to the problem \eqref{scbf} with the initial data $\x\in\H$. Then $\mathbb{X}$ is said to be \emph{accessible} to $\x_0\in\H$  if the resolvent $\mathrm{R}_{\lambda}$, $\lambda>0$ satisfies 
$$\mathrm{R}_{\lambda}(\x,\mathcal{U}):=\lambda\int_0^{\infty}e^{-\lambda t}\mathbb{P}\left\{\X(t,\x)\in\mathcal{U}\right\}\d t>0,$$ for all $\x\in\H$ and all neighborhoods $\mathcal{U}$ of $\x_0$, where $\lambda>0$ is arbitrary (\cite{JWHYJZ}). An immediate consequence of Theorem \ref{mainresult} is that 
\begin{corollary}\label{cor5.2}
	Under the assumptions of Theorem \ref{mainresult}, $\mathbb{X}$ is accessible to $\x_0\in\H$.  
\end{corollary}

\medskip\noindent
\textbf{Acknowledgments:} The first author would like to thank Ministry of Education, Government of India - MHRD for financial assistance. M. T. Mohan would  like to thank  the Science $\&$ Engineering Research Board (SERB),  Department of Science and Technology (DST), Govt. of India for a MATRICS grant (MTR/2021/000066). The first author would like to thank Dr. Kush Kinra and Mr. Ankit Kumar for helpful discussions and insightful comments.


\begin{thebibliography}{99}	
	
%	 \bibitem{FART}  F. Abergel and  R. Temam, On some control problems in fluid mechanics, \emph{Theoretical and Computational Fluid Dynamics}, {\bf 1} (1990), 303--325. 
    
    \bibitem{agr} A. A. Agrachev and A. V. Sarychev, Navier-Stokes equations: controllability by means of low modes forcing, \emph{J. Math. Fluid Mech.}, \textbf{7}(1) (2005), 108--152.
     
	\bibitem{SNA1} S. N. Antontsev and H. B. de Oliveira, Navier-Stokes equations with absorption under slip boundary conditions: existence, uniqueness and extinction in time, Kyoto Conference on the Navier-Stokes Equations and their Applications, Kyoto, \emph{Res. Inst. Math. Sci.}, (2007) 21--41.
	
	\bibitem{SNA}	S. N. Antontsev and H. B. de Oliveira, The Navier-Stokes problem modified by an absorption term, \emph{Appl. Anal.}, {\bf 89}(12) (2010), 1805--1825. 	
	
   % \bibitem{VB1} V. Barbu, \emph{Nonlinear Semigroups and Differential Equations in Banach Spaces}, Noordhoff International Publishing, 1976.
  
   \bibitem{VB2} V. Barbu, \emph{Analysis and Control of Nonlinear Infinite Dimensional Systems}, Academic Press, 1993.
 
 \bibitem{vbd} V. Barbu and G. D. Prato, Irreducibility of the transition semigroup associated with the two phase Stefan problem, Variational analysis and applications, pp. 147-159, \emph{Nonconvex Optim. Appl.,} \textbf{79}, Springer, New York, 2005.
 
 \bibitem{vbd1} V. Barbu and G. D. Prato, Irreducibility of the transition semigroup associated with the stochastic obstacle problem, \emph{Infin. Dimens. Anal. Quantum Probab. Relat. Top.}, \textbf{8}(3) (2005), 397--406.
 
 \bibitem{VB9} V. Barbu,  ,The irreducibility of transition semigroups and approximate controllability, \emph{Stochastic Partial Differential Equations and Applications-VII}, 21–26, Lect. Notes Pure Appl. Math., 245, Chapman$\&$ Hall/CRC, Boca Raton, FL, 2006.
 
  \bibitem{JBPCK} J. Babutzka and P. C. Kunstmann, $L^q$-Helmholtz decomposition on periodic domains and applications to Navier-Stokes equations, \emph{J. Math. Fluid Mech.}, \textbf{20}(3) (2018), 1093--1121.
  
  	%\bibitem{HBAM}	H. Bessaih and A. Millet,	On stochastic modified 3D Navier–Stokes equations with anisotropic viscosity, \emph{Journal of Mathematical Analysis and Applications}, 462 (2018), 915--956. 
  
   \bibitem{fab} F. A. Boussouira, R. Brockett, O. Glass, J. L. Rousseau and E. Zuazua, \emph{Control of Partial Differential Equations}, Lecture Notes in Mathematics, 2048, CIME Foundation Subseries, Springer, Heidelberg, 2012.
  
    \bibitem{ZBGD}  Z. Brze\'zniak and G. Dhariwal, Stochastic tamed Navier-Stokes equations on $\mathbb{R}^3$: the existence and the uniqueness of solutions and the existence of an invariant measure, \emph{J. Math. Fluid Mech.} \textbf{22}(2) (2020), Paper No. 23, 54 pp.
  
 % \bibitem{mmsc} M. Bues, M. Panilov, S. Crosnier and C. Oltean, Macroscale model and viscous-inertia effects for Navier-Stokes flow in a radial fracture with corrugated walls, \emph{J. Fluid Mech.}, \textbf{504} (2004), 41--60.
 
  	\bibitem{ZCQJ} Z. Cai and Q. Jiu, Weak and strong solutions for the incompressible Navier-Stokes equations with damping, \emph{J. Math. Anal. Appl.}, {\bf 343}(2) (2008), 799--809. 
 
%  \bibitem{PGC} P. G. Ciarlet, \emph{Linear and Nonlinear Functional Analysis with Applications}, SIAM Philadelphia, Philadelphia (2013).	
  
    \bibitem{jmc5} J. M. Coron, On the controllability of the 2-D incompressible Navier-Stokes equations with the Navier slip boundary conditions, \emph{ESAIM Contrôle Optim. Calc. Var.}, \textbf{1} (1995/96), 35--75.

   %\bibitem{jmc8} J. M. Coron, On the controllability of 2-D incompressible perfect fluids, \emph{J. Math. Pures Appl.9}, \textbf{75}(2) (1996), 155--188.
  
   \bibitem{jmc6} J. M. Coron and A. V. Fursikov, Global exact controllability of the 2D Navier-Stokes equations on a manifold without boundary, \emph{Russian J. Math. Phys.}, \textbf{4}(4) (1996), 429--448.
  
  %\bibitem{jmc9} J. M. Coron , F. Marbach and F. Sueur, On the controllability of the Navier-Stokes equation in spite of boundary layers,  Mathematical Analysis of Viscous Incompressible Fluid, Nov
  2016, Kyoto, Japan. pp.162--180. \url{https://hal.science/hal-01492722/document}
  
  \bibitem{bd} B. Davis, On the $L^p$ norms of stochastic integrals and other martingales, \emph{Duke Math. J.} \textbf{43}(4) (1976), 697--704.
  
    \bibitem{jLD} J. L. Doob, Asymptotic properties of Markoff transition prababilities, \emph{Trans. Amer. Math. Soc.}, \textbf{63}(1948), 393--421.
  
  \bibitem{AD}   A. Debussche, Ergodicity results for the stochastic Navier-Stokes equations: An introduction, \emph{Topics in Mathematical Fluid Mechanics}, Volume 2073 of the series Lecture Notes in Mathematics, Springer, 23--108, 2013.
  
  \bibitem{ZDYX}    Z. Dong and Y. Xie, Ergodicity of stochastic 2D Navier–Stokes equation with L\'evy noise, \emph{J. Differential Equations}, {\bf 251} (2011), 196--222.
 
  \bibitem{WEJC}   W. E and J. C. Mattingly, Ergodicity for the Navier-Stokes equation with degenerate random forcing: finite-dimensional approximation, Comm. Pure Appl. Math., 54 (2001), 1386--1402.
  
  \bibitem{jve} J. V. Egorov, Some problems in the theory of optimal control, \emph{Zh. Vychisl. Mat. Mat. Fiz.}, \textbf{3} (1963), 887--904.
  
  \bibitem{hf} H. Fattorini, Boundary control of temperature distributions in a parallepipedon, \emph{SIAM J. Control Optim.} \textbf{13} (1975), 1--13.
  
   %\bibitem{CF} C. Fabre, Uniqueness results for Stokes equations and their consequences in linear and nonlinear control problems, \emph{C. R. Acad. Sci. Paris Sér. I Math.} \textbf{322}(12) (1996), 1191--1196.
  
   \bibitem{FKS} R. Farwig, H. Kozono and H. Sohr,
  An $L^q$-approach to Stokes and Navier-Stokes equations in general domains,
  \emph{Acta Math.}, \textbf{195} (2005), 21--53.
    
  \bibitem{ffL} F. Flandoli, Irreducibility of the 3-D stochastic Navier-Stokes equation, \emph{J. Funct. Anal.}, \textbf{149}(1) (1997), 160--177. 
  
   \bibitem{FFBM}  F. Flandoli and B. Maslowski, Ergodicity of the 2-D Navier-Stokes equation under random perturbations, \emph{Comm. Math. Phys.}, {\bf 171} (1995), 119--141.
  
%  \bibitem{cft} C. Foias and R. Temam, The connection between the Navier-Stokes equations, dynamical systems, and turbulence theory, \emph{Directions in Partial Differential Equations}, 55--73, Publ. Math. Res. Center Univ. Wisconsin, \textbf{54}, Academic Press, Boston, MA, 1987.
  
  \bibitem{fmrt} C. Foias, O. Manley, R. Rosa and R. Temam, \emph{Navier-Stokes Equations and Turbulence: Encyclopedia of Mathematics and its Applications}, \textbf{83}, Cambridge University Press, Cambridge, 2001.
  
    \bibitem{DFHM} D. Fujiwara and H. Morimoto, An $L^r$-theorem of the Helmholtz decomposition of vector fields, \emph{J. Fac. Sci. Univ. Tokyo Sect. IA Math.}, \textbf{24}(3) (1977), 685--700.
  
  
  \bibitem{khoff} A. V. Fursikov,  Controllability Property for the Navier-Stokes Equations, \emph{Optimal control of partial differential equations}, 157--165, Internat. Ser. Numer. Math., \textbf{133}, Birkhäuser, Basel, 1999.
  
  \bibitem{avfoyu} A. V. Fursikov and O. Y. Imanuvilov, Exact controllability of the Navier-Stokes and Boussinesq equations, \emph{Russian Math. Surveys}, \textbf{54}(3) (1999), 565-618.
  54 (1999), no. 3, 565–618
  
  \bibitem{avf1} A.V. Fursikov, and O. Y. Imanuvilov, \emph{Controllability of evolution equations, Lecture Notes Series}, \textbf{34}, Seoul National University, Research Institute of Mathematics, Global Analysis Research Center, Seoul, 1996.
  

   %\bibitem{og1} O. Glass, Contr\^olabilit\'e exacte fronti\'ere de l'\'equation d'Euler des fluides parfaits incompressibles en dimension 3, \emph{C. R. Acad. Sci. Paris Sér. I Math.}, \textbf{325}(9) (1997), 987--992.
 
  % \bibitem{og2} O. Glass, Contr\^olabilit\'e de l'\'equation d'Euler tridimensionnelle pour les fluides parfaits incompressibles, \emph{S\'eminaire sur les \'Equations aux D\'eriv\'ees Partielles}, 1997--1998, Exp. No. XV, 11 pp., \'Ecole Polytech., Palaiseau, 1998.
  
 \bibitem{NGHJC}    N. Glatt-Holtz, J.C. Mattingly, G. Richards, On unique ergodicity in nonlinear stochastic partial differential equations, \emph{J. Stat. Phys.}, {\bf  166} (2017), 618--649.
 
 \bibitem{SGOY} S. Guerrero, O. Y.  Imanuvilov and J. P. Puel, A result concerning the global approximate controllability of the Navier-Stokes system in dimension 3, \emph{J. Math. Pures Appl. (9)}, \textbf{98}(6) (2012), 689--709. 
 
  \bibitem{sKM} S. Gautam, K. Kinra and M. T. Mohan, 2D and 3D convective Brinkman-Forchheimer Equations Perturbed by a Subdifferential and Applications to Control Problems, \emph{Math. Control Relat. Fields}, 2023.  \url{https://www.doi.org/10.3934/mcrf.2023034}.
 
 \bibitem{sKM1} S. Gautam, K. Kinra and M. T. Mohan, Feedback stabilization of Convective Brinkman-Forchheimer Extended Darcy equations, \emph{Submitted}. \url{https://arxiv.org/pdf/2311.13672.pdf}
 
 \bibitem{KWH} K. W. Hajduk and J. C. Robinson, Energy equality for the 3D critical convective Brinkman-Forchheimer equations, \emph{J. Differential Equations}, \textbf{263}(11) (2017), 7141--7161.
 
 \bibitem{Hopf} E. Hopf, \"Uber die Anfangswertaufgabe f\"ur die hydrodynamischen Grundgleichungen, \emph{Math. Nachr.}, \textbf{4} (1951), 213--231.
 
 %\bibitem{MDG}   M. D. Gunzburger, \emph{Perspectives in Flow Control and Optimization}, SIAM’s Advances in Design and  Control series, Philadelphia (2003).
 
\bibitem{MHJC}   M. Hairer, J.C. Mattingly, Ergodicity of the 2D Navier-Stokes equations with degenerate stochastic forcing, \emph{Annals of Mathematics}, {\bf 164} (2006), 993--1032.

 %\bibitem{tss1} T. Hav\^arneanu, C. Popa and S. S. Sritharan, Exact controllability for the three-dimensional Navier-Stokes equations with the Navier slip boundary conditions, \emph{Indiana Univ. Math. J.}, \textbf{54}(5) (2005), 1303--1350.
 
% \bibitem{tss2} T. Hav\^arneanu, C. Popa and S. S. Sritharan, Exact internal controllability for the two-dimensional Navier-Stokes equations with the Navier slip boundary conditions, \emph{Systems Control Lett.}, \textbf{55}(12) (2006), 1022--1028.
 
 \bibitem{oyu} O. Y. Imanuvilov, On exact controllability for the Navier-Stokes equations, \emph{ESAIM Control Optim. Calc. Var.}, \textbf{3} (1998), 97--131.
 
 \bibitem{oyu1} O. Y. Imanuvilov, Remarks on exact controllability for the Navier-Stokes equations, \emph{ESAIM Control Optim. Calc. Var.}, \textbf{6}(2001), 39--72.
 
    \bibitem{IA} A. Ichikawa, Some inequalities for martingales and stochastic convolutions, \emph{Stochastic Anal. Appl.}, \textbf{4}(3) (1986), 329--339.
  
  
  \bibitem{KT2} V. K. Kalantarov and S. Zelik, Smooth attractors for the Brinkman-Forchheimer equations with fast growing nonlinearities, \emph{Commun. Pure Appl. Anal.}, \textbf{11}(5) (2012), 2037--2054.
  
  \bibitem{akmtm} A. Kumar and M. T. Mohan, Large deviation principle for occupation measures of two dimensional stochastic convective Brinkman-Forchheimer equations, \emph{Stoch. Anal. Appl.}, \textbf{41}(2) (2023), 214--256. 
  
  \bibitem{pkmtm} P. Kumar and M. T. Mohan, Local exact controllability to the trajectories of the convective Brinkman-Forchheimer equations, \emph{Submitted.} \url{https://arxiv.org/pdf/2402.06335.pdf}
  
 \bibitem{gLLr} G. Lebeau and L. Robbiano, Contr\^ole exact de l'\'equation de la chaleur, \emph{Comm. Partial Differential Equations}, \textbf{20}(1-2) (1995), 335--356.
 
 %\bibitem{wL} W. Littman, Boundary control theory for hyperbolic and parabolic partial differential equations with constant coefficients, \emph{Ann. Scuola Norm. Sup. Pisa Cl. Sci.(4)} \textbf{5}(3) (1978), 567--580.
  
 \bibitem{JL} J. L. Lions, \emph{On the controllability of distributed systems}, \emph{Proc. Nat. Acad. Sci.}, U.S.A. \textbf{94}(10) (1997), 4828--4835.
 
 \bibitem{JL2} J. L. Lions, Exact controllability for distributed systems, Some trends and some problems, \emph{Applied and Industrial Mathematics} (Venice, 1989), 59--84, Math. Appl., 56, Kluwer Acad. Publ., Dordrecht, 1991. 
 
 %\bibitem{JL1} J. L. Lions, Remarks on the control of everything, \emph{European Congress on Computational Methods in Applied Sciences and Engineering}, Barcelona, 11–14 September, ECCOMAS (2000).
 
 \bibitem{JLEZ} J. L. Lions and E. Zuazua, Exact boundary controllability of Galerkin's approximations of Navier-Stokes equations, \emph{Ann. Scuola Norm. Sup. Pisa Cl. Sci. (4)}, \textbf{26}(4), 605--621. 
 
 \bibitem{wlm} W. Liu, and M. Rockner, \emph{Stochastic Partial Differential Equations: An Introduction}, Springer, Cham, 2015. 
     
  \bibitem{Leray} J. Leray, Sur le mouvement d'un liquide visqueux emplissant l'espace, \emph{Acta Math.}, \textbf{63}(1) (1934), 193--248.
  
 % \bibitem{ycxa} Y. Lin, C. Guo, X. G. Yang and A. Miranville, Dynamics of the three-dimensional Brinkman-Forchheimer-extended Darcy modelin the whole space, \emph{Discrete Contin. Dyn. Syst.}, 2023. \url{https://www.aimsciences.org/article/doi/10.3934/dcds.2023146}.
  
   %\bibitem{hLhG}  H. Liu and H. Gao, Stochastic 3D Navier–Stokes equations with nonlinear damping: martingale solution, strong solution and small time LDP, Chapter 2 in \emph{Interdisciplinary Mathematical Sciences Stochastic PDEs and Modelling of Multiscale Complex System}, 9--36, 2019.
  
    \bibitem{WL}  W. Liu, Well-posedness of stochastic partial differential equations with Lyapunov condition, \emph{J. Differential Equations}, {\bf 255}(3) (2013), 572--592. 
    
   % \bibitem{MRTZ1}	M. R\"ockner, T. Zhang and X. Zhang, Large deviations for stochastic tamed 3D Navier-Stokes equations, \emph{Applied Mathematics and Optimization}, {\bf 61}(2) (2010), 267--285. 
    
%   \bibitem{WLMR} W. Liu and M. R\"ockner,	Local and global well-posedness of SPDE with generalized coercivity conditions, \emph{J. Differential Equations}, {\bf 254}(2) (2013), 725--755. 
  
   \bibitem{MTT} P. A. Markowich, E. S. Titi and S. Trabelsi, Continuous data assimilation for the three dimensional Brinkman-Forchheimer-extended Darcy model, \emph{Nonlinearity}, \textbf{29}(4) (2016), 1292--1328. 
 
 % \bibitem{MT1} M. T. Mohan, On the convective Brinkman-Forchheimer equations, \emph{Submitted}.
 
  \bibitem{MT4} M. T. Mohan, Stochastic convective Brinkman-Forchheimer equations, \emph{Submitted}. \url{https://arxiv.org/pdf/2007.09376.pdf}.
 
 % \bibitem{MT5} M. T. Mohan, $\mathrm{L}^p$-solutions of deterministic and stochastic convective Brinkman-Forchheimer equations, \emph{ Anal. Math. Phys.}, \textbf{11}(4) (2021),  Paper No. 164, 33 pp.
 
  \bibitem{MT2}  M. T. Mohan, Well-posedness and asymptotic behavior of stochastic convective Brinkman-Forchheimer equations perturbed by pure jump noise, \emph{Stoch PDE: Anal. Comp.}, {\bf 10}(2) (2022), 614--690.
  
   \bibitem{MT10}  M. T. Mohan,  
  $\mathbb{L}^p$-solutions of deterministic and stochastic convective Brinkman-Forchheimer equations,
  \emph{Anal. Math. Phys.}, {\bf  11}(4) (2021),Paper No. 164, 33 pp.
  
  \bibitem{MT6} M. T. Mohan, Martingale solutions of two and three dimensional stochastic convective Brinkman-Forchheimer equations forced by L\'evy noise, \emph{Submitted}.   \url{ https://arxiv.org/pdf/2109.05510.pdf.}
  
  \bibitem{MT7} M. T. Mohan, Approximations of 2D and 3D Stochastic Convective Brinkman-Forchheimer Extended Darcy Equations, \emph{Submitted.} \url{https://arxiv.org/pdf/2305.14721.pdf} 
 
 \bibitem{MT8} M. T. Mohan, Asymptotic log-Harnack inequality for the stochastic convective Brinkman-Forchheimer equations with degenerate noise, \emph{Submitted}. \url{https://arxiv.org/pdf/2008.00955.pdf} 
 
 \bibitem{MT9} M. T. Mohan,  Backward uniqueness of 2D and 3D convective Brinkman-Forchheimer equations and its applications, \emph{Submitted}. \url{https://arxiv.org/pdf/2304.10589.pdf}
 
 
 \bibitem{MTSKSS}  M. T. Mohan, K. Sakthivel, S. S. Sritharan,  Ergodicity for the 3D stochastic Navier-Stokes equations perturbed by Lévy noise,  \emph{Math. Nachr.}, {\bf 292}(5) (2019), 1056--1088.
 
  \bibitem{pjz} S. Peszat and J. Zabczyk, Strong Feller property and irreducibility for diffusions on Hilbert spaces, \emph{Ann. Probab}, \textbf{23}(1) (1995), 157--172.
  
  \bibitem{pjz1} S. Peszat and J. Zabczyk, \emph{Stochastic Partial Differential Equations with L\'evy Noise: An Evolution Equation Approach}, Encyclopedia of Mathematics and its Applications, \textbf{113}, Cambridge University Press, Cambridge, 2007. 
  
   \bibitem{gdp2} G. D. Prato and J. Zabczyk, \emph{Ergodicity for Infinite-Dimensional Systems}, London Mathematical Society Lecture Note Series, Cambridge University Press, Cambridge, 1996.
  
  \bibitem{gdp} G. D. Prato and J. Zabczyk, \emph{Stochastic Equations in Infinite Dimensions},
  Second edition, Encyclopedia of mathematics and its applications, 152, Cambridge University Press, Cambridge, 2014. 
  
  \bibitem{GDPAD}   G. D. Prato and A. Debussche, Ergodicity for the 3D stochastic Navier-Stokes equations, \emph{J. Math. Pures Appl}, {\bf 82}(2) (2003),  877--947.
  
  % \bibitem{zDrZ} Z. Dong and R. Zhang, 3D tamed Navier-Stokes equations driven by multiplicative L\'evy noise: existence, uniqueness and large deviations, \emph{J. Math. Anal. Appl.,} \textbf{492}(1) (2020), 124404, 48 pp.
  
   %\bibitem{gdp1} G. D. Prato, Introduction to Stochastic Analysis and Malliavin Calculus, Third edition, Volume 13, Pisa, 2014.
 
%  \bibitem{drmy} D. Revuz and M. Yor, \emph{Continuous Martingales and Brownian Motion},
%  Third edition, Grundlehren der mathematischen Wissenschaften, Springer-Verlag, Berlin, 1999.
  
   \bibitem{JCR1} J. C. Robinson, \emph{Infinite-Dimensional Dynamical Systems: An Introduction to Dissipatives Parabolic PDEs and the Theory of Global Attractors}, Cambridge University Press, 2001.
  
 \bibitem{JCR} J. C. Robinson, J. L. Rodrigo and W. Sadowski, \emph{The Three-Dimensional Navier--Stokes equations, classical theory}, Cambridge University Press, Cambridge, UK, 2016.
 
\bibitem{MR1}   M. Romito, Ergodicity of the finite dimensional approximations of the 3D Navier-Stokes equations forced by a degenerate noise, \emph{J. Stat. Phys.}, {\bf  114}  (1-2) (2004), 155--177. 
 
 
\bibitem{MRLX}  M. Romito and L. Xu, Ergodicity of the 3D stochastic Navier–Stokes equations driven by mildly degenerate noise, \emph{Stochastic Process. Appl.}, {\bf 121}(4) (2011), 673--700.
 
 % \bibitem{LSM}  J. C. Robinson and  W. Sadowski, A local smoothness criterion for solutions of the 3D Navier–Stokes equations, \emph{Rend. Semin. Mat. Univ. Padova}, \textbf{131} (2014), 159--178.
  
  	\bibitem{MRXZ1}	M. R\"ockner and X. Zhang, Stochastic tamed 3D Navier-Stokes equation: existence, uniqueness and ergodicity, \emph{Probability Theory and Related Fields}, \textbf{145} (2009) 211--267.
  	
  	
  	
  
   \bibitem{dr1} D. Russell, Controllability and stabilizability theory for linear partial differential equations, Recent progress and open questions, \emph{SIAM Rev.} \textbf{20} (1978), 639--739.
  
  %\bibitem{dr2} D. Russell, A unified boundary controllability theory for hyperbolic and parabolic partial differential equations, \emph{Stud. Appl. Math.} \textbf{52} (1973), 189--212.
  
  %\bibitem{tS} T. Seidman, Two results on exact boundary controllability of parabolic equations, \emph{Appl. Math. Optim.}, \textbf{11}(2) (1984), 145--152.
  
  \bibitem{shir} A. Shirikyan, Approximate controllability of three-dimensional Navier-Stokes equations, \emph{Comm. Math. Phys.} \textbf{266}(1) (2006), 123--151.
  
  \bibitem{shir1} A. Shirikyan, Exact controllability in projections for three-dimensional Navier-Stokes equations, \emph{Ann. Inst. H. Poincaré C Anal. Non Linéaire}, \textbf{24}(4) (2007), 521--537.
  
   \bibitem{asch} A. Schenke, The tamed MHD equations, \emph{J. Evol. Equ.}, \textbf{21} (2021), 969--1018.
%   \bibitem{SSS}    S. S. Sritharan, \emph{Optimal control of viscous flow}, SIAM Frontiers in Applied Mathematics, Philadelphia,  Society for Industrial and Applied Mathematics, 1998.
  
 \bibitem{Te}  R. Temam, Navier-Stokes equations. Theory and numerical analysis. Revised edition, North-Holland Publishing Co., Amsterdam-New York, 1979.
 
%   \bibitem{RT1} R. Temam, Remarks on the control of turbulent flows, \emph{Flow Control} (Minneapolis, MN, 1992), 357--381, IMA Vol. Math. Appl., \textbf{68}, Springer, New York, 1995.
 
% \bibitem{RT2} R. Temam T. Bewley and P. Moin, Control of turbulent flows, \emph{Systems Modelling and Optimization}, Chapman $\&$ Hall/CRC Res. Notes Math., 396, Boca Raton, FL, 1999. 

%  \bibitem{ez1} E. Zuazua, Controllability of partial differential equations and its semi-discrete approximations, Current developments in partial differential equations, \emph{Discrete Contin. Dyn. Syst.} \textbf{8}(2) (2002), 469--513.
%  
%  \bibitem{ez2} E. Zuazua, Some problems and results on the controllability of partial differential equations, European Congress of Mathematics, Vol. II, \emph{Progr. Math.} \textbf{169}, 276--311, Birkhäuser, Basel, 1998.

 \bibitem{ivb} I. Vrabie, $\C_0$-\emph{Semigroups and Applications}, North-Holland Mathematics Studies, \textbf{191}, Amsterdam, 2003.
 
 \bibitem{JWHYJZ}   J. Wang, H. Yang, J. Zhai and T. Zhang, Accessibility of SPDEs driven by pure jump noise and its applications, \emph{Proc. Amer. Math. Soc.}, 2024. 
 
 
 \bibitem{ZZXW}	Z. Zhang, X. Wu and M. Lu, On the uniqueness of strong solution to the incompressible Navier-Stokes equations with damping, \emph{J. Math. Anal. Appl.}, {\bf 377}(1) (2011), 414--419. 
 
 \bibitem{SQ} S. Q. Zhang, Irreducibility and strong Feller property for non-linear SPDEs, \emph{Stochastics}, \textbf{91}(3) (2019), 352--382. 
 
 
 
 
\end{thebibliography}
\end{document}